\documentclass[12pt,a4paper]{amsart}
\usepackage{amsfonts,amsmath,amssymb}
\usepackage{hyperref}
\usepackage{latexsym,fullpage,amsfonts,amssymb,amsmath,amscd,graphics,epic}
\usepackage[all]{xy}
\usepackage{amssymb,amsthm,amsxtra}
\usepackage[usenames]{color}
\usepackage{amscd}
\usepackage{amsthm}
\usepackage{amsfonts}
\usepackage{amssymb}
\usepackage{mathrsfs}
\usepackage{url}
\usepackage{bbm}
\usepackage[noautoscale]{youngtab}%
\theoremstyle{plain}
\newtheorem*{theorem*}{Theorem}
\newtheorem*{remark*}{Remark}
\newtheorem*{example*}{Example}
\newtheorem{lemma}{Lemma}[subsection]
\newtheorem{proposition}[lemma]{Proposition}
\newtheorem{corollary}[lemma]{Corollary}
\newtheorem{theorem}[lemma]{Theorem}

\newtheorem*{conjecture*}{Conjecture}

\newtheorem{prop}[lemma]{Proposition}

\newtheorem{introtheorem}{Theorem}

\theoremstyle{definition}
\newtheorem{definition}[lemma]{Definition}
\newtheorem{cond}[lemma]{Condition}
\newtheorem{example}[lemma]{Example}

\theoremstyle{remark}
\newtheorem{remark}[lemma]{Remark}
\newtheorem{construction}[lemma]{Construction}
\newtheorem{conc}[lemma]{Conclusion}
\newtheorem{obsr}[lemma]{Observation}

\newtheorem{notation}[lemma]{Notation}

\newtheorem{notn}[lemma]{Notation}

\numberwithin{equation}{subsection}
\newcommand{\Gcd}{\mathbf{gcd}}
 \newcommand{\idealI}{\mathfrak{I}}
\newcommand{\Hom}{\operatorname{Hom}}

\newcommand{\triv}{{\mathbbm 1}}

\newcommand{\id}{\operatorname{Id}}

\renewcommand{\Im}{\operatorname{Im}}

\newcommand{\co}{\mathcal{O}}

\newcommand{\End}{\operatorname{End}}

\newcommand{\bC}{{\mathbb C}}
\newcommand{\bZ}{{\mathbb Z}}
\newcommand{\bQ}{{\mathbb Q}}
\newcommand{\bR}{{\mathbb R}}

\newcommand{\lam}{{\lambda}}

\newcommand{\fh}{{\mathfrak{h}}}

\newcommand{\cS}{{\it{S}}}

\newcommand{\T}{\tau}

\newcommand{\hhh}{\mathfrak{h}_0}

\newcommand{\sign}{\operatorname{sign}}
\newcommand{\abs}[1]{\left|{#1}\right|}

\newcommand{\brh}{{\rm{\boldsymbol h}}}

\newcommand{\InnaA}[1]{{{#1}}}
\newcommand{\InnaB}[1]{{#1}}
\newcommand{\InnaC}[1]{{#1}}

\newcommand{\InnaD}[1]{{#1}}
\newcommand{\InnaE}[1]{{#1}}
\newcommand{\InnaF}[1]{{#1}}
\begin{document}

\date{\today}
\title{On representations of rational Cherednik algebras of complex rank}
 \author{Inna Entova Aizenbud}
\address{Inna Entova Aizenbud,
Massachusetts Institute of Technology,
Department of Mathematics,
Cambridge, MA 02139 USA.}
\email{inna.entova@gmail.com}

\begin{abstract} 
We study a family of abelian categories $\InnaE{\underline{\co}}_{\text{  } c,\nu}$ depending on complex parameters $c, \nu$ which are interpolations of the category $\co$ for the rational Cherednik algebra $H_c(\nu)$ of type $A$, where $\nu$ is a positive integer. We define the notion of a Verma object in such a category (a natural analogue of the notion of Verma module).

We give some necessary conditions and some sufficient conditions for the existence of a non-trivial morphism between two such Verma objects. We also compute the character of the irreducible quotient of a Verma object for sufficiently generic values of parameters $c, \nu$, and prove that a Verma object of infinite length exists in $\InnaE{\underline{\co}}_{\text{  } c,\nu}$ only if $c \in \bQ_{<0}$. We also show that for every $c \in \bQ_{<0}$ there exists $\nu \in \bQ_{<0}$ such that there exists a Verma object of infinite length in $\InnaE{\underline{\co}}_{\text{  } c,\nu}$. 

The latter result is an example of a degeneration phenomenon which can occur in rational values of $\nu$, as was conjectured by P. Etingof.
\end{abstract}
\maketitle
\setcounter{tocdepth}{3}

\section{Introduction}
The study of representations in complex rank involves defining and studying families of abelian categories depending on a parameter $t$ which are polynomial interpolations of the categories of representations of objects such as finite groups, Lie groups, Lie algebras and more. This was done by P. Deligne in \cite{D} for finite dimensional representations of the general linear group $GL_n$, the orthogonal and symplectic groups $O_n, Sp_{2n}$ and the symmetric group $S_n$. Deligne defined Karoubian tensor categories $\InnaE{\underline{\mathrm{Rep}}}(GL_t), \InnaE{\underline{\mathrm{Rep}}}(OSp_t), \InnaE{\underline{\mathrm{Rep}}}(S_t)$, $t \in \bC$, which in points $n=t \in \bZ_+$ admit an essentially surjective functor onto the standard categories $\InnaE{\mathrm{Rep}}(GL_n), \InnaE{\mathrm{Rep}}(OSp_n), \InnaE{\mathrm{Rep}}(S_n)$. The category $\InnaE{\underline{\mathrm{Rep}}}(S_{t})$ was subsequently studied by himself and others (F. Knop in \cite{K}, V. Ostrik, J. Comes in \cite{CO}). 

The resulting categories can be shown to have superexponential growth, and are not equivalent to the category of finite-dimensional representations of any affine algebraic group or supergroup. One can look for rational non-integer values of $t$ where degeneration occurs, which would allow us to understand better how the structure of the classical category of representations is similar for all ranks $n$, and why degenerate phenomena occur at specific values of $n$.

P. Etingof, in \cite{E1}, suggested using these results to define and study the categories of representations of rational Cherednik algebras, as well as Lie superalgebras, affine Lie algebras and other ``non-compact'' representation theory objects. A possible approach, as stated in \cite{E1}, would be using the tensor categories $\InnaE{\underline{\mathrm{Rep}}}(GL_t), \InnaE{\underline{\mathrm{Rep}}}(OSp_t), \InnaE{\underline{\mathrm{Rep}}}(S_t)$ and defining the required category for a ``non-compact'' object in complex rank as a category of tuples which consist of an \InnaE{\rm{ind}}-object of the corresponding tensor category along with morphisms satisfying some relations.  

In this paper, we study the categories which are interpolations of the classical $\co$-category of lowest weight modules for the rational Cherednik algebra of type $A$, denoted by $H_c(n)$. These categories are BGG-type categories, where one can define Verma objects and study morphisms between these, as well as their irreducible quotients. 

The categories $\InnaE{\underline{\mathrm{Rep}}}(H_c(\nu))$ (the parameter $\nu$ replacing the parameter $t$ used above) have been defined in \cite{E1} as categories whose objects are \InnaE{\rm{ind}}-objects of Deligne's category $\InnaE{\underline{\mathrm{Rep}}}(S_{\nu})$ with additional structure; namely, two morphisms in $\InnaE{\underline{\mathrm{Rep}}}(S_{\nu})$, denoted by $x,y$, satisfying some relations. These morphisms represent the actions of the elements $x_1, .., x_n$ and $y_1, .., y_n$ of $H_c(n)$ respectively. We proceed to define the full subcategory $\InnaE{\underline{\co}}_{\text{  } c,\nu}$ of $\InnaE{\underline{\mathrm{Rep}}}(H_c(\nu))$, which is an interpolation of the category $\co$ of the classical Cherednik algebra $H_c(n)$.

In this category $\InnaE{\underline{\co}}_{\text{  } c,\nu}$ we have objects we will call ``Verma objects'', as they are analogues of the Verma modules in the classical category $\co$. They are parameterized by arbitrary partitions $\lambda$, and denoted by $M_{c, \nu}(\lambda)$. As in the $H_c(n)$ case, they have a natural grading as $\InnaE{\underline{\mathrm{Rep}}}(S_{\nu})$ \InnaE{\rm{ind}}-objects, with the lowest grade consisting of an indecomposable $\InnaE{\underline{\mathrm{Rep}}}(S_{\nu})$-object given by the partition $\lambda$. 

The Verma objects have about the same properties as the classical Verma modules, and one can ask for which values of parameters $c, \nu$ they are reducible, and for which values of parameters $c, \nu$ there are non-trivial morphisms between different Verma objects. It would also be interesting to compute the character of the unique irreducible quotient of $M(\lambda)$ (the character being its decomposition as a $\InnaE{\underline{\mathrm{Rep}}}(S_{\nu})$ \InnaE{\rm{ind}}-object into a sum of irreducible $\InnaE{\underline{\mathrm{Rep}}}(S_{\nu})$-objects).
 

\subsection{Summary of the results}

Recall that the category $\InnaE{\underline{\mathrm{Rep}}}(S_{\nu})$ is a \InnaB{Karoubian rigid symmetric monoidal category which is} abelian semisimple whenever $\nu \notin \bZ_+$; for any $\nu$, its indecomposable objects are parameterized by Young diagrams of arbitrary size. One can treat an indecomposable object $X_{\T}$ of $\InnaE{\underline{\mathrm{Rep}}}(S_{\nu})$, with $\T$ being a Young diagram, as corresponding to the \InnaD{``virtual'' Young diagram of size $\nu$} obtained from $\T$ by adding a very long top row (``of size $\nu-\abs{\T}$''). 

Throughout the paper we will assume that $\nu \notin \bZ_+$ unless it is explicitly stated otherwise.

\mbox{}

The indecomposable object in $\InnaE{\underline{\mathrm{Rep}}}(S_{\nu})$ corresponding to a Young diagram with one box is denoted by $\hhh$, and has dimension $\nu -1$ (it corresponds to the reflection representation of $S_n$ when $\nu=n$). One can define an algebra object $S\InnaD{\fh^*_0}$ in $\InnaE{\mathrm{ind}}-\InnaE{\underline{\mathrm{Rep}}}(S_{\nu})$ (which would correspond to the ring $\bC[x_1, ..., x_n]/\InnaE{\langle \sum_i x_i \rangle}$ when $\nu=n$).

We define the Verma objects $M_{c, \nu}(\lambda)$ of $\InnaE{\underline{\co}}_{\text{  } c,\nu}$ in Subsection \ref{ssec:Verma_obj} \InnaB{(these are sometimes denoted $M(\lambda)$ for short)}. The underlying \InnaE{\rm{ind}}-object of $\InnaE{\underline{\mathrm{Rep}}}(S_{\nu})$ is just $S\InnaD{\fh^*_0} \otimes \InnaB{X_{\lambda}}$.

\mbox{}

\InnaA{For every non-negative integer $n$, we construct a functor $\mathcal{F}_{c,n}$ from $\InnaE{\underline{\mathrm{Rep}}}(H_c(\nu=n))$ (the category $\InnaE{\underline{\mathrm{Rep}}}(H_c(\nu))$ for integer value $n \in \bZ_+$ of $\nu$) to $\InnaE{\mathrm{Rep}}(H_c(n))$ (the category of representations of $H_c(n)$). Considering the full subcategory of $\InnaE{\underline{\co}}_{\text{  } c,\nu=n}$ whose objects are $M_{c, \nu=n}(\lambda)$ such that $\abs{\lambda} \leq n/2$, we see that the restriction of $\mathcal{F}_{c,n}$ to this subcategory is fully faithful; in fact, $\mathcal{F}_{c,n}$ takes Verma object $M_{c, \nu=n}(\lambda)$ to Verma module in $\co(H_c(n))$ whose lowest weight corresponds to the Young diagram obtained from $\lambda$ by adding a top row of size $n - \abs{\lambda}$.}

\InnaD{We would like to understand the structure of the Verma objects $M_{c, \nu}(\lambda)$ of the category $\InnaE{\underline{\co}}_{\text{  } c,\nu}$. We start by introducing the following notation:}
\begin{notation}\label{b_tau_mu_notn}
\mbox{}
 \begin{itemize}
  \item Denote the set of points $(c, \nu)$ on the complex plane where $M_{c, \nu}(\T)$ is reducible by $\mathtt{B}_{\T}$. The image of $\mathtt{B}_{\T} \setminus \{c=0\}$ under the map $c \mapsto 1/c$ will be denoted by $\mathcal{B}_{\T}$.
  \item Denote the set of points $(c, \nu)$ on the complex plane where there is a non-zero morphism $M_{c, \nu}(\mu) \longrightarrow M_{c, \nu}(\T)$ by $\mathtt{B}_{\mu, \T}$. The image of $\mathtt{B}_{\mu, \T} \setminus \{c=0\}$ under the map $c \mapsto 1/c$ will be denoted by $\mathcal{B}_{\mu, \T}$.
 \end{itemize}

\end{notation}

\InnaD{The first step in studying the reducibility of the Verma objects $M_{c, \nu}(\lambda)$ would be describing the sets $\mathtt{B}_{\T}$ (resp. $\mathcal{B}_{\T}$).}

As in the classical case, we have: $\mathtt{B}_{\T} = \bigcup_{\mu} \mathtt{B}_{\mu, \T}$, the parameter $\mu$ running over all Young diagrams of arbitrary size. \InnaD{This allows us to concentrate on studying the sets $ \mathtt{B}_{\mu, \T}$ (resp. $\mathcal{B}_{\mu, \T}$).

We will show that if $\mu \neq \T$, the set $\mathtt{B}_{\mu, \T}$ lies inside a countable (disjoint) union of curves $\bigcup_{m \in \bZ_{>0}} \mathtt{L}_{\T, \mu, m}$ for which we give explicit equations (c.f. Section \ref{sec:Morph_Verma_necessary_cond}). These curves become straight lines, denoted by $\mathcal{L}_{\T, \mu, m}$, when we switch to parameters $(1/c, \nu)$ instead of $(c, \nu)$. 

For a fixed positive integer $m$, the following condition on a point $(c, \nu)$ implies $(c, \nu) \in \mathtt{L}_{\T, \mu, m}$: 
\begin{cond}
 There exists a non-zero map $M_{c,\nu}(\mu) \longrightarrow M_{c, \nu}(\T)$ with the lowest grade of $M_{c,\nu}(\mu)$ mapping to grade $m$ of $M_{c,\nu}(\T)$.
\end{cond}

}

It is easy to see that the intersection of the set $\mathcal{B}_{\mu, \T}$ with each of these lines $\mathcal{L}_{\T, \mu, m}$ is either the entire line or a finite number of points. \InnaD{As a step towards decribing the set} $\mathcal{B}_{\mu, \T}$, we give a full description of the straight lines $\mathcal{L}_{\T, \mu, m}$ lying inside $\mathcal{B}_{\mu, \T}$ whenever $\abs{\T} \neq \abs{\mu}$ in Theorem \ref{introthrm:thrm1}.

\InnaD{
\mbox{}

Recall that in the theory of classical rational Cherednik algebras, the existence of a non-zero map between Verma modules $M_{c, n}(\T)$, $M_{c, n}(\mu)$ depends on whether $\T, \mu$ have the same $e$-core, where $e$ is an integer depending on $n, c$.

Therefore, in order to understand when there \InnaE{is} a non-zero map between two given Verma objects in $\InnaE{\underline{\co}}_{\text{  } c,\nu}$, we would like to explain what does it mean for two virtual Young diagrams of size $\nu$ to have the same core.

Let $\T$ be a Young diagram. As we explained before, we have the corresponding virtual Young diagram $\tilde{\T}(\nu)$ obtained by adding a very long top row of size $\nu -\abs{\T}$ to $\T$.

Define $C_{\T}$ to be the set of integers $\{\abs{\T} -1 +j- \T\check{}_j \mid j \in \bZ_{>0}\}$ (here $\T\check{}$ is the transpose of $\T$, and so $\T\check{}_j$ is the length of the $j$-th column of $\T$). 

The set $C_{\T}$ is the set of all integers $s$ for which the virtual Young diagram $\tilde{\T}(\nu)$ has a hook of length $\nu-s$ (this is again a ``virtual'' hook, with a ``very long arm'', and a leg of length $\T\check{}_{j}$). 

When this virtual hook is removed, we obtain an ``normal'' Young diagram of size $s$, denoted by $\mathbf{core}_{(\nu-s)}(\T)$. Inserting into $\mathbf{core}_{(\nu-s)}(\T)$ another virtual hook of size $\nu-s$ with leg $\T\check{}_{j}+l$ (for any integer $l \geq -\T\check{}_{j}$), we obtain a new virtual Young diagram of size $\nu$. 

Removing the very long top row of the new virtual Young diagram, we obtain a (``normal'') Young diagram which we will call $\Gamma(\T, s, l)$. 

\mbox{}

The conditions under which there exists a non-zero map between Verma modules for $H_c(n)$ allow us to expect that there exists a non-zero morphism $M_{c, \nu}(\mu) \longrightarrow M_{c, \nu}(\T)$ only if $\mu = \Gamma(\T, s, l)$ for some $l$. In fact, we prove the following:}
%
%
%

 \begin{introtheorem}\label{introthrm:thrm1}
 For two \InnaA{distinct} Young diagrams $\mu, \T$ and an integer $m>0$, the following are equivalent:
\begin{enumerate}
 \item $\abs{\mu} \neq \abs{\T}$ and $\mathcal{L}_{\T, \mu, m} \subset \mathcal{B}_{\mu, \T}$,
\item $\mu = \Gamma(\T, s, \sign({\abs{\mu} -\abs{\T}}))$ for some $s \in C_{\T}$ (in particular, $\abs{\mu} \neq \abs{\T}$), \InnaA{ and $(\abs{\mu} -\abs{\T}) \mid m $}.
\end{enumerate}

\end{introtheorem}

One can easily see that if \InnaD{$\mu, \lambda, m$ satisfy the conditions of the above Theorem, then for a generic point $(\frac{1}{c}, \nu)$ of $\mathcal{L}_{\T, \mu, m}$, the Verma object $M_{c, \nu}(\T)$ has length $2$. }

\InnaD{In such a case, we give a formula for the character of the simple object $L_{c, \nu}(\T)$; moreover, we show that
there exists a long exact sequence}
\begin{align*}
 &... \rightarrow M(\Gamma(\T, s, \pm l)) \rightarrow M(\Gamma(\T, s, \pm (l-1))) \rightarrow  ...
 \rightarrow M(\Gamma(\T, s, \pm 1) = \mu) \rightarrow M(\T) \rightarrow L(\T) \rightarrow 0
\end{align*}

The sign is $\sign({\abs{\mu} -\abs{\T}})$, and this is a finite sequence if $\abs{\mu}<\abs{\T}$.

The character formula for $L_{c, \nu}(\T)$ is then obtained from Euler's formula applied to the above long exact sequence. 

\InnaD{In particular, we compute explicitly the character of the simple object $L_{c, \nu}(\emptyset)$, whenever $(1/c, \nu)$ is a generic point on a line $\mathcal{L}_{\emptyset, \mu, k} = \{ (1/c, \nu) \mid c\nu=k\}$ (here $k \in \bZ_{>0}$ is fixed).}

We conclude the paper by proving the following theorem:

\begin{introtheorem}\label{introthrm:length_of_Verma_obj}
If there exists a Verma object of infinite length in $\InnaE{\underline{\co}}_{\text{  } c,\nu}$, then $c \in \bQ_{<0}$. In fact, one can show that for any $c \in \bQ_{<0}$, there exists $\nu \in \bQ_{<0}$ such that $M_{c, \nu}(\emptyset)$ has infinite length.
\end{introtheorem}

Theorem \ref{introthrm:length_of_Verma_obj} is an example of a degenerate phenomenon which can occur for rational non-integer values of $\nu$. The classical representation theory of Cherednik algebras says that for a non-negative integer $n$, all modules in $\co(H_c(n))$ have finite length (cf. \cite[Corollary 3.26]{EM}), and we see from Theorem \ref{introthrm:length_of_Verma_obj} that the same is true for Verma objects in $\InnaE{\underline{\co}}_{\text{  } c,\nu}$ for $c \not\in \bQ$. 

It would be interesting to know about other phenomena which can occur in $\InnaE{\underline{\co}}_{\text{  } c,\nu}$ only for $\nu \in \bQ \setminus \bZ_+$. Such phenomena were conjectured by P. Etingof in \cite{E1}.

\subsection{Structure of the paper}
In Section \ref{sec:classical_Cherednik_alg}, we recall basic definitions and facts about $H_c(n)$, the Cherednik algebra of type $A$. Some further facts about the blocks of the $\co$-category for $H_c(n)$ will be given in Section \ref{sec:blocks_classical_case}.

In Section \ref{sec:Del_cat} we recall basic facts about the Deligne category $\InnaE{\underline{\mathrm{Rep}}}(S_{\nu})$. We do not give a definition (it can be found in \cite{D}), instead we mention the facts about this category which we will use.

In Sections \ref{sec:Rep_H_c_nu} and \ref{sec:O_H_c_nu}, we define the category $\InnaE{\underline{\mathrm{Rep}}}(H_c(\nu))$ (interpolation of the category of representations of $H_c(n)$) and the category $\InnaE{\underline{\co}}_{\text{  } c,\nu}$ (interpolation of the $\co$-category of $H_c(n)$). We then define the Verma objects $M_{c,\nu}(\T)$ in $\InnaE{\underline{\co}}_{\text{  } c,\nu}$.

\InnaA{In Section \ref{sec:functor_int_case}, we discuss the functor from $\InnaE{\underline{\mathrm{Rep}}}(H_c(\nu=n))$ (the category $\InnaE{\underline{\mathrm{Rep}}}(H_c(\nu))$ for integer value $n \in \bZ_+$ of $\nu$) to $\InnaE{\mathrm{Rep}}(H_c(n))$ (the category of representations of $H_c(n)$).}
In Section \ref{sec:Morph_Verma_necessary_cond}, we give a necessary condition on $(c,\nu)$ which is required in order for a non-trivial map $M_{c,\nu}(\mu) \longrightarrow M_{c,\nu}(\T)$ to exist. This condition is an equation which defines the lines $\mathcal{L}_{\T, \mu, m}$.

In Section \ref{sec:constr_for_H_c_nu} we define constructions for $\InnaE{\underline{\co}}_{\text{  } c,\nu}$ which are necessary in order to describe the triplets $(\T, \mu, m)$ for which $\mathcal{L}_{\T, \mu, m} \subset \mathcal{B}_{\mu, \T}$ (the latter is done in Section \ref{sec:blocks_O_H_c_nu}; we give a full description in the case $\abs{\T} \neq \abs{\mu}$, and a partial description when $\abs{\T} = \abs{\mu}$).

In Section \ref{sec:char_simple_obj} we define the formal character of an object of $\InnaE{\underline{\co}}_{\text{  } c,\nu}$, give a formula for the character of a Verma object $M_{c, \nu}(\T)$, and compute the character of a simple object $L_{c, \nu}(\T)$ in the simplest cases. We also give a positive character formula for the Verma objects and look at examples when such a formula can be derived for some simple objects $L_{c, \nu}(\T)$.


In Section \ref{sec:length_of_Verma_obj} we prove Theorem \ref{introthrm:length_of_Verma_obj} mentioned above, and in particular give a lower bound on the grade of a given Verma object $M_{c, \nu}(\T)$ in which a simple $\InnaE{\underline{\mathrm{Rep}}}(S_{\nu})$-object $X_{\mu}$ can lie.  

\subsection{Acknowledgments}
I would like to thank my advisor, Pavel Etingof, for suggesting the problem and guiding me throughout the project. I would also like to thank Alexander Kleshchev for an explanation on the blocks of the category $\co(H_c(n))$, and Avraham Aizenbud for the helpful discussions. 

I am very grateful to Oleksandr Tsymbaliuk for his comments on the paper, \InnaE{and to the anonymous referee for the helpful comments and suggestions.}

\section{Notation and conventions}

The base field throughout this paper will be $\bC$.

\begin{notation}\label{rational_number_notation}
For a rational number $q \in \bQ$, we write $q = \frac{num(q)}{den(q)}$, where $num(q), den(q) \in \bZ, den(q)>0, \Gcd(num(q), den(q)) =1$.

\end{notation}

\subsection{Symmetric group and Young diagrams}
\begin{notation}
\mbox{}

 \begin{itemize}

 \item $S_n$ will denote the symmetric group ($n \in \bZ_+$). We will denote by $\cS$ the set of reflections $s_{ij}$ in $S_n$. 

 \item The notation $\lambda$ will stand for a partition (weakly decreasing sequence of non-negative integers), a Young diagram $\lambda$ (considered in the English notation), and the corresponding irreducible representation of $S_{\abs{\lambda}}$. Here $\abs{\lambda}$ is the sum of entries of the partition, or, equivalently, the number of cells in the Young diagram $\lambda$. \InnaB{We will sometimes write $\lambda \vdash n$ to say that $\lambda$ is a partition of $n$.} \InnaE{We will denote the set of all Young diagrams of arbitrary size by $\mathcal{P}$}.

\item When referring to a cell in a Young diagram $\lambda$, $(i,j)$ will be the cell in row $i$ and column $j$, with $i,j \geq 1$.

\item The length of the partition $\lambda$, i.e. the number of rows of Young diagram $\lambda$, will be denoted by $l(\lambda)$. 

\item The $i$-th entry of a partition $\lambda$, as well as the length of the $i$-th row of the corresponding Young diagram, will be denoted by $\lambda_i$ (if $i>l(\lambda)$, then $\lambda_i :=0$). The transpose of the Young diagram $\lambda$ will be denoted by $\lambda\check{}$ (so the length of the $i$-th column of $\lambda$ is $\lambda\check{}_i$).

\begin{example}
Consider the Young diagram $\lambda$ corresponding to the partition $(6,5,4,1)$. The length of $\lambda$ is $4$, the size is $16$, and the transpose of $\lambda$ is the Young diagram corresponding to partition $(4,3,3,3,2,1)$:
$$\yng(6,5,4,1) \mapsto \yng(4,3,3,3,2,1) $$
\end{example}

\item $\fh$ (in context of representations of $S_n$) will denote the permutation representation of $S_n$, i.e. the $n$-dimensional representation $\bC^n$ with $S_n$ acting by $g.e_j =e_{g(j)}$ on the standard basis $e_1, .., e_n$ of $\bC^n$.

\item $\hhh$ (in context of representations of $S_n$) will denote the reflection representation of the symmetric group $S_n$: this is the $(n-1)$-dimensional irreducible representation of $S_n$ corresponding to the partition $(n-1, 1)$ of $n$. This representation is the subrepresentation of $\fh$ given by the $S_n$-invariant subspace of all vectors whose sum of coordinates is zero. 

We will use the same notation for the corresponding object in $\InnaE{\underline{\mathrm{Rep}}}(S_{\nu})$.

\item Let $s \in \bZ_{\geq -1}, k \in \bZ_{\geq 0}$.
We will denote by $\tau^s$ a Young diagram consisting of a row with $s+1$ cells, and by $\pi^k$ a Young diagram consisting of a column with $k$ cells. 
\end{itemize}
\end{notation}

\subsection{Monoidal categories}
In the rigid symmetric monoidal category $\InnaE{\underline{\mathrm{Rep}}}(S_{\nu})$, we will use the following notation:
\begin{itemize}
 \item $\triv $ denotes the unit object.
\item For any $U \in \InnaE{\underline{\mathrm{Rep}}}(S_{\nu})$, $ev_{U}: U \otimes U^* \longrightarrow \triv$ is the evaluation morphism, and $coev_{U}: \triv \longrightarrow U^* \otimes U  $ is the coevaluation morphism.
\item For any $U, U' \in \InnaE{\underline{\mathrm{Rep}}}(S_{\nu})$, $\InnaE{c}_{U, U'} : U \otimes U' \longrightarrow U' \otimes U$ is the symmetry isomorphism in $\InnaE{\underline{\mathrm{Rep}}}(S_{\nu})$.
\end{itemize}
 
\begin{notation}\label{notn:act_tens_prod}
 \InnaE{Consider a tensor product $V_1 \otimes V_2 \otimes ... \otimes V_n$ (of objects in a symmetric monoidal category) and a sequence $1\leq i_1 < ...< i_k \leq n$. Denote by $V_{i_1, i_2, ..., i_k}$ the tensor product $V_{i_1} \otimes V_{i_2} \otimes ... \otimes  V_{i_k}$. Let $\Omega \in \End( V_{i_1, i_2, ..., i_k})$.
 
 We write $$\Omega^{i_1,..., i_k} \in \End(V_1 \otimes V_2 \otimes ... \otimes V_n)$$ to denote the endomorphism
 which acts as $\Omega$ on the tensor product of factors $i_1, ..., i_k$, and as $\id$ on the rest. Namely, $\Omega^{i_1,..., i_k}$ corresponds to the endomorphism $ \Omega \otimes \id_{V_{\{1, ...,n\} \setminus \{i_1, i_2, ..., i_k\}}} \in \End(V_{i_1, i_2, ..., i_k} \otimes V_{\{1, ...,n\} \setminus \{i_1, i_2, ..., i_k\}})$ via the symmetry isomorphism $$V_1 \otimes V_2 \otimes ... \otimes V_n \cong V_{i_1, i_2, ..., i_k} \otimes V_{\{1, ...,n\} \setminus \{i_1, i_2, ..., i_k\}}$$
 
 Here is an example: given vector spaces $U_1, U_2, ..., U_5$ and an operator $A \in \End(U_2 \otimes U_4 \otimes U_5)$, the endomorphism $A^{2, 4, 5} \in \End( U_1 \otimes U_2 \otimes ... \otimes  U_5)$ would be the composition 
 $$U_1 \otimes ... \otimes  U_5 \stackrel{\sim}{\rightarrow}  U_2 \otimes U_4\otimes  U_5 \otimes U_1 \otimes U_3 \stackrel{A \otimes \id_{U_1 \otimes U_3}}{\longrightarrow} U_2 \otimes U_4\otimes  U_5 \otimes U_1 \otimes U_3 \stackrel{\sim}{\rightarrow} U_1 \otimes ... \otimes  U_5$$}
 
\end{notation}

\subsection{Blocks in Karoubian categories}

The following definition of a block in a \InnaB{Karoubian} category was given in \cite{CO} \InnaE{(we will mostly use this notion in the context of abelian categories)}:

\begin{definition}[Block of a \InnaB{Karoubian} category]

Let $\mathcal{C}$ be a \InnaB{Karoubian} category. Consider the weakest equivalence relation on the set of isomorphism classes of indecomposable objects such that any two indecomposable objects with a non-zero morphism between them are equivalent.

An equivalence class in this relation will be called a block.

We will use the same term to refer to a full subcategory in $\mathcal{C}$ whose objects are direct sums of indecomposable objects from a single block (in the above sense).

\end{definition}

\section{Classical Cherednik algebra of type A}\label{sec:classical_Cherednik_alg}

This section follows \cite[Chapter 2]{EM}, \cite[Section 2]{ES}.

\subsection{Rational Cherednik algebra of type A}

Let $c \in \bC$ (can be considered as a formal parameter).

\begin{definition}\label{int_Cherednik_alg_def}
 The rational Cherednik algebra of type A of rank $n$ with parameter $c$, denoted by $\overline{H}_c(n)$, is the 
quotient of the algebra $\bC S_n \ltimes T(\fh \oplus \fh^*)$ by the ideal generated by the following relations:

\begin{enumerate}

\item $[x, x']=0, [y,y']=0$ for any $x, x' \in \fh^*$, $y, y' \in \fh$,
\item $[y, x] = (y, x) -\InnaA{ c\sum_{s  \in \cS}((1-s)y, x)s}$ , where  $x\in \fh^*, y \in \fh$, $(,)$ is the natural pairing between $\fh, \fh^*$.

\end{enumerate}

\end{definition}

\begin{remark}
 \InnaD{The last relation can also be rewritten as $[y, x] = (y, x) -\InnaA{ \frac{c}{2}\sum_{s  \in \cS}((1-s)y, (1-s)x)s}$. Both forms of the relation can be interpolated to complex rank; we choose the first one since it is more convenient to interpolate it to complex rank.}
\end{remark}

If we choose dual bases $x_1, ..., x_n$ of $\fh^*$, $y_1, ..., y_n$ of $\fh$ such that $\forall i, j, (y_j, x_i) = \delta_{ij}$, and $s_{ij}(y_i) = y_j, s_{ij}(x_i) = x_j$, then the last relation from Definition \ref{int_Cherednik_alg_def} can be rewritten as 

$$ [y_j, x_i] = cs_{ij} \text{ for any } i \neq j, [y_i, x_i] = 1 - c\sum_{j \neq i} s_{ij}$$

Similarly to the universal enveloping algebra of a Lie algebra, we have:

\begin{proposition}[PBW-type theorem] 
 $\overline{H}_c(n) \cong S\fh^* \otimes \bC S_n \otimes S\fh$ as vector spaces.
\end{proposition}

Continuing the analogy between the universal enveloping algebra of a Lie algebra and a Cherednik algebra, we have a Cartan-type element:

\begin{definition}\label{int_brh_def}
 The element $\overline{\brh} \in \overline{H}_c(n)$ is defined as 
$$\overline{\brh} := \sum x_i y_i + \dim(\fh)/2 -c\sum _{s  \in \cS} s = \sum x_i y_i + n/2 -c\sum _{s  \in \cS} s$$
\end{definition}

This element satisfies: $g \overline{\brh} g^{-1} =\overline{\brh}, [\overline{\brh}, x_i] = x_i, [\overline{\brh}, y_i] = -y_i \text{ for any } g \in S_n, i,j \in \{1,.., n\}$, and can also be written as $\overline{\brh}= \frac{1}{2} \sum_i (x_i y_i + y_i x_i)$.

We will use the ``reduced'' Cherednik algebra $H_c(n)$, which corresponds to the reflection representation $\hhh = span_{\bC}\{y_i -y_j \mid i, j \in \{1,...,n\}, i\neq j\}$ rather than the permutation representation $\fh$ of $S_n$. It is defined in the same way as $\overline{H}_c(n)$, but with $\hhh$ appearing instead of $\fh$ (see \cite{EM}, or \cite[Section 2]{GORS}). In fact, we have: $\overline{H}_c(n) = H_c(n) \otimes \mathbb{A}_1$, $\mathbb{A}_1$ being the Weyl algebra. This follows from the fact that the subalgebra of $\overline{H}_c(n)$ generated by $\sum y_i, \sum x_i$ is isomorphic to $\mathbb{A}_1$, and commutes with the subalgebra $H_c(n)$ of $\overline{H}_c(n)$.

For $H_c(n)$, $\overline{\brh}$ is not an element of $H_c(n)$, but the element $$\brh :=\overline{\brh} - \frac{(\sum x_i)(\sum y_i) + (\sum y_i)(\sum x_i)}{2n} $$ of $H_c(n)$ plays the analogous role.

\begin{remark}
\InnaE{Note that the relation between the algebras $H_c(n), \overline{H}_c(n)$ is similar to the relation between the Lie algebras $\mathfrak{sl}_n, \mathfrak{gl}_n$.} 
\end{remark}


\subsection{\texorpdfstring{Category $\co(H_c(n))$}{Category O for the rational Cherednik algebra}}

The category $\co(H_c(n))$ is defined as the category of all modules over $H_c(n)$ which are finitely generated over $S\InnaD{\fh^*_0}$, and on which $\hhh$ acts locally nilpotently.

\subsection{Verma modules}\label{ssec:class_Verma_modules}
Let $\T$ be an irreducible representation of $S_n$. 
Then one can define $M_{c,n}(\T) : = H_c(n) \otimes_{\bC  S_n \ltimes S\hhh} \T$, where \InnaE{the augmentation ideal of} $S\hhh$ acts on $\T$ by zero. This is the Verma module corresponding to $\T$. It is a lowest weight module whose underlying space is isomorphic to  $S\InnaD{\fh^*_0} \otimes \T$, where $S_n$ acts on each subspace $S^m \InnaD{\fh^*_0} \otimes \T$, $\InnaD{\fh^*_0}$ acts by multiplication on $S\InnaD{\fh^*_0}$, and $\hhh$ acts according to the commutator relation given in the definition (if $c=0$, then $\hhh$ acts by differentiation on $S\InnaD{\fh^*_0}$).
The element $\brh$ acts locally finitely on modules from $\co(H_c(n))$, with finite dimensional generalized eigenspaces.
In particular, $\brh$ acts semisimply on $M_{c,n}(\T)$, with lowest eigenvalue $$h_{c,n}(\T) = \dim(\hhh)/2 - c\sum _{s  \in \cS}s\rvert_{\T} $$
(the rest of the eigenvalues on \InnaB{$ S^m \InnaD{\fh^*_0} \otimes \T$}, $m > 0$ being $h_{c,n}(\T) + m$).

\subsubsection{Sum of transpositions element}\label{ssec:sum_of_transp_elem}
\InnaA{Recall that $\bC S_n$ has a central element $\Omega := \sum _{s  \in \cS} s \in \bC[S_n]$.} This element acts on an irreducible representation parameterized by the Young diagram $\T$ by the scalar $ct(\T) = \sum_{(i,j) \in \T} (j-i)$, called ``content of $\T$'' (here $(i,j)$ denotes cell in row $i$, column $j$ of the Young diagram $\T$).

So $$h_{c,n}(\T) = \dim(\hhh)/2 - c\sum _{s  \in \cS}s\rvert_{\T} = \dim(\hhh)/2 - c \Omega\rvert_{\T} = \frac{n-1}{2} - c \cdot ct(\T) $$


\section{\texorpdfstring{Category $\InnaE{\underline{\mathrm{Rep}}}(S_\nu)$}{Deligne's category}}\label{sec:Del_cat}

\subsection{General description}
We recall some basic facts about Deligne's category $\InnaE{\underline{\mathrm{Rep}}}(S_{\nu})$ (see \cite{D}, \cite{E1}, and \cite[Section 2]{M}).

This is a family of Karoubian \InnaE{rigid symmetric monoidal} categories over $\bC$ defined for any $\nu \in \bC$, and ``flat with respect to parameter $\nu$'' (one can view $\nu$ as a formal parameter, with $Obj(\InnaE{\underline{\mathrm{Rep}}}(S_{\nu}))$ not depending on $\nu$, and the $\Hom$ spaces being modules over $\bC((\nu))$).

\InnaA{For any $\nu \in \bC$, the category $\InnaE{\underline{\mathrm{Rep}}}(S_{\nu})$ is generated, as a Karoubian tensor category, by one object, denoted $\fh$. This object is the analogue of the permutation representation of $S_n$, and any object in $\InnaE{\underline{\mathrm{Rep}}}(S_{\nu})$ is a direct summand in a direct sum of tensor powers of $\fh$. 

\begin{remark}
 \InnaE{Since $\InnaE{\underline{\mathrm{Rep}}}(S_{\nu})$ is a rigid symmetric monoidal category, the notion of dimension of an object is defined. By definition, the dimension of the object $\fh$ is $\nu$.}
\end{remark}

For $\nu \notin \bZ_{+}$, $\InnaE{\underline{\mathrm{Rep}}}(S_{\nu})$ is a semisimple abelian category. 

For $\nu \in \bZ_{+}$, the category $\InnaE{\underline{\mathrm{Rep}}}(S_{\nu})$ has a tensor ideal $\idealI_{\nu}$, called the ideal of negligible morphisms (this is the ideal of morphisms $f: X \longrightarrow Y$ such that $tr(fu)=0$ for any morphism $u: Y \longrightarrow X$). In that case, the classical category $\InnaE{\mathrm{Rep}}(S_n)$ of finite-dimensional representations of the symmetric group for $n:=\nu$ is equivalent to $\InnaE{\underline{\mathrm{Rep}}}(S_{\nu})/\idealI_{\nu}$ (equivalent as Karoubian rigid symmetric monoidal categories). Note that $\fh$ is sent to the permutation representation of $S_n$ under this equivalence.

\begin{remark}
 Although $\InnaE{\underline{\mathrm{Rep}}}(S_{\nu})$ is not semisimple and not even abelian when $\nu \in \bZ_+$, a weaker statement holds (see \cite[Proposition 5.1]{D}): consider the full subcategory $\InnaE{\underline{\mathrm{Rep}}}(S_{\nu})^{(\nu/2)}$ of $\InnaE{\underline{\mathrm{Rep}}}(S_{\nu})$ whose objects are directs summands of sums of $\fh^{\otimes m}, 0\leq m \leq \frac{\nu}{2}$. This subcategory is abelian semisimple.
\end{remark}

}
\begin{notation}
 We will denote Deligne's category for integer value $n \geq 0$ of $\nu$ as $\InnaE{\underline{\mathrm{Rep}}}(S_{\nu=n})$, to distinguish it from the classical category $\InnaE{\mathrm{Rep}}(S_{n})$ of representations of the symmetric group $S_{n}$. Similarly for other categories arising in this text.
\end{notation}

The indecomposable objects of $\InnaE{\underline{\mathrm{Rep}}}(S_{\nu})$, regardless of the value of $\nu$, are labeled by all Young diagrams (of arbitrary size). We will denote the indecomposable object in $\InnaE{\underline{\mathrm{Rep}}}(S_{\nu})$ corresponding to the Young diagram $\T$ by $X_{\T}$.

\begin{notation}
 We will denote by $\tilde{\lambda}(n)$ the Young diagram obtained from $\lambda$ by adding a top row of size $n-\abs{\lambda}$, where $n \geq \lambda_1+\abs{\lambda}$ is an integer.
\end{notation}

\begin{example}
 $$ \lambda=\yng(6,5,4,1) \mapsto \tilde{\lambda}(\InnaA{31})=\yng(15,6,5,4,1)$$
\end{example}

\InnaA{For non-negative integer $\nu =:n$, we have: the partitions $\lambda$ for which $X_{\lambda}$ has a non-zero image in the quotient $\InnaE{\underline{\mathrm{Rep}}}(S_{\nu=n})/\idealI_{\nu=n} \cong \InnaE{\mathrm{Rep}}(S_n)$ are exactly the $\lambda$ for which $\lambda_1+\abs{\lambda} \leq n$.}

If $\lambda_1+\abs{\lambda} \leq n$, then the image of $X_{\lambda} \in \InnaE{\underline{\mathrm{Rep}}}(S_{\nu=n})$ in $\InnaE{\mathrm{Rep}}(S_n)$ would be the irreducible representation of $S_n$ corresponding to the Young diagram $\tilde{\lambda}(n)$. 

Intuitively, one can treat the indecomposable objects of $\InnaE{\underline{\mathrm{Rep}}}(S_{\nu})$ as if they were parameterized by ``Young diagrams with very long top row''. The indecomposable object $X_{\lambda}$ would be treated as if it corresponded to $\tilde{\lambda}(\nu)$, i.e. a Young diagram obtained by adding a very long top row (``of size $\nu -\abs{\lambda}$''). This point of view is useful to understand how to extend constructions for $S_n$ involving Young diagrams to $\InnaE{\underline{\mathrm{Rep}}}(S_{\nu})$.

\begin{example}
 The indecomposable object $X_{\lambda}$, where $$\lambda= \yng(6,5,4,1)$$ can be thought of as a Young diagram with a ``very long top row of length $(\nu- 16)$'':
$$ \yng(35,6,5,4,1)$$
\end{example}

\begin{notation}
 Let $\hhh:=X_{\yng(1)}$ be the indecomposable object in $\InnaE{\underline{\mathrm{Rep}}}(S_{\nu})$ corresponding to the one-box Young diagram (that would be the analogue of the reflection representation in $\InnaE{\mathrm{Rep}}(S_n)$). 
\end{notation}

Its dual in $\InnaE{\underline{\mathrm{Rep}}}(S_{\nu})$, denoted by $\InnaD{\fh^*_0}$, is isomorphic to $\hhh$, since all the objects in $\InnaE{\underline{\mathrm{Rep}}}(S_{\nu})$ have a symmetric form defined on them (analogue of the invariant symmetric form on the representations of the symmetric group over a field of characteristic zero).

\begin{remark}
 $\fh \cong \hhh \oplus \triv$ for any $\nu \neq 0$, in particular, for any $\nu \not\in \bZ_+$. But for $\nu=0$, we have $\fh \cong \hhh$.
 \InnaE{Computing the dimension of $\hhh$, we conclude that $\dim_{\InnaE{\underline{\mathrm{Rep}}}(S_{\nu})}(\hhh) = \nu-1$ when $\nu \neq 0$, and $\dim_{\InnaE{\underline{\mathrm{Rep}}}(S_{\nu})}(\hhh) = 0$ if $\nu =0$.}
\end{remark}

\begin{example}
$\Lambda^k (\hhh) = X_{\pi^k}$ for any $k \geq 0$; in particular, $X_{\emptyset} = \triv$.
\end{example}

\subsection{Pieri's rule}
We have an interpolation of Pieri's rule (see \cite[Proposition 5.15]{CO}, \cite{E1}, \cite{D} \InnaE{concerning the Pieri rule in the Deligne categories} $\InnaE{\underline{\mathrm{Rep}}}(S_{\nu})$, and \cite[Par. 4.3]{FH} for Pieri's rule \InnaE{for representations of $S_n$}):

\begin{prop}\label{Pieri}[Pieri's rule]
 Let \InnaA{$\nu \notin \bZ_+$ and let} $X_{\T}$ be a simple object of $\InnaE{\underline{\mathrm{Rep}}}(S_{\nu})$. As objects of $\InnaE{\underline{\mathrm{Rep}}}(S_{\nu})$, 
$$\hhh \otimes X_{\T} \cong \bigoplus_{\mu \in P_{\T}^{+} \cup P_{\T}^{-} \cup P_{\T}^{0}} X_{\mu} \oplus cc(\T)X_{\T} $$
where ${P_{\T}}^{+}, {P_{\T}}^{-},{P_{\T}}^{0}$ are the sets of all Young diagrams obtained from $\T$ by adding, deleting, or moving a corner cell, respectively (note that $\T \notin {P_{\T}}^{0}$), and $cc(\T)$ is the number of corner cells of $\T$.
\end{prop}

\subsection{Analogue of the sum of transpositions element for $\InnaE{\underline{\mathrm{Rep}}}(S_{\nu})$}\label{ssec:Omega_elem}
Consider the following function \InnaE{$C_{\nu}: \mathcal{P} \rightarrow \bC$}:
$$C_{\nu}(\mu):= \frac{(\nu - \abs{\mu})(\nu - \abs{\mu} -1)}{2}- \abs{\mu} + ct(\mu)$$
This function interpolates the function \InnaE{$ct:\mathcal{P} \rightarrow \bZ$ in the sense that $ct(\tilde{\mu}(n)) = C_{\nu = n}(\mu)$ for any Young diagram $\mu$ and any $n \in \bZ_{\geq \abs{\mu}+\mu_1}$.}

J. Comes and V. Ostrik defined in \cite{CO} an endomorphism $\Omega$ of the identity functor of the category $\InnaE{\underline{\mathrm{Rep}}}(S_{\nu})$ for arbitrary $\nu$ (including non-negative integers) \InnaE{such} that for any indecomposable object $X_{\mu}$ in $\InnaE{\underline{\mathrm{Rep}}}(S_{\nu})$, the morphism $(\Omega - C_{\nu}(\mu) \id_{X_{\mu}})$ is a nilpotent endomorphism of $X_{\mu}$. \InnaE{The endomorphism $\Omega$ is an interpolation of the central element $\Omega \in \bC[S_n]$ defined in Subsection \ref{ssec:sum_of_transp_elem}}.

\begin{remark}\label{rmk:Omega_endom_Deligne}
 For $X_{\mu} \in \InnaE{\underline{\mathrm{Rep}}}(S_{\nu})$ such that $\nu \notin \{0, 1, ...,\abs{\mu}+\mu_1 \}$, we have: 
$$\dim \End_{\InnaE{\underline{\mathrm{Rep}}}(S_{\nu})}(X_{\mu}) =1$$ Since the endomorphism $\Omega \rvert_{X_{\mu}} - C_{\nu}(\mu) \id_{X_{\mu}}$ of $X_{\mu}$ is nilpotent, we conclude that it is zero, and 
 $$\Omega \rvert_{X_{\mu}} = C_{\nu}(\mu) \id_{X_{\mu}} = \left(\frac{(\nu - \abs{\mu})(\nu - \abs{\mu} -1)}{2}- \abs{\mu} + ct(\mu)\right) \id_{X_{\mu}} $$
\end{remark}

\section{\texorpdfstring{Category $\InnaE{\underline{\mathrm{Rep}}}(H_c(\nu))$}{Category of representations of the rational Cherednik algebra in complex rank}}\label{sec:Rep_H_c_nu}

Let $c, \nu \in \bC$.

Based on the definition of the category $\InnaE{\underline{\mathrm{Rep}}}(S_{\nu})$ defined by Deligne, P. Etingof defined in \cite{E1} the category $\InnaE{\underline{\mathrm{Rep}}}(H_c(\nu))$:
\InnaE{
 \begin{definition}[Category $\InnaE{\underline{\mathrm{Rep}}}(H_c(\nu))$]\label{Cat_Rep_H_nu_def}
  $\InnaE{\underline{\mathrm{Rep}}}(H_c(\nu))$ is the category whose objects are triples $(M, {x}, {y})$, where $M$ is an \InnaE{\rm{ind}}-object of Deligne's category $\InnaE{\underline{\mathrm{Rep}}}(S_{\nu})$, and $${x}: \InnaD{\hhh^*} \otimes M \longrightarrow  M $$ 
$${y}: \hhh \otimes M \rightarrow  M$$ are morphisms in $\InnaE{\underline{\mathrm{Rep}}}(S_{\nu})$ satisfying the conditions: 

\begin{enumerate}
 \item The morphism $${x} \circ (\id \otimes {x}) \circ ((\id- \InnaE{c}_{\InnaD{\hhh^*}, \InnaD{\hhh^*}}) \otimes \id): \InnaD{\hhh^*} \otimes \InnaD{\hhh^*} \otimes M \longrightarrow  M$$ is 0. 
 \item The morphism $${y} \circ (\id \otimes {y}) \circ ((\id- \InnaE{c}_{\hhh, \hhh}) \otimes \id): \hhh \otimes \hhh \otimes M \longrightarrow  M$$ is $0$.
 \item The morphism $${y} \circ (\id \otimes {x}) - {x} \circ (\id \otimes {y}) \circ (\InnaE{c}_{\hhh, \InnaD{\hhh^*}} \otimes \id): \hhh \otimes \InnaD{\hhh^*} \otimes M \longrightarrow  M$$ equals
$$ev_{\hhh} \otimes  \id - c (ev_{\hhh} \otimes \id) \circ (\InnaF{ \Omega^3 - \Omega^{1,3}}) $$
\end{enumerate}

Here 
\begin{itemize}
 \item $\sigma$ is the symmetry isomorphism in $\InnaE{\underline{\mathrm{Rep}}}(S_{\nu})$ (i.e. $\InnaE{c}_{U, U'} : U \otimes U' \longrightarrow U' \otimes U$ is the canonical isomorphism in $\InnaE{\underline{\mathrm{Rep}}}(S_{\nu})$, where $U, U'$ are objects of $\InnaE{\underline{\mathrm{Rep}}}(S_{\nu})$),
\item $ev_{\hhh}: \hhh \otimes \InnaD{\hhh^*} \longrightarrow \triv $ is the evaluation map in \InnaE{the} category $\InnaE{\underline{\mathrm{Rep}}}(S_{\nu})$, 
\item $\Omega$ is the endomorphism of the identity functor of $\InnaE{\underline{\mathrm{Rep}}}(S_{\nu})$ described in \ref{ssec:Omega_elem}, \InnaF{and the notation $(\cdot)^3, (\cdot)^{1,3}$ is defined in \ref{notn:act_tens_prod}}.

\end{itemize}
 
Morphisms in $\InnaE{\underline{\mathrm{Rep}}}(H_c(\nu))$ are defined to be morphisms in $\InnaE{\mathrm{ind}}-\InnaE{\underline{\mathrm{Rep}}}(S_{\nu})$ compatible with maps ${x},{y}$.
 \end{definition}

 This category corresponds to the category $\mathrm{Rep}(H_c(n))$ of representations of the reduced rational Cherednik algebra. One can similarly define the Deligne category $\InnaE{\underline{\mathrm{Rep}}}(\overline{H}_c(\nu))$ corresponding to the category $\mathrm{Rep}(\overline{H}_c(\InnaF{n}))$ of representations of the non-reduced rational rational Cherednik algebra. 
 
 As in the classical case, the relation between the categories $\InnaE{\underline{\mathrm{Rep}}}({H}_c(\nu))$ and $\InnaE{\underline{\mathrm{Rep}}}(\overline{H}_c(\nu))$ will be similar to the relation between the categories $\mathrm{Rep}(\mathfrak{sl}_n), \mathrm{Rep}(\mathfrak{gl}_n)$.

}
\section{\texorpdfstring{Category $\InnaE{\underline{\co}}_{\text{  } c,\nu}$}{Category O in complex rank}}\label{sec:O_H_c_nu}

\subsection{\texorpdfstring{Definition of $\InnaE{\underline{\co}}_{\text{  } c,\nu}$}{Definition of the category O in complex rank}}\label{ssec:O_H_c_nu_def}

Consider the algebra objects $S\InnaD{\fh^*_0}$, $S\hhh$ in $\InnaE{\mathrm{ind}}-\InnaE{\underline{\mathrm{Rep}}}(S_{\nu})$. 


Given maps $x,y$ as above, we \InnaE{obtain} maps $$\overline{x}_M:S\InnaD{\fh^*_0} \otimes M \longrightarrow  M $$ 
$$\overline{y}_M:S\hhh \otimes M \rightarrow  M$$
So we can speak of $M$ as an $S\InnaD{\fh^*_0}$-module and an $S\hhh$-module in $\InnaE{\mathrm{ind}}-\InnaE{\underline{\mathrm{Rep}}}(S_{\nu})$.

\begin{definition}
 The category $\InnaE{\underline{\co}}_{\text{  } c,\nu}$ is defined as the full subcategory of $\InnaE{\underline{\mathrm{Rep}}}(H_c(\nu))$ whose objects are $(M, x, y)$ such that 
\begin{itemize}
 \item $M$ is finitely generated over $S\InnaD{\fh^*_0}$ in $\InnaE{\underline{\mathrm{Rep}}}(S_{\nu})$ (i.e. $M$ is a quotient of a ``finitely generated free $S\InnaD{\fh^*_0}$ module'' $S\InnaD{\fh^*_0} \otimes V$, where $V$ is an object of $\InnaE{\underline{\mathrm{Rep}}}(S_{\nu})$). 
\item $M$ is locally nilpotent over $\hhh$, in the following sense:

 for any $\InnaE{\underline{\mathrm{Rep}}}(S_{\nu})$-subobject $T \subset M $, there exists a non-negative integer $m$ (depending on $T$) such that the map $S^m \hhh \otimes T \longrightarrow M$ (this is the restriction to $T$ of the map $S^m\hhh \otimes M \rightarrow  M$) equals zero.
\end{itemize}

\end{definition}



\begin{definition}\label{def:brh}
\InnaE{Define ${\brh} \in \End(\InnaE{\mathrm{Forget}}_{\InnaE{\underline{\co}}_{\text{  } c,\nu} \rightarrow \InnaE{\mathrm{ind}}-\InnaE{\underline{\mathrm{Rep}}}(S_{\nu})})$ \InnaA{by letting ${\brh}_M := \frac{1}{2}(\psi_1 + \psi_2)$ for every $M \in \InnaE{\underline{\co}}_{\text{  } c,\nu}$, where $\psi_1, \psi_2: M \longrightarrow M$, }
$$\psi_1 = {x} \circ (\id \otimes {y}) \circ (coev_{\hhh^*} \otimes \id_M): M \longrightarrow \hhh^* \otimes \hhh \otimes M \longrightarrow M $$
$$\psi_2 = {y} \circ (\id \otimes {x}) \circ (coev_{\hhh}\otimes \id_M) : M \longrightarrow \hhh \otimes \hhh^* \otimes M \longrightarrow M $$}
\end{definition}


Similarly to the classical case, \InnaE{an easy computation shows that $\brh$ has the following properties:}
\begin{lemma}\label{lem:brh_property}
\InnaE{The endomorphism $\brh \in \End(\InnaE{\mathrm{Forget}}_{\InnaE{\underline{\co}}_{\text{  } c,\nu} \rightarrow \InnaE{\mathrm{ind}}-\InnaE{\underline{\mathrm{Rep}}}(S_{\nu})})$ can also be defined by setting 
$\brh \rvert_M: M \rightarrow M$ to be $$\brh \rvert_M:= {x} \circ (\id \otimes {y}) \circ (coev_{\hhh^*} \otimes \id_M) + \frac{\dim (\hhh)}{2} - c \Omega_M$$ (the endomorphism $\Omega$ is defined in Subsection \ref{ssec:Omega_elem}).

In addition, $\brh$ satisfies the the following commutation relations:} 
 \begin{align*}
&y_M \circ (\id_{\InnaA{\hhh}} \otimes \brh) - \brh \circ y_M = y_M : \hhh \otimes M \longrightarrow M \\
&x_M \circ (\id_{\InnaA{\InnaD{\hhh^*}}} \otimes \brh) - \brh \circ x_M = \InnaA{-} x_M : \InnaD{\hhh^*} \otimes M \longrightarrow M
\end{align*}
\end{lemma}
\InnaA{
\begin{definition}[Indecomposable singular subobject]
Let $U \in \InnaE{\underline{\co}}_{\text{  } c,\nu}, V \in \InnaE{\underline{\mathrm{Rep}}}(S_{\nu})$ \InnaE{such} that $V \subset U$ and the map $y_U$ restricted to $\hhh \otimes V$ is zero. Then $V$ is called ``singular in $U$''. 
If, furthermore, $V$ is an indecomposable (resp. simple) object in $\InnaE{\underline{\mathrm{Rep}}}(S_{\nu})$, then we say that $V$ is an indecomposable (resp. simple) singular $\InnaE{\underline{\mathrm{Rep}}}(S_{\nu})$-subobject of $U$.
 
\end{definition}

\begin{remark}
 Note that by definition of $\InnaE{\underline{\co}}_{\text{  } c,\nu}$, each object contains an indecomposable singular $\InnaE{\underline{\mathrm{Rep}}}(S_{\nu})$-subobject.
\end{remark}

%
%
%
}

\subsection{\texorpdfstring{Verma objects in \InnaE{the} category $\InnaE{\underline{\co}}_{\text{  } c,\nu}$}{Verma object in category O (complex rank)}}\label{ssec:Verma_obj}

One can define Verma objects in the category $\InnaE{\underline{\co}}_{\text{  } c,\nu}$ as follows: 

Consider the category $\InnaF{\underline{\mathrm{Rep}}}(B_{\nu})$ of pairs $(M, y_M)$, where $M$ is an \InnaE{\rm{ind}}-object of $\InnaE{\underline{\mathrm{Rep}}}(S_{\nu})$, $$y: \hhh \otimes M \rightarrow  M$$ is a morphism in $\InnaE{\underline{\mathrm{Rep}}}(S_{\nu})$ satisfying the condition: the morphism $$y \circ (\id \otimes y) \circ ((\id- \InnaE{c}_{\hhh, \hhh}) \otimes \id): \hhh \otimes \hhh \otimes M \longrightarrow  M$$ is zero.
The morphisms in $\InnaF{\underline{\mathrm{Rep}}}(B_{\nu})$ are morphisms of $\InnaE{\underline{\mathrm{Rep}}}(S_{\nu})$ \InnaE{\rm{ind}}-objects compatible with the $y_M$ maps.

This is the analogue of the category of representations of the ``Borel'' subalgebra $B_n = \bC[S_n] \ltimes S\hhh$.

Then we have the restriction functor $Res: \InnaE{\underline{\mathrm{Rep}}}(H_c(\nu)) \longrightarrow \InnaF{\underline{\mathrm{Rep}}}(B_{\nu})$, and it has a left adjoint which is the induction functor $Ind: \InnaF{\underline{\mathrm{Rep}}}(B_{\nu}) \longrightarrow \InnaE{\underline{\mathrm{Rep}}}(H_c(\nu))$. The induction functor takes a pair $(M, y_M)$ to a triple $(  S\InnaD{\fh^*_0} \otimes M, y_{S\InnaD{\fh^*_0} \otimes M}, x_{S\InnaD{\fh^*_0} \otimes M})$, where the map $x_{S\InnaD{\fh^*_0} \otimes M}:\InnaD{\fh^*_0} \otimes S\InnaD{\fh^*_0} \otimes M \rightarrow S\InnaD{\fh^*_0} \otimes M$ is the multiplication map of $\InnaD{\fh^*_0} \otimes S\InnaD{\fh^*_0} \rightarrow S\InnaD{\fh^*_0}$ tensored with $\id_{M}$, and $y_{S\InnaD{\fh^*_0} \otimes M}:\hhh \otimes S\InnaD{\fh^*_0} \otimes M \rightarrow S\InnaD{\fh^*_0} \otimes M$ is \InnaE{the unique} map satisfying conditions (1, 3) of Definition \ref{Cat_Rep_H_nu_def}, \InnaE{such} that $y_{S\InnaD{\fh^*_0} \otimes M}\mid_{\hhh \otimes 1 \otimes M} = 
y_M$.

\begin{definition}[Verma Object]
 
Consider an indecomposable object $X_{\T}$ in $\InnaE{\underline{\mathrm{Rep}}}(S_{\nu})$, and let the map $y_{X_{\T}}:\hhh \otimes X_{\T} \rightarrow  X_{\T}$ be zero. This makes $(X_{\T},y_{X_{\T}})$ an object of $\InnaF{\underline{\mathrm{Rep}}}(B_{\nu})$. 

Define the Verma object of lowest weight $X_{\T}$ as $Ind(X_{\T},y_{X_{\T}})$ (so it is an \InnaE{\rm{ind}}-object of $\InnaE{\underline{\mathrm{Rep}}}(S_{\nu})$). It will be denoted by $M_{c,\nu}(\T)$ (or, if $c, \nu$ are fixed, $M(\T)$ for short).
\end{definition}
 
\begin{obsr}\label{obsr:x_y_act_Verma_obj}

It is easy to see that as a \InnaB{$\InnaE{\underline{\mathrm{Rep}}}(S_{\nu})$ \InnaE{\rm{ind}}-object} $M_{c,\nu}(\T) =  S\InnaD{\fh^*_0} \otimes X_{\T}$, where $X_{\T}$ is an indecomposable object of $\InnaE{\underline{\mathrm{Rep}}}(S_{\nu})$. 

The map $\overline{x}_{M_{c,\nu}(\T)}: S\InnaD{\fh^*_0} \otimes S\InnaD{\fh^*_0} \otimes X_{\T} \longrightarrow S\InnaD{\fh^*_0} \otimes X_{\T}$ is the multiplication map of $S\InnaD{\fh^*_0}$ tensored with $\id_{X_{\T}}$.

The map $\overline{y}_{M_{c,\nu}(\T)}: S\hhh \otimes S\InnaD{\fh^*_0} \otimes X_{\T} \longrightarrow S\InnaD{\fh^*_0} \otimes X_{\T}$ is the map for which condition (3) from Definition \ref{Cat_Rep_H_nu_def} holds, and thus is a deformation of the differentiation map $ S\hhh \otimes S\InnaD{\fh^*_0} \longrightarrow S\InnaD{\fh^*_0}$ tensored with $\id_{X_{\T}}$.

\end{obsr}


\begin{notation}
 To avoid confusion, the Verma module in $\co(H_c(n))$ whose lowest weight corresponds to the Young diagram $\tilde{\T}(n)$ is denoted by $M_{c,n}(\tilde{\T}(n))$, while the Verma object in $\InnaE{\underline{\co}}_{\text{  } c,\nu=n}$ corresponding to the Young diagram $\T$ is denoted by $M_{c,\nu=n}(\T)$. 
\end{notation}

\begin{definition}\label{h_c_nu}
\InnaE{Let $c, \nu \in \bC$. We define the function $h_{c, \nu}: \mathcal{P} \rightarrow \bC$ as $$h_{c,\nu}(\T) := \frac{\dim(\hhh)}{2} - c \cdot C_{\nu}(\T) = \frac{\nu-1}{2} - c \cdot \left(\frac{(\nu - \abs{\T})(\nu - \abs{\T} -1)}{2} - \abs{\T} +ct(\T)\right) $$ }

\end{definition} 
\InnaE{We immediately see that this function interpolates the functions $h_{c, n}$ defined in Subsection \ref{ssec:sum_of_transp_elem}. Namely, $h_{c, n}(\tilde{\T}(n)) = h_{c, \nu=n}(\T)$ for large enough $n$.}

\begin{remark}
 Note that given $X_{\T} \in \InnaE{\underline{\mathrm{Rep}}}(S_{\nu})$ such that $\nu \notin \{0, 1, ...,\abs{\T}+\T_1 \}$, we have $$h_{c,\nu}(\T) \id_{X_{\T}} = \frac{\dim(\hhh)}{2}\id_{X_{\T}} - c \Omega\rvert_{X_{\T}}$$ \InnaE{which is an interpolation of the formula in Subsection \ref{ssec:class_Verma_modules}. Here $\Omega$ is the endomorphism of the identity functor of $\InnaE{\underline{\mathrm{Rep}}}(S_{\nu})$ described in \InnaB{Subsection} \ref{ssec:Omega_elem}, and identity follows from Remark \ref{rmk:Omega_endom_Deligne}.}  
\end{remark}
\newpage
\begin{example}\mbox{} 
\begin{itemize}
 \item $ct(\T^{\InnaA{s-1}}) = \frac{s(s-1)}{2}$, so $h_{c,\nu}(\T^s) = \frac{\dim(\hhh)}{2} - c \cdot \left( \frac{\nu^2-\nu}{2} -s\nu +s^2-s \right)$.
\item $ct(\pi^n) = -\frac{n(n-1)}{2}$, so $h_{c,\nu}(\pi^n) = \frac{\dim(\hhh)}{2} - c \cdot \left( \frac{\nu^2-\nu}{2} -n\nu  \right)$.
\end{itemize}

\end{example}

\InnaE{We now consider the action of the endomorphism $\brh$ on Verma objects.}
\begin{prop}\label{prop:action_brh_Verma}
\InnaE{Let $c \in \bC$.

The endomorphism $\brh: M_{c,\nu}(\T) \longrightarrow  M_{c,\nu}(\T)$ of the \InnaE{\rm{ind}}-object $M_{c,\nu}(\T) = \bigoplus_{m \geq 0} S^m \InnaD{\fh^*_0} \otimes X_{\T}$ of $\InnaE{\underline{\mathrm{Rep}}}(S_{\nu})$ decomposes as a sum 
$$\brh = \oplus_{m \geq 0} \phi_m$$ where $\phi_m \in \End( S^m \InnaD{\fh^*_0} \otimes X_{\T})$ for every $m \geq 0$. \InnaF{Furthermore, $$\phi_0 = \frac{\dim(\hhh)}{2}\id_{X_{\T}} - c \Omega\rvert_{X_{\T}}$$ and the endomorphism $ (\phi_m - (h_{c,\nu}(\T) + m)\id_{X_{\T}})$ of $S^m \InnaD{\fh^*_0} \otimes X_{\T}$ is nilpotent for every $m \geq 0$. In particular, if $\nu \notin \{0, 1, ..., \abs{\T} + \T_1 \}$, then $$\phi_m  = ( h_{c,\nu}(\T) +m ) \id_{S^m \fh^*_0 \otimes X_{\T}}$$ for any $m \geq 0$.}}
\end{prop}

\begin{proof}
\InnaE{From the definition of $\brh$ (Definition \ref{def:brh}) and Observation \ref{obsr:x_y_act_Verma_obj}, we immediately obtain a decomposition $\brh = \oplus_{m \geq 0} \phi_m$ where $\phi_m \in \End( S^m \InnaD{\fh^*_0} \otimes X_{\T})$ (this amounts to checking that $\brh$ acts by operators of degree zero on the graded space $S \InnaD{\fh^*_0} \otimes X_{\T}$, with the $\bZ_+$-grading inherited from $S \InnaD{\fh^*_0}$). 

\mbox{}

We now want to prove that the endomorphism $ (\phi_m - (h_{c,\nu}(\T) + m)\id_{X_{\T}})$ of $S^m \InnaD{\fh^*_0} \otimes X_{\T}$ is nilpotent.



First we consider the case $m=0$. By Lemma \ref{lem:brh_property}, the endomorphism $\phi_0$ acts on $X_{\T}$ as the operator $\frac{\dim (\hhh)}{2} - c \Omega_M$, which has a unique generalized eigenvalue $h_{c,\nu}(\T)$, as it was said in Subsection \ref{ssec:Omega_elem}.

\InnaF{Next, we prove by induction on $m \geq 0$} that $\phi_m$ has only one generalized eigenvalue, which is $h_{c,\nu}(\T) + m$. Just like in the classical setting, this is a direct consequence of the commutation relation $$(\brh - (a+1))^N \circ x_{M_{c,\nu}(\T)} = x_{M_{c,\nu}(\T)} \circ \left( \id_{\hhh^*} \otimes (\brh - a)^N \right)$$ Here both sides are morphisms $\hhh^* \otimes S \hhh^* \otimes X_{\T} \rightarrow S \hhh^* \otimes X_{\T}$, $a$ is (any) scalar and $N$ is any non-negative integer. This commutation relation is easily obtained from Lemma \ref{lem:brh_property}.

\InnaF{Finally, if $\nu \notin \{0, 1, ..., \abs{\T} + \T_1 \}$, then the object $X_{\T}$ is simple, and thus 
$$\phi_0 = \frac{\dim(\hhh)}{2}\id_{X_{\T}} - c \Omega\rvert_{X_{\T}} = h_{c, \nu}(\T) \id_{X_{\T}}$$ Due to the commutation relation above, we conclude that 
$$\phi_m  = ( h_{c,\nu}(\T) +m ) \id_{S^m \InnaD{\fh^*_0} \otimes X_{\T}}$$ for any $m \geq 0$.
}}
\end{proof}

\mbox{}

This proposition means that $M(\T)$ (as a $\InnaE{\underline{\mathrm{Rep}}}(S_{\nu})$ \InnaE{\rm{ind}}-object) has a grading by eigenvalues of $\brh$, which is the natural $\bZ_{+}$-grading on \InnaF{$M(\T) \cong S \hhh^* \otimes X_{\T}$} shifted by $h_{c,\nu}(\T)$.

Now let $U$ be a subquotient of $M(\T)$. Then $U$ automatically inherits a grading from the grading of $M(\T)$ by eigenvalues of $\brh$, and the definition of $\brh$ implies that maps between subquotients of Verma objects preserve this grading.

For simplicity, we will use the natural $\bZ_{+}$-grading on \InnaF{$M(\T) \cong S \hhh^* \otimes X_{\T}$} when we refer to the degree in which a $\InnaE{\underline{\mathrm{Rep}}}(S_{\nu})$-object lies in $U$.

%
%
%

\begin{prop}\label{functoriality_prop_Verma_obj}
 Let $U \in \InnaE{\underline{\co}}_{\text{  } c,\nu}$. Let $X_{\mu}$ be \InnaA{an indecomposable} singular $\InnaE{\underline{\mathrm{Rep}}}(S_{\nu})$-subobject in $U$. Then there exists a morphism $M_{c,\nu}(\mu) \longrightarrow U$ in $\InnaE{\underline{\co}}_{\text{  } c,\nu}$, inducing $\id_{X_\mu}$ on $X_\mu$ when regarded as a morphism of $\InnaE{\underline{\mathrm{Rep}}}(S_{\nu})$ \InnaE{\rm{ind}}-objects.
\end{prop}
\begin{proof}
 Recall the definition of a Verma object. $X_{\mu}$ being \InnaA{an indecomposable} singular $\InnaE{\underline{\mathrm{Rep}}}(S_{\nu})$-subobject in $U$ means that $\Hom_{\InnaF{\underline{\mathrm{Rep}}}(B_{\nu})}(X_{\mu}, Res(U)) \neq 0$, where $X_{\mu}$ is considered as a $\InnaF{\underline{\mathrm{Rep}}}(B_{\nu})$-object with $y_{X_{\mu}} =0$, and there exists a non-zero map in $\Hom_{\InnaF{\underline{\mathrm{Rep}}}(B_{\nu})}(X_{\mu}, Res(U))$ whose image is the chosen copy of $X_{\mu}$ in $U$. Then by definition of a Verma object, we have: $\Hom_{\InnaF{\underline{\mathrm{Rep}}}(B_{\nu})}(X_{\mu}, Res(U)) = \Hom_{\InnaE{\underline{\co}}_{\text{  } c,\nu}}(M_{c,\nu}(\mu),U) \neq 0$ and there exists a non-zero map in $\Hom_{\InnaE{\underline{\co}}_{\text{  } c,\nu}}(M_{c,\nu}(\mu),U)$ inducing $\id_{X_\mu}$ on $X_\mu$ when regarded as a morphism of $\InnaE{\underline{\mathrm{Rep}}}(S_{\nu})$ \InnaE{\rm{ind}}-objects.
\end{proof}



\begin{prop}
Assume $\nu \notin \bZ_+$. Then 
\begin{itemize}
 \item Each Verma object $M_{c,\nu}(\T)$ has a maximal proper $\InnaE{\underline{\co}}_{\text{  } c,\nu}$-subobject $J$, i.e. it has a unique simple quotient $L_{c,\nu}(\T)$ in $\InnaE{\underline{\co}}_{\text{  } c,\nu}$. 
\item The simple objects of the category $\InnaE{\underline{\co}}_{\text{  } c,\nu}$ are exactly $L_{c,\nu}(\T)$.
\end{itemize}

\end{prop}

\begin{proof}
\InnaF{Consider the natural $\bZ_+$-grading on $M_{c,\nu}(\T)$ (as an \InnaE{\rm{ind}}-object of $\InnaE{\underline{\mathrm{Rep}}}(S_{\nu})$). Recall that by Proposition \ref{prop:action_brh_Verma}, $\brh$ acts on grade $m$ of $M_{c,\nu}(\T)$ by $(h_{c,\nu}(\T) + m)\id$.}
\begin{itemize}
 \item Let $J$ be the sum of all the proper $\InnaE{\underline{\co}}_{\text{  } c,\nu}$-subobjects of $M_{c,\nu}(\T)$.
For any proper $\InnaE{\underline{\co}}_{\text{  } c,\nu}$-subobject $N$ of $M_{c,\nu}(\T)$, $N$ inherits the \InnaF{$\bZ_+$-}grading from $M_{c,\nu}(\T)$, with $\brh$ acting on grade $m$ of $N$ by $(h_{c,\nu}(\T) + m)\id$. Since $N$ is a proper \InnaF{$S\hhh^*$-submodule}, we have $m>0$. So $J$ can be presented as a direct sum of $\InnaE{\underline{\mathrm{Rep}}}(S_{\nu})$-objects on which the restriction of $\brh$ acts by \InnaF{diagonally, with eigenvalues whose real part is greater or equal to $h_{c,\nu}(\T) +1$}. This proves that $J$ is a proper $\InnaE{\underline{\co}}_{\text{  } c,\nu}$-subobject of $M_{c,\nu}(\T)$.

\item Let $L$ be a simple object of $\InnaE{\underline{\co}}_{\text{  } c,\nu}$. Then $L$ has a simple singular $\InnaE{\underline{\mathrm{Rep}}}(S_{\nu})$-subobject $X_{\T}$. By Proposition \ref{functoriality_prop_Verma_obj}, there exists a non-zero map $M(\T) \longrightarrow L$, with induced map on $X_{\T}$ being $\id_{\T}$. But $L$ is simple, so $M(\T) \longrightarrow L$ is surjective.
\end{itemize}

\end{proof}

\begin{remark}

Let $M_{c,\nu}(\T)$ be a Verma object in $\InnaE{\underline{\co}}_{\text{  } c,\nu}$. By the propositions above, to check its reducibility, we only need to check whether there are any non-zero morphisms from other Verma objects to $M_{c,\nu}(\T)$.
\end{remark}

\section{\texorpdfstring{Functor $\InnaE{\underline{\co}}_{\text{  } c,\nu=n} \rightarrow \co(H_c(n))$}{Specialization functors}}\label{sec:functor_int_case}
\InnaA{
In this section we construct a functor $\mathcal{F}_{c,n}: \InnaE{\underline{\mathrm{Rep}}}(H_c(\nu=n)) \longrightarrow \InnaE{\mathrm{Rep}}(H_c(n))$ restricting to $\mathcal{F}_{c,n}: \InnaE{\underline{\co}}_{\text{  } c,\nu=n} \longrightarrow \co(H_c(n))$ for any $c \in \bC, n\in \bZ_+$. This functor is analogous to (and based on) the functor $\mathcal{F}_n: \InnaE{\underline{\mathrm{Rep}}}(S_{\nu=n}) \longrightarrow \InnaE{\mathrm{Rep}}(S_{n})$ which induces the equivalence $\InnaE{\underline{\mathrm{Rep}}}(S_{\nu=n})/\idealI_{\nu=n} \cong \InnaE{\mathrm{Rep}}(S_n)$ discussed in Section \ref{sec:Del_cat}.

The functor $\mathcal{F}_{c,n}$ is constructed as follows: 
let $(M, x, y) \in \InnaE{\underline{\mathrm{Rep}}}(H_c(\nu))$ (see Section \ref{sec:Rep_H_c_nu}), and consider $\mathcal{F}_n (M) \in \InnaE{\mathrm{ind}}-\InnaE{\mathrm{Rep}}(S_n)$. The maps $\mathcal{F}_n (x), \mathcal{F}_n (y)$ then make $\mathcal{F}_n (M)$ a module over the algebra $\bC S_n \ltimes T(\InnaE{\hhh \oplus \hhh^*})$, and the conditions on the maps $x,y$ given in Section \ref{sec:Rep_H_c_nu} imply that this action factors through $H_c(n)$, thus making $\mathcal{F}_n (M)$ an $H_c(n)$-module. 

Recall that a morphism $\phi: (M_1, x_1, y_1) \longrightarrow (M_2, x_2, y_2) $ in $\InnaE{\underline{\mathrm{Rep}}}(H_c(\nu))$ is by definition a morphism $\phi: M_1 \longrightarrow M_2$ of $\InnaE{\underline{\mathrm{Rep}}}(S_{\nu=n})$ \InnaE{\rm{ind}}-objects such that $x_2 \circ (\id_{\fh^*} \otimes \phi) = \phi \circ x_1$ and $y_2 \circ (\id_{\fh} \otimes \phi) = \phi \circ y_1$.
So $\mathcal{F}_n (\phi)$ is automatically a morphism of $H_c(n)$-modules.

Thus we define $$\mathcal{F}_{c,n}(M, x,y) : = \mathcal{F}_n (M) \text{ with action of } H_c(n) \text{ given by } \mathcal{F}_n (x), \mathcal{F}_n (y) $$
$$ \mathcal{F}_{c,n}(\phi):= \mathcal{F}_n ( \phi)$$

To see that $\mathcal{F}_{c,n}$ restricts to $ \InnaE{\underline{\co}}_{\text{  } c,\nu=n} \longrightarrow \co(H_c(n))$, note that $\mathcal{F}_n(S \fh) = S \fh, \mathcal{F}_n(S \fh^*) = S \fh^*$. The definitions of $ \InnaE{\underline{\co}}_{\text{  } c,\nu=n}$, $\co(H_c(n))$ then imply that $\mathcal{F}_{c,n}( \InnaE{\underline{\co}}_{\text{  } c,\nu=n}) \subset \co(H_c(n))$. In particular, if $\abs{\lambda}+\lambda_1 \leq n$ then the Verma object of lowest weight $\lambda$ goes to the Verma module of lowest weight $\tilde{\lambda}(n)$ .

Finally, a lemma which will be useful to us later on:

\begin{lemma}\label{lem:Verma_int_dim_hom}
 Fix Young diagrams $\mu, \T$, and let $n>>0$ (in fact, it is enough that $n >2\max(\abs{\mu}, \abs{\T}) +1$). Then $$\mathcal{F}_{c,n}: \Hom_{\InnaE{\underline{\co}}_{\text{  } c,\nu=n}}(M_{c, \nu=n}(\mu), M_{c, \nu=n}(\T)) \longrightarrow \Hom_{\co(H_c(n))}(M_{c, n}(\tilde{\mu}(n)), M_{c, n}(\tilde{\T}(n)))$$ is a bijection.
\end{lemma}
\begin{proof}
Recall that we have:
\begin{align*}
&\Hom_{\InnaE{\underline{\co}}_{\text{  } c,\nu=n}}(M_{c, \nu=n}(\mu), M_{c, \nu=n}(\T)) \cong \\
&\cong \{ \phi \in \Hom_{\InnaE{\underline{\mathrm{Rep}}}(S_{\nu=n})}(X_{\mu}, M_{c, \nu=n}(\T)): y_{M_{c, \nu=n}(\T)} \lvert_{\hhh \otimes \phi(X_{\mu})} =0 \}\end{align*}
and
\begin{align*}
&\Hom_{\co(\overline{H}_c(n))}(M_{c, n}(\tilde{\mu}(n)), M_{c, n}(\tilde{\T}(n)))  \cong \\
&\cong \{ \psi \in \Hom_{\InnaE{\mathrm{Rep}}(S_n)}(\tilde{\mu}(n), M_{c, n}(\tilde{\T}(n))): \mathcal{F}_{c,n} \left( y_{M_{c, \nu=n}(\T)} \right) \lvert_{\hhh \otimes \psi(\tilde{\mu}(n))} =0 \}
\end{align*}

It remains to check that 
\begin{align*}
&\mathcal{F}_{n}: \{ \phi \in \Hom_{\InnaE{\underline{\mathrm{Rep}}}(S_{\nu=n})}(X_{\mu}, M_{c, \nu=n}(\T)):  y_{M_{c, \nu=n}(\T)} \lvert_{\hhh \otimes \phi(X_{\mu})} =0 \}  \longrightarrow \\
&\longrightarrow \{ \psi \in \Hom_{\InnaE{\mathrm{Rep}}(S_n)}(\tilde{\mu}(n), M_{c, n}(\tilde{\T}(n))):  \mathcal{F}_{c,n} \left( y_{M_{c, \nu=n}(\T)} \right) \lvert_{\hhh \otimes \psi(\tilde{\mu}(n))} =0 \}
\end{align*}

is an isomorphism. 
Let $\psi \in \Hom_{\InnaE{\mathrm{Rep}}(S_n)}(\tilde{\mu}(n), M_{c, n}(\tilde{\T}(n)))$ such that $ \mathcal{F}_{c,n} \left( y_{M_{c, \nu=n}(\T)} \right) \lvert_{\hhh \otimes \psi(\tilde{\mu}(n))} =0$.

Since $\mathcal{F}_n$ is surjective on $\Hom$-spaces, there is a $\phi \in \Hom_{\InnaE{\underline{\mathrm{Rep}}}(S_{\nu=n})}(X_{\mu}, M_{c, \nu=n}(\T))$ such that $\mathcal{F}_{n}(\phi) = \psi$. We need to check that there exists a unique such $\phi$ so that $y_{M_{c, \nu=n}(\T)} \lvert_{\hhh \otimes \phi(X_{\mu})} =0$.

Let $m \in \bZ_+$ be such that $\psi(\tilde{\mu}(n)) \subset S^m \InnaD{\fh^*_0} \otimes \tilde{\T}(n) \subset M_{c, n}(\tilde{\T}(n))$. Recall that $$ \mathcal{F}_{n}(y\lvert_{\hhh \otimes \phi(X_{\mu})})=y_n \lvert_{\hhh \otimes \tilde{\mu}(n)} = 0$$ so $y \lvert_{\hhh \otimes \phi(X_{\mu})}:\hhh \otimes X_{\mu} \longrightarrow S^m \InnaD{\fh^*_0} \otimes X_{\T}$ is a negligible morphism in $\InnaE{\underline{\mathrm{Rep}}}(S_{\nu=n})$. But since $n$ is assumed to be very large, $\hhh \otimes \phi(X_{\mu}), S^m \InnaD{\fh^*_0} \otimes X_{\T} \in \InnaE{\underline{\mathrm{Rep}}}(S_{\nu=n})^{(\nu/2)}$, which is a semisimple subcategory of $\InnaE{\underline{\mathrm{Rep}}}(S_{\nu=n})$ (see Section \ref{sec:Del_cat}). So a negligible morphism in this subcategory is zero, and thus $y \lvert_{\hhh \otimes \phi(X_{\mu})}=0$.

Similarly, the fact that $X_{\mu}, S^m \InnaD{\fh^*_0} \otimes X_{\T} \in \InnaE{\underline{\mathrm{Rep}}}(S_{\nu=n})^{(\nu/2)}$ implies that such a morphism $\phi$ is unique.

\end{proof}
}

%

\newpage
\section{Morphisms between Verma objects: general remarks and necessary conditions}\label{sec:Morph_Verma_necessary_cond}
Let $\nu \notin \bZ_+$.
\subsection{Necessary conditions}
The following formula will serve as one of the main tools of this paper.

\begin{prop}\label{main_for_prop}
Let $\mu$ be a partition and $m > 0$ an integer, and assume there is a non-zero morphism $\InnaF{\phi:} M_{c,\nu}(\mu) \longrightarrow M_{c,\nu}(\T)$ such that $\InnaF{\phi(X_{\mu})}$ (image of the lowest weight of $M(\mu)$) sits in degree $m$ of $M_{c,\nu}(\T)$. Then 
\begin{equation}\label{main_for1}
c\left( \frac{\abs{\T}^2-\abs{\mu}^2 -(\abs{\T} - \abs{\mu} ) }{2} +ct(\T) - ct(\mu) \right) = m + (\abs{\T} - \abs{\mu})c\nu 
\end{equation}
\end{prop}

\begin{proof}
\InnaF{Consider the endomorphism $\brh \in \End(\InnaE{\mathrm{Forget}}_{\InnaE{\underline{\co}}_{\text{  } c,\nu}})$ acting on Verma objects} $M_{c,\nu}(\mu) , M_{c,\nu}(\T)$.

\InnaF{By Proposition \ref{prop:action_brh_Verma}, $\brh_{M_{c,\nu}(\mu)}$ acts on $X_{\mu} \subset M_{c,\nu}(\mu)$ with generalized eigenvalue $h_{c,\nu}(\mu)$, while $\brh_{M_{c,\nu}(\T)}$ acts on $\phi(X_{\mu}) \subset S^m \hhh^* \otimes X_{\T}$ with generalized eigenvalue $ h_{c,\nu}(\T) + m$. We conclude that }
$$ h_{c,\nu}(\T) + m = h_{c,\nu}(\mu)$$
\InnaF{That is, }
\begin{align*}
 &\frac{\nu-1}{2} - c \cdot \left(\frac{(\nu - \abs{\T})(\nu - \abs{\T} -1)}{2} -\abs{\T} + ct(\T)\right) + m = \\
&= \frac{\nu-1}{2} - c \cdot \left( \frac{(\nu - \abs{\mu})(\nu - \abs{\mu} -1)}{2} -\abs{\mu} + ct(\mu)\right) 
\end{align*}

which can be rewritten as 
$$ c \cdot \left( \frac{\abs{\T}^2-\abs{\mu}^2- (\abs{\T} - \abs{\mu})}{2} - (\abs{\T} - \abs{\mu})\nu + ct(\T) - ct(\mu) \right) = m $$

\end{proof}

\begin{remark}
 Note that if $c=0$, Equation \eqref{main_for1} implies that $m=0$, $\mu = \T$ and the morphism $M(\mu) \longrightarrow M(\T)$ is the identity map. This means that for $c=0$, all the Verma objects are simple. The category $\co(H_0(\nu))$ is a continuation of the categories of modules over $\bC S_n \ltimes \InnaB{\mathbb{A}}_n$, which are $\mathcal{O}$-coherent $D$-modules and whose Fourier transform has support $\{0\}$. Here $\InnaB{\mathbb{A}}_n$ is the $n$-th Weyl algebra (the algebra of differential operators on $\bC^n$). 
\end{remark}

From now on, we will assume that $c \neq 0$ and denote:
\begin{notn}
  $c':= \frac{1}{c}$.
\end{notn}

\begin{remark}
 In this notation, Equation \eqref{main_for1} can be rewritten as
\begin{equation}\label{main_for}
 \frac{\abs{\T}^2-\abs{\mu}^2 -(\abs{\T} - \abs{\mu} ) }{2} +ct(\T) - ct(\mu) = mc' + (\abs{\T} - \abs{\mu})\nu 
\end{equation}
\end{remark}

\begin{notation}\label{L_tau_mu_m_notn}
Let $\T, \mu$ be Young diagrams, and $m$ be a positive integer.
 Denote by $\mathcal{L}_{\T, \mu, m}$ the set of points $(c', \nu)$ in $\bC^2$ satisfying the Equation \eqref{main_for}. 
\end{notation}

The above proposition shows (with the notations as in Notation \ref{b_tau_mu_notn}): $\mathcal{B}_{\T, \mu} \subset \biguplus_{m \in \bZ_{>0}} \mathcal{L}_{\T, \mu, m}$.

\InnaB{
\begin{example}
$\mathcal{L}_{\T, \T, 0} = \bC^2$ and $\mathcal{L}_{\T, \T, m} = \emptyset$ for any Young diagram $\T$ and any $m \in \bZ_{>0}$.
\end{example}}

\InnaA{\begin{notation}\label{B_tau_mu_m_notn}
Let $\T, \mu$ be Young diagrams, and $m$ be a positive integer.
 Denote: $\mathcal{B}_{\T, \mu, m}:= \mathcal{L}_{\T, \mu, m} \InnaB{\cap} \mathcal{B}_{\T, \mu}$.
\end{notation}}
\begin{remark}
 Note that Proposition \ref{main_for_prop} implies the following statement:

Fix $c, \nu$, and consider the lowest weight $X_{\mu}$ of the Verma object $M_{c,\nu}(\mu)$. Then for all non-trivial maps $M_{c,\nu}(\mu) \longrightarrow M_{c,\nu}(\T)$ the degree $m$ of $M_{c,\nu}(\T)$ in which $X_{\mu}$ sits is the same, since it is given by Equation \eqref{main_for}.
\end{remark}

\InnaB{As a special case of Proposition \ref{main_for_prop}}, we have the following lemma (for its statement and proof for the \InnaB{classical Cherednik algebra}, see \cite[Lemma 3.5]{ES}):

\begin{lemma}\label{necessary_and_suff_cond_in_deg_1}
 Let $X_{\mu} \subset \InnaD{\fh^*_0} \otimes X_{\T}$ in $\InnaE{\underline{\mathrm{Rep}}}(S_{\nu})$. We can regard $X_{\mu}$ as sitting in degree $1$ of the $\InnaE{\underline{\co}}_{\text{  } c,\nu}$-object $M_{c,\nu}(\T)$. Then there is a morphism $M_{c,\nu}(\mu) \longrightarrow M_{c,\nu}(\T)$ in $\InnaE{\underline{\co}}_{\text{  } c,\nu}$, inducing $\id_{X_{\mu}}$ on $X_{\mu}$, if and only if 
$$
  \frac{\abs{\T}^2-\abs{\mu}^2 -(\abs{\T} - \abs{\mu} ) }{2} +ct(\T) - ct(\mu)  = c' + (\abs{\T} - \abs{\mu})\nu 
$$

\end{lemma}

\begin{proof}
 \InnaE{Consider the morphism $y_{M_{c,\nu}(\T)} \mid_{\InnaD{\fh^*_0} \otimes X_{\T}}: \hhh \otimes \InnaD{\fh^*_0} \otimes X_{\T} \longrightarrow X_{\T}$. By the commutation relations in Section \ref{sec:Rep_H_c_nu}, we have an equality between the following morphisms $\hhh \otimes \InnaD{\fh^*_0} \otimes X_{\T} \longrightarrow X_{\T}$ in $\underline{Rep}(S_{\nu})$:
 \begin{align*}
&y_{M_{c,\nu}(\T)} \mid_{\InnaD{\fh^*_0} \otimes X_{\T}}\circ \left(\id_{\hhh} \otimes x_{M_{c,\nu}(\T)} \rvert_{X_{\T}} \right)- x_{M_{c,\nu}(\T)} \circ \left(\id_{\hhh^*} \otimes  y_{M_{c,\nu}(\T)} \rvert_{X_{\T}} \right)\circ \left( c_{\hhh, \hhh^*} \otimes \id_{X_{\T}} \right) = \\
&= (ev_{\hhh, \hhh^*} \otimes_{X_{\T}})\circ  (\id-c\Omega^3 - c\Omega^{23})  
 \end{align*}

 Furthermore, the definition of $x_{M_{c,\nu}(\T)}, y_{M_{c,\nu}(\T)}$ tells us that the right hand side of the above equality is just $y_{M_{c,\nu}(\T)} \mid_{\InnaD{\fh^*_0} \otimes X_{\T}}$, so 
 $$y_{M_{c,\nu}(\T)} \mid_{\InnaD{\fh^*_0} \otimes X_{\T}} = (ev_{\hhh, \hhh^*} \otimes_{X_{\T}})\circ  (\id-c\Omega^3 - c\Omega^{23})$$
 
 Denote by $Y$ the endomorphism the isomorphism of $\InnaD{\fh^*_0} \otimes X_{\T}$ corresponding to the morphism $y_{M_{c,\nu}(\T)} \mid_{\InnaD{\fh^*_0} \otimes X_{\T}}$ under the correspondence $$ \Hom_{Rep(S_{\nu})}(\hhh \otimes \InnaD{\fh^*_0} \otimes X_{\T}, X_{\T}) \cong \End_{Rep(S_{\nu})}(\InnaD{\fh^*_0} \otimes X_{\T})$$ We immediately see that
 $$Y= \id-c\Omega^2 - c\Omega^{12}$$ (note that in the classical category $\co(H_c(n))$, we would have: $Y = 1- c \sum_{s \in \cS} 1\otimes s + c \sum_{s \in \cS} s\otimes s$).
 
 \mbox{}
 
 By definition, an indecomposable object $X_{\mu} \subset \hhh^* \otimes X_{\T}$ is singular in $M_{c,\nu}(\T)$ iff $y_{M_{c,\nu}(\T)} \rvert_{X_{\mu}}=0$, which is equivalent to requiring that $Y \rvert_{X_{\mu}}=0$.
 


As it was said in Subsection \ref{ssec:Omega_elem}, the endomorphism $c\Omega^2 = c \id_{\hhh^*} \otimes \Omega$ has a unique generalized eigenvalue $$ c \cdot C_{\nu}(\T)= h_{c,\nu}(\T)-\frac{\nu-1}{2}$$ on $\InnaD{\fh^*_0} \otimes X_{\T}$ (i.e. $c\Omega^2 - c \cdot C_{\nu}(\T)$ is a nilpotent endomorphism of $\InnaD{\fh^*_0} \otimes X_{\T}$). Similarly, the restriction of the endomorphism $c\Omega^{12} \in End_{Rep(S_{\nu})}(\InnaD{\fh^*_0} \otimes X_{\T})$ to $X_{\mu}$ is just $\Omega_{X_{\mu}}$, and has a unique generalized eigenvalue $$c \cdot C_{\nu}(\mu)=h_{c,\nu}(\mu)-\frac{\nu-1}{2}$$

Thus $Y$ acts on \InnaF{an indecomposable} object $X_{\mu} \subset \hhh^* \otimes X_{\T}$ with the generalized eigenvalue 

$$1 +h_{c,\nu}(\T) -\frac{\nu-1}{2} - \left(h_{c,\nu}(\mu)-\frac{\nu-1}{2}\right)=1 +h_{c,\nu}(\T) - h_{c,\nu}(\mu)$$ 

This proves that $X_{\mu}$ is an indecomposable singular $\InnaE{\underline{\mathrm{Rep}}}(S_{\nu})$-subobject of $M(\T)$ iff 

$1 +h_{c,\nu}(\T) - h_{c,\nu}(\mu)=0$, which is equivalent to \InnaF{the requirement} $$
  \frac{\abs{\T}^2-\abs{\mu}^2 -(\abs{\T} - \abs{\mu} ) }{2} +ct(\T) - ct(\mu)  = c' + (\abs{\T} - \abs{\mu})\nu 
$$}
\end{proof}

An additional condition for $X_{\mu}$ to sit in degree $m$ of $M_{c,\nu}(\T)$ arises from Pieri's rule (cf. Proposition \ref{Pieri}). 

 In terms of Notations \ref{L_tau_mu_m_notn}, \ref{B_tau_mu_m_notn}, Lemma \ref{necessary_and_suff_cond_in_deg_1} means that $\mathcal{B}_{\T, \mu, 1} = \mathcal{L}_{\T, \mu, 1}$ whenever $\mu$ is obtained from $\T$ as in Pieri's rule, i.e. by adding/moving/deleting a cell, or if $\mu =\T$.

\subsection{Remarks on Equation \eqref{main_for}}\label{ssec:rmrks_main_for}

For a Young diagram $\mu$, denote: $$f(\mu):= \frac{\abs{\mu}^2- \abs{\mu} }{2} +ct(\mu)$$ Then Equation \eqref{main_for} can be rewritten as 
\begin{equation}\label{main_for_rewritten} 
f(\T) - f(\mu)  = c'm + (\abs{\T} - \abs{\mu})\nu 
\end{equation}

\begin{lemma}\label{lemma_values_f}
 Let $\mu$ be a Young diagram. Then $f(\mu)$ is a non-negative integer less or equal to $\abs{\mu}^2 - \abs{\mu}$, and it is equal to zero if $\mu$ is a column diagram (i.e. $l(\mu\check{}) \leq 1$), and to $\abs{\mu}^2 - \abs{\mu}$ if $\mu$ is a row diagram (i.e. $l(\mu) \leq 1$).
\end{lemma}

\begin{proof}
 This is equivalent to saying that $\abs{ct(\mu)} \leq \frac{\abs{\mu}^2 - \abs{\mu}}{2} $. The latter can be proved by induction on the number of columns of $\mu$:
 
Base: Assume the number of columns of $\mu$ is zero, i.e. $\mu = \emptyset$. Then  $\abs{ct(\mu)} = 0 = \frac{\abs{\mu}^2 - \abs{\mu}}{2} $. Also, if $\mu$ is a column diagram, then $ct(\mu) = -\frac{\abs{\mu}^2 - \abs{\mu}}{2} $.

Step: Denote by $k$ the number of columns of $\mu$ ($k >1$), by $\mu'$ the diagram $\mu$ without the last column, and by $l$ the number of boxes in the last column of $\mu$. By induction assumption, $\abs{ct(\mu')} \leq \frac{\abs{\mu'}^2 - \abs{\mu'}}{2} $. Next, $ct(\mu) =ct(\mu') + (k-1)\cdot l - \ell(l-1)/2$, and so $\frac{\abs{\mu}^2 - \abs{\mu}}{2} = \frac{\abs{\mu'}^2 - \abs{\mu'}}{2} +\abs{\mu'}l +\frac{l^2 - l}{2} \geq \abs{ct(\mu')} + \frac{l^2 - l}{2} +\abs{\mu'}l \geq \abs{ct(\mu') + (k-1)\cdot l - \ell(l-1)/2} = \abs{ct(\mu)} $ (for the last inequality, note that by definition, $\abs{\mu'} \geq k-1$, with equality if and only if $ \mu$ is a row diagram).
For a row diagram $\mu$ (with $k$ cells), there is an equality $ct(\mu) = \frac{\abs{\mu}^2 - \abs{\mu}}{2} $, and in general, $\abs{ct(\mu)} \InnaA{<} \frac{\abs{\mu}^2 - \abs{\mu}}{2} $ for $\mu$ having $k >1$ columns and not a row diagram.

\end{proof}

For the right hand side of Equation \eqref{main_for_rewritten}, note that we have, by Pieri's rule (Proposition \ref{Pieri}): $m \geq \abs{\abs{\mu}-\abs{\T}}$.

\section{Blocks of the category \texorpdfstring{$\co(H_c(n))$}{O} (classical case)}\label{sec:blocks_classical_case}

\subsection{KZ functor and connection to the representations of the \InnaE{Hecke algebras}}

A powerful tool in studying the category $\co$ for the Cherednik algebra is the KZ functor (see e.g. \cite[2.8.2]{GS}, \cite[Section 6]{EM}, \cite[Section 3.3]{BEG}). The KZ functor is a functor $$\co(H_c(n)) \longrightarrow \InnaE{\mathrm{Rep}}(\mathcal{H}_q(n))$$ $\InnaE{\mathrm{Rep}}(\mathcal{H}_q(n))$ being the category of representations of the Hecke algebra \InnaE{$\mathcal{H}_q(n)$} of type A, where $q = exp(2\pi i c)$.

It turns out that this functor is essentially surjective on objects, surjective on Homs and exact (see \cite{GGOR}).
This functor induces an equivalence of categories $\overline{KZ}: \co/\co^{tor} \longrightarrow \InnaE{\mathrm{Rep}}(\mathcal{H}_q(n))$, where $\co^{tor}$ is \InnaE{the} full subcategory of the category $\co$ whose objects are modules which, when considered as $\bC[\hhh]$-modules, have Krull dimension less than $n-1$ (see \cite[Section 6.3]{EM}). The non-zero objects of $\co/\co^{tor}$ would thus be $H_c(n)$-modules which are supported on the whole $\bC^n$.

Moreover, by \cite[Proposition 2.10]{BEG1} (also proved in \cite{GGOR}), the KZ functor is faithful on the full subcategory of Verma modules in $\co(H_c(n))$.

We will use the fact that the $KZ$ functor takes $M(\lambda)$ to $S_{\lambda}\check{}$ (the dual of the Specht module) if $c \geq 0$, and to $S_{\lambda}$ (Specht module) if $c < 0$. For $c<0$, the module $L(\lambda)$ goes to $D_{\lambda}$ (which is either the unique simple quotient of the Specht module $S_{\lambda}$, if it exists, or zero).

\subsection{Representations of \InnaE{Hecke algebras} of type A}

We will use the following facts about the representations of \InnaE{Hecke algebras} of type A (see \cite{Mat}, \cite[Section 3]{BEG}):

Some definitions:
\begin{notation}
 Denote $e := ord_{\bC}(q)$ (i.e. $e$ is the order of $q$ if $q$ is a root of unity and $\infty$ otherwise), and denote $s := n-e$.
\end{notation}

First of all, we have:

\begin{theorem}\label{Hecke_alg_rep_semisimple}
 If $e > n$ (in particular, if $e = \infty$), then the category $\InnaE{\mathrm{Rep}}(\mathcal{H}_q(n))$ is semisimple.
\end{theorem}

From now on, we will assume $e \leq n$, and thus $0 \leq s \leq n$.

\begin{definition}[$e$-hook]\label{e_hook}
\mbox{}\begin{itemize}
 
\item
Hooks are parameterized by cells $(i, j)$ in the diagram; a hook corresponding to $(i,j)$ consists of all the cells $(i', j')$ such that either $i' = i, j' \geq j$, or $i' \geq i, j'=j$. The cell $(i,j)$ is the ``vertex'' of the hook corresponding to it. Cells $(i, j'), j'>j$ are called the ``arm" of the hook, while cells $(i', j), i'>i$ are called the ``leg" of the hook.
\item An $e$-hook of a Young diagram is a hook of length $e$ (i.e. contains exactly $e$ cells).
\end{itemize}

\begin{example}
    Let $\lambda = (5,4,2,2)$. The hook corresponding to $(1,2)$ is a $7$-hook with arm $3$ and leg $3$:
$$ \young(\hfil\circ\circ\circ\circ,\hfil\circ\hfil\hfil,\hfil\circ,\hfil\circ) $$ 
\end{example}

\end{definition}

\begin{lemma}
 For $n$ large relatively to $s$ (namely, $n>2s, n>1$), a Young diagram $\T$ with $\abs{\T}=n$ can have at most
one $e$-hook.
\end{lemma}

\begin{proof}
\InnaC{Given two distinct hooks of same length, their intersection can contain at most one cell. Now, fix a Young diagram $\T$ with $\abs{\T}=n$. If $\T$ has two distinct $e$-hooks, \InnaE{we obtain} the inequality} $2e - 1 \leq \abs{\T}=n$, and thus $n \leq 2s +1$. Note that, in fact, a Young diagram cannot consist of two hooks of equal size intersecting each other in one cell unless this diagram contains just one cell. So if $n \geq 2$, then $2e - 1 < \abs{\T}=n$, i.e. $n < 2s+1$. So for $n>2s, n>1$, there is at most one $e$-hook in $\T$.
\end{proof}

\begin{definition}[Core]\label{int_s_core}
 Let $s \InnaC{\leq n}$ be a non-negative integer. The $(n-s)$-core of a Young diagram $\T$ is the Young diagram obtained by performing the following procedure on $\lambda:=\T$:

\begin{enumerate}
 \item Take the Young diagram $\lambda$.
\item If $\lambda$ has no $(n-s)$-hook, stop. The diagram $\lambda $ is then the $(n-s)$-core of $\T$.
\item Otherwise, remove an $(n-s)$-hook from $\lambda$ and move the boxes underneath this $(n-s)$-hook one position up and one position to the left. Denote the obtained Young diagram by $\lambda^{(1)}$.
\item Repeat the procedure for $\lambda := \lambda^{(1)}$.
\end{enumerate}
\end{definition}

So for $n>\max(2s, 1)$, either $\mathtt{core}_{(n-s)}(\T)=\T$, or $\mathtt{core}_{(n-s)}(\T)$ has $s$ boxes.

\begin{example}
    Let $\lambda = (5,4,2,2)$. Its $7$-core is obtained by removing the hook corresponding to $(1,2)$ and moving the boxes below it one space up and left, i.e. $\mathtt{core}_{7}(\lambda) = (3,1,1,1)$
$$ \lambda=\young(\hfil\circ\circ\circ\circ,\hfil\circ\hfil\hfil,\hfil\circ,\hfil\circ)  \mapsto \mathtt{core}_{7}(\lambda) =\yng(3,1,1,1)$$ 
 
\end{example}

Now, assume $n>\max(2s, 1)$.
Classification of partitions of size $n$ with given core $\lambda$ such that $\abs{\lambda}=s$ is given by the following proposition:

\begin{proposition}\label{insert_hook_prop}
\InnaC{Let $\lambda$ be a Young diagram of size $s$.} Let $hook(l,(n-s))$ ($0 \leq l \leq (n-s-1)$) be the $(n-s)$-hook with leg $l$:
$((n-s)-l,1,1,...,1)$. Then there is a unique Young diagram $\mathtt{rec}(l,\lambda)$ of size $n$ such that $\lambda$
is obtained by deleting the hook $hook(l,n-s)$ from $\mathtt{rec}(l,\lambda)$ \InnaC{(thus $\lambda$ is the $(n-s)$-core of $\mathtt{rec}(l,\lambda)$)}. 
\end{proposition}

\begin{proof}
Basically, $\mathtt{rec}(l,\lambda)$ is obtained by "inserting" a $hook(l,(n-s))$ to $\lambda$ in a certain
not completely trivial way.
 To insert a hook $hook(l,(n-s))$ into a diagram $\lambda$, one should put its vertex in the position $(i,j)$ such that the following equations are satisfied:
\begin{align}\label{prop_hook_insert_ineq}
 &\lambda_i +1 \leq j-1 +n-s -l \leq \lambda_{i-1}
 &\lambda\check{}_j +1 \leq i+l \leq \lambda\check{}_{j-1}
\end{align}
($\lambda_0, \lambda\check{}_0 := \infty$).

We now explain how to find such $i,j$.

First, we show that if the inequalities \eqref{prop_hook_insert_ineq} hold, then $i=1$ or $j=1$. Indeed, adding the two inequalities \eqref{prop_hook_insert_ineq}, \InnaE{we obtain}: $$ j-1 +n-s -l +i+l = i+j-1+n-s  \leq \lambda_{i-1} + \lambda\check{}_{j-1}$$
But $n-s >s =\abs{\lambda}$, so $i+j-1+n-s \geq i+j +\abs{\lambda} \geq \abs{\lambda} +2 $. Thus \InnaE{we obtain}:  $\abs{\lambda} +2 \leq \lambda_{i-1} + \lambda\check{}_{j-1} $. But this is clearly possible only when $i=1$ or $j=1$.

Next, find $j$ such that $$ \lambda\check{}_j +1 \leq l+1 \leq \lambda\check{}_{j-1}$$
(which is always possible).
We now have two cases:
\begin{enumerate}
 \item If \InnaE{we obtain} $j \geq 2$, then we put $i:=1$ and we only need to check that $ \lambda_1 +1 \leq j-1 +n-s -l $.
But we assumed that $n>2s$, so $n-s > \abs{\lambda}=s$, and thus $$j-1 +n-s -l \geq j+\abs{\lambda} -l \geq 2+ \abs{\lambda} - \lambda\check{}_{j-1} \geq \lambda_1 +1$$
so we found indices $(i,j)$ satisfying the inequalities \eqref{prop_hook_insert_ineq}.
\item If \InnaE{we obtain} $j=1$, then $ \lambda\check{}_j \leq l$. We now find $i$ such that 
$$\lambda_i +1 \leq n-s -l \leq \lambda_{i-1}  $$ (again, this is always possible). To show that we found indices $(i,j)$ satisfying the inequalities \eqref{prop_hook_insert_ineq}, we only need to check that $\lambda\check{}_1 +1 \leq i+l $, which is true since we have $i \geq 1$ and $ \lambda\check{}_j \leq l$.
\end{enumerate}

This shows that there exists exactly one way to insert the hook $hook(l,(n-s))$ into $\lambda$.

\end{proof}

\begin{example}
    Let $\lambda = (5,4,2,2)$. 
\begin{enumerate}
 \item Adding the $7$-hook $(6,1)$ to its $7$-core ($\mathtt{core}_{7}(\lambda) = (3,1,1,1)$), we obtain the Young diagram $\mathtt{rec}(1, \mathtt{core}_{7}(\lambda) ) = (7,4,1,1)$:
\begin{align*}
 \lambda=\young(\hfil\circ\circ\circ\circ,\hfil\circ\hfil\hfil,\hfil\circ,\hfil\circ)  \mapsto \mathtt{core}_{7}(\lambda) =\yng(3,1,1,1) \mapsto \mathtt{rec}(1, \lambda) =
\young(\hfil\circ\circ\circ\circ\circ\circ,\hfil\circ\hfil\hfil,\hfil,\hfil)
\end{align*}
\item Adding the $7$-hook $(2,1,1,1,1,1)$ to $\mathtt{core}_{7}(\lambda)$, we obtain the Young diagram $\mathtt{rec}(5, \mathtt{core}_{7}(\lambda) ) = (3,2,2,2,2,1,1)$:
\begin{align*}
 \lambda=\young(\hfil\circ\circ\circ\circ,\hfil\circ\hfil\hfil,\hfil\circ,\hfil\circ)  \mapsto \mathtt{core}_{7}(\lambda) =\yng(3,1,1,1) \mapsto \mathtt{rec}(5, \lambda) =
\young(\hfil\hfil\hfil,\circ\circ,\circ\hfil,\circ\hfil,\circ\hfil,\circ,\circ)
\end{align*}
\end{enumerate}
\end{example}

\begin{definition}\label{def:e_regular_partition}
Let $e \in \bZ_+$. 
\begin{itemize}
 \item A partition $\lambda$ is called $e$-regular if the multiplicity of each part of $\lambda$ is smaller than $e$.
 \item A partition $\lambda$ is called $e$-restricted if $\lambda_i -\lambda_{i+1} < e$ for all $i$ (i.e. if $\lambda\check{}$ is $e$-regular).
\end{itemize}
\end{definition}

\begin{theorem}[cf. \cite{Mat}]\label{Hecke_alg_rep_main_thrm}
 Let $\T, \mu$ be partitions of $n$.
\InnaA{\begin{enumerate}
 \item 
\begin{itemize}
\item All the simple modules of $\mathcal{H}_q(n)$ are $D_{\lambda}
$, where $\lambda$ is $e$-regular. These are exactly the partitions for which $D_{\lambda}\neq 0$.
 \item The $\mathcal{H}_q(n)$-modules $S_{\T}\check{}$ and $S_{\T}$ have the same composition factors.                                                                         
\item Two modules $S_{\T}$, $S_{\mu}$ belong to the same block iff $\T, \mu$ have the same $(n-s)$-core.
\item If $\T$ has no $(n-s)$-hook, then $S_{\T}$ (and hence $S_{\T}\check{}$) is irreducible.

\end{itemize}
\item Assume $n >2s$. Then we have:
\begin{itemize}
\item Let $\T=\mathtt{rec}(l,\beta)$, $0 \leq l\leq (n-s-2)$ and $\mu\ne \T$.
Then there is a non-trivial morphism $S_{\mu} \to S_{\T}$ iff $\mu=\mathtt{rec}(l+1,\beta)$.
\item If $\T=\mathtt{rec}(l,\beta)$, $\mu=\mathtt{rec}(l+1,\beta)$, then the composition factors of $S_{\T}$ are $D_{\T}, D_{\mu}$, with multiplicity $1$ (only one of them if the other is zero).

\end{itemize}
\end{enumerate}}

\end{theorem}

\subsection{Blocks of the category \texorpdfstring{$\co(H_c(n))$}{O}}
We now give the relevant results for the category $\co(H_c(n))$:

Let $n$ be a non-negative integer. 
\subsubsection{Equivalences}

First of all, we have a useful theorem proved by Rouquier and expanded by Losev (see \cite[Theorem 5.12]{Rou}, \cite{Lo}):

\begin{theorem}[Rouquier]\label{Rouquier}
Let $n >1, r,a> 0, b \neq 0$ be integers, $ \Gcd(r, a)=1$. Then the categories $\co(H_{c_1=b/a}(n)), \co(H_{c_2=(br)/a}(n))$ are equivalent if $c_1 \not\in \{\frac{2k+1}{2}, k\in \bZ\}$, with a correspondence:
\\
morphism $M_{c_1,n}(\mu) \longrightarrow M_{c_1,n}(\T)$ ($\abs{\T} =\abs{\mu} = n$) such that $\mu$ sits in degree $m$ of $M_{c_1,n}(\T)$

corresponds to 

morphism $M_{c_2,n}(\mu) \longrightarrow M_{c_2,n}(\T)$ ($\abs{\T} =\abs{\mu} = n$) such that $\mu$ sits in degree $rm$ of $M_{c_2,n}(\T)$.
\end{theorem}

Next, we have the following simple equivalence of categories: (see \cite[3.1.4]{Rou2}): 

\begin{obsr}\label{equiv_c_and_minus_c}
 The rational Cherednik algebras $H_{-c}(n), H_c(n)$ are isomorphic: $$ H_c(n) \longrightarrow H_{(-c)}(n),x \longmapsto x, y \longmapsto y, \sigma \in S_n \longmapsto (sign(\sigma)\cdot \sigma) \in \bC[S_n]$$

This means that the categories of representations of these algebras are equivalent, with equivalence given by 
$$\co(H_c(n)) \longrightarrow \co(H_{(-c)}(n)), V \longmapsto sign \otimes V,$$
where $sign$ is the sign representation of $S_n$. Note that $sign \otimes \mu \cong \mu\check{}$ for representation $\mu$ of $S_n$. 
\end{obsr}

We also have the following statement (see \cite[Section 6.2]{GGOR}):
\begin{prop}\label{pseudo_anti_equiv_c_and_minus_c}
If $c \not\in \frac{1}{2} + \bZ$, then $$\Hom_{H_c(n)}(M_{c, n}(\mu), M_{c, n}(\tau)) \cong \Hom_{H_{(-c)}(n)}(M_{(-c), n}(\tau), M_{(-c), n}(\mu))$$

\end{prop}

\subsubsection{Results on \texorpdfstring{$\co(H_c(n))$}{category O} obtained from the theory of representations of the Hecke algebra $\mathcal{H}_q(n)$}

The correspondence between the lowest weight representations of $H_c(n)$ and the finite dimensional representations of the Hecke algebra $\mathcal{H}_q(n)$ gives us the following theory (see \cite{EN}):

First of all, we have (see \cite[6.13]{EM}):

\begin{proposition}\label{prop:KZ_simples}
 For $c \leq 0$, $KZ(M_{c, n}(\T)) = S_{\T}, KZ(L_{c, n}(\T)) = D_{\T}$, and thus $KZ(L_{c, n}(\T)) \neq 0$ iff $\T$ is $(n-s)$-regular.

For $c > 0$,  $KZ(L_{c, n}(\T)) \neq 0$ iff $\T$ is $(n-s)$-restricted.
\end{proposition}

We next have the following theorem:

\begin{theorem}[Dipper, James; cf. \cite{DJ}]\label{DJ}

If $c'$ satisfies one of the following conditions:
\begin{itemize}
 \item $c' \not\in \bQ$,
\item $c' \in \bQ$, and $ c' = \frac{d}{a}, \Gcd(a,d)=1, d>n $,
\end{itemize}
then the category $\co(H_c(n))$ is semisimple (in particular, all Verma modules are simple).

\end{theorem}

 \begin{proof}
 This is a direct consequence of Theorem \ref{Hecke_alg_rep_semisimple}, the standard property of Verma modules (Proposition \ref{functoriality_prop_Verma_obj}) and the fact that the KZ functor is faithful on the full subcategory of Verma modules in $\co(H_c(n))$.
\end{proof}

This means that we remain with only one interesting case: $$ c' = \frac{d}{a}, \Gcd(a,d)=1, a>0, 0 \leq d \leq n $$
For $d \leq n/2$, the theory is more complicated, but for $d > n/2$, it is rather simple and explained below.
\InnaA{Note that Theorem \ref{Rouquier} allows us to assume that $a=1$ and study just the case $c' = d \in \{0, 1, ..., n\}$.}

Fix integer $s \geq 0$, and from now on, consider $n > 2s$, $c' = d= n-s$, and a Young diagram $\T$ with $\abs{\T}=n$.

\begin{corollary}\label{int_red}
 $M_{c, n}(\T)$ is simple if and only if $\T$ has no $(n-s)$-hook, or its $(n-s)$-hook is a vertical strip.
\end{corollary}
 
\begin{proof}
 This is a consequence of Theorem \ref{Hecke_alg_rep_main_thrm} and of the standard property of Verma modules (Proposition \ref{functoriality_prop_Verma_obj}).
\end{proof}

\begin{theorem}\label{int_blocks}
\InnaD{ Let $s \geq 0$, $n > 2s$ be integers, and put $c' := n-s$. }
\InnaA{Let $\beta$ be an arbitrary Young diagram of size $s$.} Let $\T=\mathtt{rec}(l,\beta)$, with $0 \leq l \leq (n-s-2)$ and $\mu\neq \T$ be a Young diagram of size $n$.

Then there is a non-trivial morphism $\psi_l: M_{c, n}(\mu)\to M_{c, n}(\T)$ iff $\mu=\mathtt{rec}(l+1,\beta)$. 

In that case, we have:
\begin{itemize}
 \item $\dim \Hom( M_{c, n}(\mu), M_{c, n}(\T)) =1$.
\item There is a short exact sequence $$ 0 \longrightarrow L_{c, n}(\mu) \longrightarrow M_{c, n}(\T) \longrightarrow L_{c, n}(\T) \longrightarrow 0$$
\end{itemize}

\end{theorem}

\begin{proof}
 This is a generalization of the result in \cite[Section 3]{BEG} (there it is proved for the case $s=0$, i.e. $e=n$). 
 
First of all, recall from \cite[Proposition 2.10]{BEG1} (also proved in \cite{GGOR}) that the KZ functor is fully faithful on the full subcategory of Verma modules in $\co(H_c(n))$. Then Theorem \ref{Hecke_alg_rep_main_thrm} implies that the dimension of $\Hom( M_{c, n}(\mu), M_{c, n}(\T)) $ is one iff $\mu=\mathtt{rec}(l+1,\beta)$, and zero otherwise.
 
To prove the rest of the theorem, we only need to check that $L_{c, n}(\mu)$ appears among the composition factors of $M_{c, n}(\T)$ with multiplicity at most one if $\mu=\mathtt{rec}(l+1,\beta)$, and with multiplicity zero otherwise.

If $KZ(L_{c, n}(\mu)) \neq 0$, then the statement follows immediately from Theorem \ref{Hecke_alg_rep_main_thrm}(2). So we only need to show that the same is true when $KZ(L_{c, n}(\mu)) = 0$. Recall from Proposition \ref{prop:KZ_simples} that the latter happens iff $\mu$ is not $(n-s)$-restricted, so we will assume that this is the case.

We now have two possible cases:
\begin{itemize}
 \item The Young diagram $\mu$ does not have an $(n-s)$-hook, and thus $\mathtt{core}_{(n-s)}(\mu)=  \mu$. In this case $M_{c, n}(\mu)$ is simple and lies in a semisimple block (see Corollary \ref{int_red}), so $L_{c, n}(\mu)$ cannot appear as a composition factor in $M_{c, n}(\T)$ unless $\mu = \T$ (which is clearly impossible, since $\mathtt{core}_{(n-s)}(\T)=  \beta$, which has size $s <n$).
 \item The Young diagram $\mu$ has an $(n-s)$-hook. Denote: $\beta' := \mathtt{core}_{(n-s)}(\mu)$. It is then easy to see that the condition that $\mu$ is not $(n-s)$-restricted implies that $\mu$ is obtained from $\beta'$ by adding $(n-s)$ cells to the first row (i.e. $\mu=\mathtt{rec}(0,\beta')$).
 
 Assume $M_{c, n}(\mu), M_{c, n}(\T)$ lie in the same block of $\co(H_c(n))$ (otherwise we are done). Since the KZ functor is fully faithful on the full subcategory of Verma modules in $\co(H_c(n))$, the modules $S_{\mu}\check{} \cong KZ(M_{c, n}(\mu)), S_{\T}\check{} \cong KZ(M_{c, n}(\T))$ belong to the same block of $\InnaE{\mathrm{Rep}}(\mathcal{H}_q(n))$ (here $q = exp(2\pi i c)$). By Theorem \ref{Hecke_alg_rep_main_thrm}(1), this implies that $\beta = \beta'$. 
 
 So it remains to check that for $l>0$ and $\mu=\mathtt{rec}(0,\beta)$, $L_{c, n}(\mu)$ does not appear among the composition factors of $M_{c, n}(\T)$. Indeed, since $\mu=\mathtt{rec}(0,\beta)$, $\T=\mathtt{rec}(l,\beta)$ and $l>0$, we have: $ct(\T) <ct(\mu)$, and thus $$h_{c,n}(\T)= \frac{n-1}{2} - c \cdot ct(\T) > \frac{n-1}{2} - c \cdot ct(\mu) = h_{c,n}(\mu)$$ But $h_{c,n}(\T)$ is the lowest eigenvalue of $\brh$ on $M_{c, n}(\T)$, so $L_{c, n}(\mu)$ cannot be a composition factor of $M_{c, n}(\T)$.
\end{itemize}
\end{proof}
%
%
%

\begin{example}
The Verma module $M_{c, n}(\T^{n \InnaA{-1}})$ (whose lowest weight is the trivial representation $\T^{n\InnaA{-1}}= \widetilde{\emptyset}(n)$ of $S_n$) is reducible if and only if the following condition on $c'=\frac{1}{c},n$ holds:

$$c' = \frac{n-s}{r}, s \geq 0, r \geq 1$$

In this case, $M_{c, n}(\T^{n\InnaA{-1}})$ contains a lowest weight $H_c(n)$-submodule, whose lowest weight $\mu$ lies in degree $r(s+1)$ of $M_{c, n}(\T^{n\InnaA{-1}})$. The Young diagram of $\mu$ is a two row diagram, the lower row containing $s+1$ boxes.
\end{example}

Theorem \ref{int_blocks} gives the following corollary:

\begin{corollary}\label{int_blocks_conclusion}
 Let $n >3, s \geq 0, r \geq 1$ be integers, $ n>2s, \Gcd(r, n-s)=1$. Then the blocks of the category $\co(H_{c=\frac{r}{n-s}}(n))$ correspond to $(n-s)$-cores, and we have two types of blocks: blocks containing only one Verma module (up to isomorphism), corresponding to a diagram which doesn't have an $(n-s)$-hook, and blocks which correspond to Young diagrams of size $s$; in such a block $Block_{\beta}$ ($\abs{\beta} = s$) lie the Verma modules whose $(n-s)$-core is $\beta$, and these Verma modules form a long exact sequence:
\begin{align*}
 &0 \longrightarrow M_{c, n}(\mathtt{rec}((n-s-1), \beta)) \stackrel{\psi_{(n-s-2)}}{\longrightarrow} M_{c, n}(\mathtt{rec}((n-s-2), \beta)) \stackrel{\psi_{(n-s-3)}}{\longrightarrow}  ...\\
&...
 \stackrel{\psi_{1}}{\longrightarrow} M_{c, n}(\mathtt{rec}(1, \beta))\stackrel{\psi_{0}}{\longrightarrow} M_{c, n}(\mathtt{rec}(0, \beta)) \longrightarrow L_{c, n}(\mathtt{rec}(0, \beta)) \longrightarrow 0
\end{align*}

\end{corollary} 

(Note that $\mathtt{rec}((n-s-1), \beta)$ is the diagram $\beta$ with a vertical $(n-s)$-hook added, and $\mathtt{rec}(0, \beta)$ is the diagram $\beta$ with a horizontal $(n-s)$-hook added).

\begin{remark}
 For $s \geq n/2$, the situation is more complicated, but one can still say the following (see \cite{EN}). 

\InnaB{Let $c' = n-s$. If $M_{c, n}(\mu), M_{c, n}(\T)$ belong to the same block, then they have the same $(n-s)$-core. In particular, if $\T$ has no $(n-s)$-hooks, then $M_{c, n}(\T)$ is a simple module lying in a semisimple block.} 
\end{remark}

\section{Constructions for \texorpdfstring{$\InnaE{\underline{\co}}_{\text{  } c,\nu}$}{the category O in complex rank}}\label{sec:constr_for_H_c_nu}

We now define constructions analogous to those described in Section \ref{sec:blocks_classical_case} for the category $\InnaE{\underline{\co}}_{\text{  } c,\nu}$, where we consider generic values of $\nu$. We use the idea described in Section \ref{sec:Del_cat} of treating a simple object $X_{\T}$ of $\InnaE{\underline{\mathrm{Rep}}}(S_{\nu})$ as a Young diagram obtained by adding a very long top row to $\T$ (``the top row having length $(\nu-\abs{\T})$''). This idea is only meant to provide some intuition, but it is true that for $n\in \bZ, n>>0$, the image of $X_{\T}$ under the functor $\InnaE{\underline{\mathrm{Rep}}}(S_{\nu=n}) \longrightarrow \InnaE{\mathrm{Rep}}(S_n)$ is isomorphic to $\tilde{\T}(n)$, which is exactly a Young diagram obtained by adding a top row of length $(n-\abs{\T})$ to $\T$. 

So instead of adding and removing $(n-s)$-hooks as we did in Section \ref{sec:blocks_classical_case}, we would be adding and removing ``$(\nu-s)$-hooks'' from diagrams with ``very long top row of size $(\nu-\abs{\T})$''. Below is a formal description of the relevant constructions.

Let $\T$ be a Young diagram, and $s \geq 0$ be an integer.

\begin{definition}\label{s_core}
\mbox{}
\begin{itemize}
 \item Define $C_{\T} := \{ \abs{\T} -1 +j - \T\check{}_j  \mid j \geq 1\}$.
\item For $s \in C_{\T}, s= \abs{\T} -1 +j_s - \T\check{}_{j_s}$ for some $j_s \geq 1$, define the $(\nu-s)$-core of $\T$ as 
$\mathbf{core}_{(\nu-s)}(\T)\check{}_j = \T\check{}_j  +1$ if $ 1\leq j <j_s$, and $\mathbf{core}_{(\nu-s)}(\T)\check{}_j = \T\check{}_{j+1}$ if $ j\geq j_s$.
\end{itemize}
 
\end{definition}

That is, $\mathbf{core}_{(\nu-s)}(\T)$ is the Young diagram obtained by taking out column ${j_s}$ of $\T$, moving columns $1, .., {j_s}-1$ down and adding a row of length ${j_s}-1$ on top of them, and moving left the columns ${j_s}+1, ...$. 

\begin{example}
 Let $\T = (8,5,4,3,3,2)$. Then $\abs{\T} = 25$, $C_{\T} = \{19, 20, 22, 25, 27, 29, 30, 31\} \cup \bZ_{\geq 33}$. 
Let $s = 22$. Then $s \in  C_{\T}, s= \abs{\T} -1 +3 - \T\check{}_3$. So $j_s=3$, and $\mathbf{core}_{(\nu-s)}(\T) = (7, 4,3,2,2,2,2)$. 
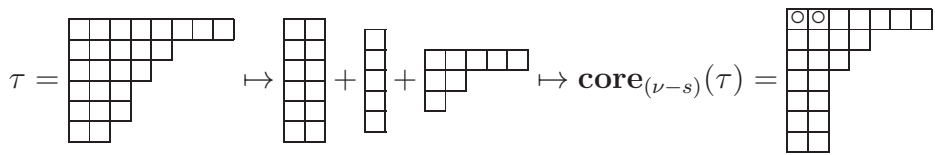
 \begin{align*}
  &\T =\yng(8,5,4,3,3,2)  \mapsto \yng(2,2,2,2,2,2) + \yng(1,1,1,1,1) + \yng(5,2,1) 
 \mapsto \mathbf{core}_{(\nu-s)}(\T) = \young(\circ\circ\hfil\hfil\hfil\hfil\hfil,\hfil\hfil\hfil\hfil,\hfil\hfil\hfil,\hfil\hfil,\hfil\hfil,\hfil\hfil,\hfil\hfil)
 \end{align*}

Let $s = 34$. Then $s \in  C_{\T}, s= \abs{\T} -1 +10 - \T\check{}_{10}$. So $j_s =10$, $\mathbf{core}_{(\nu-s)}(\T) = (9,8,5,4,3,3,2)$.
 \begin{align*}
&\T =\yng(8,5,4,3,3,2)  \mapsto \mathbf{core}_{(\nu-s)}(\T) = \young(\circ\circ\circ\circ\circ\circ\circ\circ\circ,\hfil\hfil\hfil\hfil\hfil\hfil\hfil\hfil,\hfil\hfil\hfil\hfil\hfil,\hfil\hfil\hfil\hfil,\hfil\hfil\hfil,\hfil\hfil\hfil,\hfil\hfil)
 \end{align*}
\end{example}

\begin{remark}
 Deligne defined $\mathbf{core}_{(\nu-s)}(\lambda)$, which he denoted by $\{\lambda\}^+_n$, with $n:=s$, in \cite[7.5]{D}. 
\end{remark}

The process of ``reconstruction'' (corresponding to inserting a hook into a Young diagram) is defined as follows:

\begin{construction}\label{reconstr_description}
 Given any Young diagram $\eta$ and an integer $l \geq 0$, we define $\T := \mathbf{rec}(l, \eta)$ as the Young diagram obtained through the following steps:
 \begin{itemize}
  \item Find the index $k \geq 1$ such that $\eta_{k-1}\InnaA{\check{}} \geq l+1 > \eta_k\InnaA{\check{}} $ (here $\eta_0 := \infty$).
 \item Define $\T =\mathbf{rec}(l, \eta)$ by $\T\check{}_j  := \eta\check{}_{j}- 1$ for $j <k$, $\T\check{}_k  := l$ and $\T\check{}_j := \eta\check{}_{j-1}$ for $j > k$.
 \end{itemize}

\end{construction}

That is, we divide $\eta$ into two parts: part $1$ consisting of columns $1,..., k-1$ and part $2$ consisting of columns $k, k+1,...$. Then we delete the top row of part $1$, add a $k$-th column of length $l$, and add part $2$ as columns $k+1, k+2, ...$.

\begin{example} 
\mbox{}
 \begin{itemize} 
  \item For $ l > \eta\check{}_{1}$, $\mathbf{rec}(l, \eta)$ is the Young diagram obtained from $\eta$ by adding a column of length $l$ to $\eta$ (this will become the first column). 
\item For $l=0$, $\mathbf{rec}(l, \eta)$ is the Young diagram obtained from $\eta$ by removing its top row.
 \end{itemize}

\end{example}
\begin{example}
 $\mathbf{rec}(\T\check{}_{j_s}, \mathbf{core}_{(\nu-s)}(\T)) = \T$.
\end{example}
\begin{example}
Let $\T = (10,8,8,6,5,4,1)$, $s=\InnaD{43}$ (so $j_s =6$). Then $\mathbf{rec}(7, \mathbf{core}_{(\nu-s)}(\T)) =  (10,8,8,6,6,6,5)$.
\begin{align*}
&\yng(10,8,8,6,5,4,1)  \mapsto \yng(5,5,5,5,5,4,1) + \yng(1,1,1,1) + \yng(4,2,2)\mapsto \mathbf{core}_{(\nu-s)}(\T) =\young(\times\times\times\times\times\hfil\hfil\hfil\hfil,\hfil\hfil\hfil\hfil\hfil\hfil\hfil,\hfil\hfil\hfil\hfil\hfil\hfil\hfil,\hfil\hfil\hfil\hfil\hfil,\hfil\hfil\hfil\hfil\hfil,\hfil\hfil\hfil\hfil\hfil,\hfil\hfil\hfil\hfil,\hfil) \mapsto \\
&\mapsto 
\young(\times,\hfil,\hfil,\hfil,\hfil,\hfil,\hfil,\hfil) +\young(\times\times\times\times\hfil\hfil\hfil\hfil,\hfil\hfil\hfil\hfil\hfil\hfil,\hfil\hfil\hfil\hfil\hfil\hfil,\hfil\hfil\hfil\hfil,\hfil\hfil\hfil\hfil,\hfil\hfil\hfil\hfil,\hfil\hfil\hfil)
\mapsto
\yng(1,1,1,1,1,1,1) +\young(\circ,\circ,\circ,\circ,\circ,\circ,\circ) + \young(\times\times\times\times\hfil\hfil\hfil\hfil,\hfil\hfil\hfil\hfil\hfil\hfil,\hfil\hfil\hfil\hfil\hfil\hfil,\hfil\hfil\hfil\hfil,\hfil\hfil\hfil\hfil,\hfil\hfil\hfil\hfil,\hfil\hfil\hfil) \mapsto  \mathbf{rec}(7, \mathbf{core}_{(\nu-s)}(\T))=\young(\hfil\circ\times\times\times\times\hfil\hfil\hfil\hfil,\hfil\circ\hfil\hfil\hfil\hfil\hfil\hfil,\hfil\circ\hfil\hfil\hfil\hfil\hfil\hfil,\hfil\circ\hfil\hfil\hfil\hfil,\hfil\circ\hfil\hfil\hfil\hfil,\hfil\circ\hfil\hfil\hfil\hfil,\hfil\circ\hfil\hfil\hfil) 
\end{align*}

\end{example}

\begin{notation}\label{notn:rec_core}
\InnaA{ Let $\T$ be a Young diagram, $s\in C_{\T}, l\in \bZ_{\geq -\T\check{}_{j_s}}$ \InnaB{($j_s$ given by $s=\abs{\T} -1 +j_{s} - \T\check{}_{j_s}$)}. We will denote the Young diagram $\mathbf{rec}(\T\check{}_{j_s}+l, \mathbf{core}_{(\nu-s)}(\T))$ by $\Gamma(\T, s, l)$.}
\end{notation}

\begin{example}\label{empty_diagr_core}
 Consider $\T = \emptyset$. Given $s \geq 0$, the $(\nu-s)$-core of $\emptyset$ is a row of length $s$. Then $\Gamma(\emptyset, s, l)=\mathbf{rec}(l, \mathbf{core}_{(\nu-s)}(\emptyset))$ is a hook with arm length $s$ and leg length $l-1$.
\end{example}

\begin{lemma}\label{identity_for_s_and_diagrams}
 Let $\T$ be a Young diagram, $l >0 $ an integer, $s :=\abs{\T} -1 +j_s - \T\check{}_{j_s}$ for some $j_s \geq 1$, $\mu = \Gamma(\T, s, l)$. Then $\abs{\mu} - \abs{\T}>0$, and $s =  \frac{f(\mu)-f(\T)}{\abs{\mu} -\abs{\T}}$ ($f$ is defined in Subsection \ref{ssec:rmrks_main_for}).
\end{lemma}
\begin{proof}
 We first describe $\mu$ in terms of $\T$. By Constructions \ref{s_core}, \ref{reconstr_description} described above, $\mu$ is obtained from $\T$ by taking out column $j_s$ of $\T$, inserting a column of length $\T\check{}_{j_s}+l$ after column $k_{s,l}-1$ for some $k_{s,l} \leq j_s$ ($k_{s,l}$ is uniquely determined by Construction \ref{reconstr_description}), and adding a cell to each of the columns $k_{s,l}+1, ..., j_s$ of the newly constructed diagram. Thus $\mu\check{}_j=\T\check{}_j $ for $j<k_{s,l}$ and $j>j_s$, $\mu\check{}_{k_{s,l}}=\T\check{}_{j_s}+l$, and $\mu\check{}_j=\T\check{}_\InnaA{j}  +1$ for $j=k_{s,l}+1, ..., j_s$.
This means that we have:

$$\abs{\mu} - \abs{\T} = j_s -k_{s,l}+l$$

Note that since $k_{s,l} \leq j_s$, $l>0$, we have: $\abs{\mu} - \abs{\T}>0$.

\InnaB{From this description of $\mu$, we see that }

\begin{align*}
 &ct(\mu) -ct(\T)= -\left(j_s \T\check{}_{j_s} - \frac{ \T\check{}_{j_s}(\T\check{}_{j_s}+1)}{2}\right) + \left( \frac{ j_s(j_s-1)}{2}- \frac{ k_{s,l}(k_{s,l}-1)}{2}\right) + \\
&+\left(k_{s,l} (\T\check{}_{j_s} +l) - \frac{ (\T\check{}_{j_s} +l)(\T\check{}_{j_s}+ l+1)}{2}\right) =\\
&= -\T\check{}_{j_s}(j_s -k_{s,l}+l) + \frac{ j_s(j_s-1)}{2}- \frac{ k_{s,l}(k_{s,l}-1)}{2} - \frac{ \ell(l+1)}{2} +k_{s,l} l =\\
&= -\T\check{}_{j_s}(j_s -k_{s,l}+l) - \frac{j_s -k_{s,l}+l}{2} + \frac{ j_s^2 - (k_{s,l}- l)^2}{2} =\\
&= -\left(\T\check{}_{j_s}+\frac{1}{2} \right)(j_s -k_{s,l}+l) + \frac{ (j_s+k_{s,l}-l)(j_s -k_{s,l}+l)}{2}
\end{align*}

And thus
$$\frac{ct(\mu) -ct(\T)}{\abs{\mu} -\abs{\T}} = -\T\check{}_{j_s}-\frac{1}{2} + \frac{j_s+k_{s,l}-l}{2} $$

Recall that we also have, by definition:

$$ \frac{f(\mu)-f(\T)}{\abs{\mu} -\abs{\T}} = \frac{\abs{\mu} + \abs{\T} -1 }{2} +\frac{ct(\mu) -ct(\T)}{\abs{\mu} -\abs{\T}}$$

And so 
\begin{align*}
 &\frac{f(\mu)-f(\T)}{\abs{\mu} -\abs{\T}} =\frac{\abs{\mu} + \abs{\T} -1 }{2}  -\T\check{}_{j_s}-\frac{1}{2} + \frac{j_s+k_{s,l}-l}{2} =\\
&= \frac{2\abs{\T} +j_s -k_{s,l}+l-1 }{2}  -\T\check{}_{j_s}-\frac{1}{2} + \frac{j_s+k_{s,l}-l}{2} = \abs{\T} -1 +j_s - \T\check{}_{j_s} =s
\end{align*}
 
\end{proof}

These constructions are compatible with the constructions described in Section \ref{sec:blocks_classical_case} in the following sense:

\begin{proposition}\label{compat_of_constr_cores}
 Let $n >>0$. Put $\lambda^{(l)} := \mathbf{rec}(l,\mathbf{core}_{(\nu-s)}(\T))$. Then $\widetilde{\lambda^{(l)}}(n) = \mathtt{rec}(l,\mathtt{core}_{(n-s)}(\tilde{\T}(n)))$.
\end{proposition}

\begin{proof}

One can easily see that the procedure for constructing $\mathbf{rec}(l,\mathbf{core}_{(\nu-s)}(\T))$ coincides with the procedure for constructing $\mathtt{rec}(l,\mathtt{core}_{(n-s)}(\tilde{\T}(n)))$ and then removing the top row.
\end{proof}



\section{Blocks in \InnaE{the} category \texorpdfstring{$\InnaE{\underline{\co}}_{\text{  } c,\nu}$}{O in complex rank}}\label{sec:blocks_O_H_c_nu}

\subsection{Morphisms between two Verma objects}\label{ssec:morphs_between_2_objects}

We now give some necessary and some sufficient conditions for the existence of a non-trivial morphism between Verma objects.
Fix Young diagrams $\T$, $\mu$.

\InnaE{The purpose of this section is to prove the main theorem:}
\begin{theorem}\label{thrm:L_tau_mu_m_lies_in_B}
 For two \InnaA{distinct} Young diagrams $\mu, \T$ and an integer $m>0$, the following are equivalent:
\begin{enumerate}
 \item $\abs{\mu} \neq \abs{\T}$ and $\mathcal{L}_{\T, \mu, m} \subset \mathcal{B}_{\mu, \T}$,
\item $\mu = \Gamma(\T, s, \sign({\abs{\mu} -\abs{\T}}))$ for some $s \in C_{\T}$ (in particular, $\abs{\mu} \neq \abs{\T}$), \InnaB{ and $(\abs{\mu} -\abs{\T}) \mid m $}.
\end{enumerate}

\end{theorem}

\InnaE{We begin by proving the following proposition (this shows that $(2)\Rightarrow (1)$ in the Theorem above):}

\begin{proposition}\label{prop:suff_cond_for_morph}

Let $r \in \bZ \setminus \{0\}$, $s \in C_{\T}$ (i.e. $s =\abs{\T} -1 +j_s - \T\check{}_{j_s}$ for some $j_s \geq 1$). \InnaA{Assume $  \T\check{}_{j_s}+ \sign(r) \geq 0$.}

Let $ \mu = \Gamma(\T, s, \sign(r))$. 

If $c' = \frac{\nu -s}{r}$, then there exists a non-trivial morphism $M_{c, \nu}(\mu) \longrightarrow M_{c,\nu}(\T)$.
%

\end{proposition}

\begin{proof}
Assume $(c', \nu), \mu$ are as above.
  
Let $n>>0$ be an integer, such that $\Gcd(n-s, r) =1$ (there are infinitely many such positive integers).

By Corollary \ref{int_blocks_conclusion}, \InnaA{Proposition \ref{equiv_c_and_minus_c}} and Proposition \ref{compat_of_constr_cores}, we have a non-trivial morphism $M_{c,n}(\tilde{\mu}(n)) \longrightarrow M_{c,n}(\tilde{\T}(n))$. Moreover, $s = \frac{f(\mu)-f(\T)}{\abs{\mu} -\abs{\T}}$ ($f$ defined in Subsection \ref{ssec:rmrks_main_for}) by Lemma \ref{identity_for_s_and_diagrams}, and the image of $\tilde{\mu}(n)$ sits in degree $m:=r(\abs{\mu} -\abs{\T})$ of 
$M_{c,n}(\tilde{\T}(n))$ (this is a direct consequence of Equation \eqref{main_for_rewritten}). 

Then by \InnaA{Lemma \ref{lem:Verma_int_dim_hom}}, there is a non-trivial morphism $M_{c,\nu=n}(\mu) \longrightarrow M_{c,\nu=n}(\T)$ with the image of $X_{\mu}$ sitting in degree $m=r(\abs{\mu} -\abs{\T})$ of $M_{c,\nu=n}(\T)$, and $s= \frac{f(\mu)-f(\T)}{\abs{\mu} -\abs{\T}}$.

\InnaA{
Consider the space $\InnaE{V^m}_{\nu} = \Hom_{\InnaE{\underline{\mathrm{Rep}}}(S_{\nu})}(X_{\mu}, S^m \InnaD{\fh^*_0} \otimes X_{\T})$ for any $\nu \in \bC$. By the construction of Deligne's categories, these spaces are isomorphic for almost every $\nu \in \bC$ (in particular, for all $\nu \notin \bZ_+$). We will denote the locus of such ``generic'' points $\nu \in \bC$ by $A$ (so $\abs{\bC \setminus A} < \infty, \bC \setminus \bZ_+ \subset A$), and write just $\InnaE{V^m}$ instead of $\InnaE{V^m}_{\nu}$ whenever $\nu \in A$.

\InnaB{
The map $y_{M_{c, \nu}(\T)}$ from the definition of $\InnaE{\underline{\mathrm{Rep}}}(H_c(\nu))$ defines a polynomial family of maps 
\begin{align*}
 \InnaE{Y^m}_{c, \nu}: \InnaE{V^m} &\longrightarrow \Hom_{\InnaE{\underline{\mathrm{Rep}}}(S_{\nu})}(\hhh \otimes X_{\mu}, S^{m-1} \InnaD{\fh^*_0} \otimes X_{\T})\\
 \phi &\mapsto y_{M_{c, \nu}(\T)} \circ \left( \id_{\hhh} \otimes \phi \right)
\end{align*}
}

It is easy to see that the locus of points $(c, \nu) \in \bC \times A$  such that $\InnaE{Y^m}_{c, \nu}$ is not injective is a Zariski closed subset of $\bC \times A$. Denote it by $\mathcal{C}$.

By definition, $\mathcal{C}$ is exactly the locus of points $(c, \nu) \in \bC \times A$ such that $y_{M_{c, \nu}(\T)} \lvert_{\hhh \otimes \phi(X_{\mu})} =0$ for some $\phi \in \InnaE{V^m \setminus \{0\}}$. The latter condition is equivalent to the existence of a non-trivial morphism $\Theta_{\nu}:M_{c,\nu}(\mu) \longrightarrow M_{c,\nu}(\T)$ such that $\Theta_{\nu}(X_{\mu})$ sits in degree $m$ of $M_{c,\nu}(\T)$).

The morphisms $M_{c,\nu=n}(\mu) \longrightarrow M_{c,\nu=n}(\T), n\in \bZ_{>>0}$ constructed above guarantee that $\InnaE{Y^m}_{c, \nu=n }$ is not injective whenever $n \in \bZ_{>>0}, c'=\frac{1}{c} = \frac{n -s}{r}$, which means that 
$(c, \nu=n) \subset \mathcal{C}$ for any $c, n$ satisfying $n>>0, \frac{1}{c} = \frac{n -s}{r}$.

Thus $\{ (c, \nu) \in \bC^2 \mid c' = \frac{1}{c} = \frac{\nu -s}{r} \} \subset \mathcal{C}$, and \InnaE{we obtain} the required statement.
}

\end{proof}

\begin{remark}
 In the proof of this proposition, we established that $s = \frac{f(\mu)-f(\T)}{\abs{\mu} -\abs{\T}}$, so by Equation \eqref{main_for_rewritten}, the image of $X_{\mu} \InnaE{\subset M_{c,\nu}(\mu)}$ sits in degree $m$ of $M_{c,\nu}(\T)$ for all $(c', \nu)$ such that $c' = \frac{\nu -s}{r}$. In general, for each pair $(c', \nu)$, Equation \eqref{main_for_rewritten} allows us to compute the unique degree $m$ in which the image of $X_{\mu}$ could sit in $M_{c,\nu}(\T)$.
\end{remark}

We continue with Young diagrams $\T$, $\mu$ fixed, \InnaE{and proceed to prove the remainder of Theorem \ref{thrm:L_tau_mu_m_lies_in_B}}.

\begin{proof}[Proof of Theorem \ref{thrm:L_tau_mu_m_lies_in_B}]
 \InnaE{We prove the direction $(1) \Rightarrow (2)$ of Theorem \ref{thrm:L_tau_mu_m_lies_in_B}.
 
 Assume $\abs{\mu} \neq \abs{\T}$.}

Consider the line $\mathcal{L}_{\T, \mu, m} \subset \bC^2 $ (defined in Notation \ref{L_tau_mu_m_notn}). We would like to check whether $\mathcal{L}_{\T, \mu, m}$ satisfies the following condition:

\begin{cond}\label{non_triv_morph_cond}
 For all \InnaE{$(c',\nu) \in \mathcal{L}_{\T, \mu, m}$ there exists a non-trivial morphism} $\Theta_{\nu}:M_{c,\nu}(\mu) \longrightarrow M_{c,\nu}(\T)$ such that $\Theta_{\nu}(X_{\mu})$ sits in degree $m$ of $M_{c,\nu}(\T)$.
\end{cond}

Assume this is indeed the case. \InnaE{We will show that this implies that $\T, \mu, m$ satisfy Part (2) of Theorem \ref{thrm:L_tau_mu_m_lies_in_B}.

Since $\abs{\mu} \neq \abs{\T}$, the Equation \eqref{main_for1} tells us that for almost any $n \in \bZ_+$ there exists $c' \in \bC$ such that $(c',\nu=n) \in \mathcal{L}_{\T, \mu, m}$.

Due to} Lemma \ref{lem:Verma_int_dim_hom}, Condition \ref{non_triv_morph_cond} implies that for integer $n>>0$ there must exist non-trivial morphisms $\tilde{\Theta}_{n}: M_{c,n}(\tilde{\mu}(n)) \longrightarrow M_{c,n}(\tilde{\T}(n))$ (corresponding to $\Theta_{\nu=n}$) with the image of $\tilde{\mu}(n)$ sitting in degree $m$ of $M_{c,n}(\tilde{\T}(n))$.

\InnaE{As we saw in Section \ref{sec:blocks_classical_case}, this statement imposes a rather strong condition on the Young diagrams $\T, \mu$, as well as on the integer $m$. The remainder of the proof consists of translating these conditions to the setting of Deligne categories, which will result in the conditions from Part (2) of the statement of Theorem \ref{thrm:L_tau_mu_m_lies_in_B}.}

\mbox{}


Recall that 
\InnaE{the requirement that} $(c',n) \in \mathcal{L}_{\T, \mu, m}$ is equivalent to \InnaE{the requirement} $$ c'= \frac{( \abs{\mu}- \abs{\T})n - (f(\mu) - f(\T))}{m}$$ 

Denote: $$a:= \abs{\mu}- \abs{\T}, \; b:=f(\mu) - f(\T), \; d_n:= \Gcd(an-b, m)$$ Then $$c' = \frac{an-b}{m} = \frac{(an-b)/d_n}{m/d_n}$$ where $$\frac{an-b}{d_n} \in \bZ, \; \frac{m}{d_n} \in \bZ_{>0}, \; \Gcd \left( \frac{an-b}{d_n}, \frac{m}{d_n} \right)=1$$

By Theorem \ref{Rouquier}, if $(an-b)/d_n \neq 2 $, we can pass to the case $(c' = (an-b)/d_n, n)$, with the image of $\tilde{\mu}(n)$ sitting in degree $d_n$ of $M_{c,n}(\tilde{\T}(n))$. 

Note that \InnaE{since} $a \neq 0$ \InnaE{and} since $0< d_n \leq m$, we have $(an-b)/d_n \neq 2 $ for $n>>0$.

Now we have the following cases: 

\mbox{}

{\bf Case $a>0$.}

\begin{lemma}
 For $n>>0$, \InnaE{if} there are no non-trivial morphisms $\tilde{\Theta}_{n}$ \InnaE{then} 

$a \mid b, \InnaE{a \mid} m$, i.e. $d_n = a$.
\end{lemma}
\begin{proof}
 \InnaE{Assume there are non-trivial morphisms $\tilde{\Theta}_{n}$ for $n>>0$}. Then we have, by Theorem \ref{DJ}: $$\frac{an-b}{d_n} < n$$ for $n >>0$; \InnaE{that is}, $a \leq d_n$. Also, putting $$d := \Gcd(a,b), \; a':= \frac{a}{d}, \; b':= \frac{b}{d}$$ we have: $d_n = \Gcd\left(d (a'n-b'), m\right) \geq a= d a' \geq  1$ for $n>>0$. This implies, for $n>>0$: either $d_n \mid d$ (possible only if $d_n=d=a$, and then we have $a \mid b$, $\InnaE{a \mid } m$), or $\Gcd(a'n -b', m) \neq 1$.

So we only need to check that there is no $n_0 \in \bZ$ such that $\forall n > n_0, \Gcd(a'n -b', m) \neq 1$. Indeed, $b',\Gcd(a', m)$ are relatively prime (since $a',b'$ are), and so $\Gcd(a'n -b', m) = \Gcd(\frac{a'}{\Gcd(a', m)}n -b', \frac{m}{\Gcd(a', m)})$. Now $$\Gcd \left(\frac{a'}{\Gcd(a', m)}, b'\right)=\Gcd \left(\frac{a'}{\Gcd(a', m)}, \frac{m}{d'}\right)=1$$ so $$\Gcd(a'n -b', m) =\Gcd \left(\frac{a'}{\Gcd(a', m)}n -b', \frac{m}{\Gcd(a', m)}\right)= 1$$ for infinitely many integer values of $n$.
\end{proof}

Thus $d_n = a$, and $(an-b)/d_n \neq 2 $ for $n>>0$. So to understand what happens for $(c', n) \in \mathcal{L}_{\T, \mu,m }$, it is enough to check what happens the case $(c' = n-(b/a), n)$ (recall: $b/a = \frac{f(\mu)-f(\T)}{\abs{\mu} -\abs{\T}} \in \bZ$). Still assuming the existence of non-trivial morphisms $\tilde{\Theta}_{n}$, Proposition \ref{compat_of_constr_cores} and Corollary \ref{int_blocks_conclusion} give:

For $n>>0$ and $c' = n- b/a$, put $s:=b/a$, and we have $$ \mu = \Gamma(\T, s, 1) \text{ for } j_{s} \geq 1, \text{ such that } s=\abs{\T} -1 +j_{s} - \T\check{}_{j_s}$$

 Note that a priori, $j_{s}$ could depend on $n$, but the equality $s=\abs{\T} -1 +j_{s} - \T\check{}_{j_s}$ defines $j_{s}$ uniquely for each $s$, so since $s$ doesn't depend on $n$, neither does $j_{s}$.


\begin{conc}\label{case_a_positive_conclusion}
 For $\abs{\mu} >\abs{\T}$, $\mathcal{L}_{\T, \mu, m} $ satisfies Condition \ref{non_triv_morph_cond} if and only if the following hold:
\begin{enumerate}
  \item $s :=  \frac{f(\mu)-f(\T)}{\abs{\mu} -\abs{\T}} \in C_{\T} \subset \bZ$,
  \item $\mu = \Gamma(\T, s, 1)$,
  \item $(\abs{\mu} -\abs{\T}) \mid m $.
\end{enumerate}

In that case, points $(c', \nu) \in \mathcal{L}_{\T, \mu, m}$ satisfy the condition: $c' = \frac{\nu -s}{r} $ for $ r=\frac{m}{\abs{\mu} -\abs{\T}}$.

\end{conc}

\mbox{}


{\bf Case $a<0$.} 

By Observation \ref{equiv_c_and_minus_c}, the existence of a non-trivial morphism  $\tilde{\Theta}_{n}$ for $n>>0$ is equivalent  to the existence of a non-trivial morphism $\tilde{\Upsilon}_{n}: M_{-c',n}(\tilde{\T}(n)) \longrightarrow M_{-c',n}(\tilde{\mu}(n))$. The necessary and sufficient conditions for the existence of a non-trivial morphism $\tilde{\Upsilon}_{n}$ for all $n>>0$ are given in Conclusion \ref{case_a_positive_conclusion}, with $\mu$ and $\T$ switching roles.

\begin{conc}
For $\abs{\mu} <\abs{\T}$, $\mathcal{L}_{\T, \mu, m} $ satisfies Condition \ref{non_triv_morph_cond} if and only if the following hold:
\begin{enumerate}
  \item $s :=  \frac{f(\mu)-f(\T)}{\abs{\mu} -\abs{\T}} \in C_{\mu} \subset \bZ$,
  \item $\mu = \Gamma(\T, s, -1)$,
  \item $(\abs{\T} -\abs{\mu}) \mid m $.
\end{enumerate}

In that case, points $(c', \nu) \in \mathcal{L}_{\T, \mu, m}$ satisfy the condition: $c' =- \frac{\nu -s}{r}$ for $r=\frac{m}{\abs{\T} -\abs{\mu}}$.
\end{conc}

\end{proof}

\begin{remark}
 The results in this subsection agree with the statement of Lemma \ref{necessary_and_suff_cond_in_deg_1} when $m=1$, but they imply the statement of Lemma \ref{necessary_and_suff_cond_in_deg_1} only if $\abs{\mu} \neq \abs{\T}$.
\end{remark}

\subsection{Properties of lines \texorpdfstring{$\mathcal{L}_{\T, \mu, m}$}{defined in Section 8}}\label{ssec:prop_lines_L_tau_mu_m}
We conclude this section with a list of (almost trivial) properties of the lines $\mathcal{L}_{\T, \mu, m}$:

\begin{lemma} \InnaA{Let $m_1, m_2 \in \bZ_+$, and let $\T, \mu, \mu^{(1)}, \mu^{(2)}$ be Young diagrams. Then }
 \begin{itemize}
  \item $\mathcal{L}_{\T, \mu, m_1} \cap \mathcal{L}_{\T, \mu, m_2} \neq \emptyset \Leftrightarrow m_1 =m_2$.
  \item If $(c', \nu) \in \mathcal{L}_{\T, \mu^{(1)}, m_1} \cap \mathcal{L}_{\T, \mu^{(2)}, m_2} $, and the lines $\mathcal{L}_{\T, \mu^{(1)}, m_1}, \mathcal{L}_{\T, \mu^{(2)}, m_2}$ do not coincide, then $c', \nu \in \bQ$ (and there is only one point $(c', \nu)$ like this).
  \item Assume the lines $\mathcal{L}_{\T, \mu^{(1)}, m_1}, \mathcal{L}_{\T, \mu^{(2)}, m_2}$ coincide, $\abs{\mu^{(1)}} \neq \abs{\T}$, and $\mathcal{L}_{\T, \mu^{(1)}, m_1} \subset \mathcal{B}_{\mu^{(1)}, \T}$. Then $\mathcal{L}_{\T, \mu^{(2)}, m_2} \subset \mathcal{B}_{\mu^{(2)}, \T}$ iff $\mu^{(1)}= \mu^{(2)}$.
 \end{itemize}
 \end{lemma}

 \begin{proof}
\mbox{}
  \begin{itemize}
   \item Follows immediately from the definition of $\mathcal{L}_{\T, \mu, m}$ (see Notation  \ref{L_tau_mu_m_notn}).
   \item Follows immediately from the definition of $\mathcal{L}_{\T, \mu, m}$ by a linear equation with rational coefficients.
   \item First, $\abs{\mu^{(1)}} \neq \abs{\T}$, and $\mathcal{L}_{\T, \mu^{(1)}, m_1} \subset \mathcal{B}_{\mu^{(1)}, \T}$, so by Theorem \ref{thrm:L_tau_mu_m_lies_in_B}, putting $$s:= \frac{f(\mu^{(1)}) - f(\T)}{\abs{\mu^{(1)}} - \abs{\T}}$$ \InnaE{we obtain}
    $$\mu^{(1)} = \Gamma(\T, s, \sign({\abs{\mu^{(1)}} -\abs{\T}}))$$ with $j_s$ given by $s=\abs{\T} -1 +j_{s} - \T\check{}_{j_s}$.
    
    Next, since the lines $\mathcal{L}_{\T, \mu^{(1)}, m_1}, \mathcal{L}_{\T, \mu^{(2)}, m_2}$ coincide and  $\abs{\mu^{(1)}} \neq \abs{\T}$, \InnaE{we obtain}: 
   \begin{align*}
    [m_1: (\abs{\mu^{(1)}} - \abs{\T}): f(\mu^{(1)}) - f(\T)] =  [m_2: (\abs{\mu^{(2)}} - \abs{\T}): f(\mu^{(2)}) - f(\T)]
    \end{align*}
 which implies $ s = \frac{f(\mu^{(1)}) - f(\T)}{\abs{\mu^{(1)}} - \abs{\T}} = \frac{f(\mu^{(2)}) - f(\T)}{\abs{\mu^{(2)}} - \abs{\T}}$.
   
    \InnaA{We also get:} $\abs{\mu^{(2)}} \neq \abs{\T}$ and $\sign({\abs{\mu^{(1)}} -\abs{\T}}) = \sign({\abs{\mu^{(2)}} -\abs{\T}})$.
    
    Now assume that $\mathcal{L}_{\T, \mu^{(2)}, m_2} \subset \mathcal{B}_{\mu^{(2)}, \T}$. Since 
    $$\abs{\mu^{(2)}} \neq \abs{\T}, \text{ } s = \frac{f(\mu^{(1)}) - f(\T)}{\abs{\mu^{(1)}} - \abs{\T}} = \frac{f(\mu^{(2)}) - f(\T)}{\abs{\mu^{(2)}} - \abs{\T}}$$ and $$\sign({\abs{\mu^{(1)}} -\abs{\T}}) = \sign({\abs{\mu^{(2)}} -\abs{\T}})$$
    \InnaE{we obtain}, from \ref{thrm:L_tau_mu_m_lies_in_B}:
    $$\mu^{(2)} = \Gamma(\T, s, \sign({\abs{\mu^{(2)}} -\abs{\T}})) = \mu^{(1)}$$
  \end{itemize}

 \end{proof}

From the third part of the above lemma, we immediately get: 
\begin{corollary}
 Let $\abs{\mu} \neq \abs{\T}$, and $\mathcal{L}_{\T, \mu, m} \subset \mathcal{B}_{\mu, \T}$. Then for a generic point $(\InnaE{c'}, \nu) \in \mathcal{L}_{\T, \mu, m}$, $(\InnaE{c'}, \nu) \notin \mathcal{B}_{\mu', \T}$ for any $\mu' \neq \mu$.
\end{corollary}

\section{Characters of simple objects in \texorpdfstring{$\InnaE{\underline{\co}}_{\text{  } c,\nu}$}{category O (complex rank)}}\label{sec:char_simple_obj}

\InnaB{\InnaF{Throughout this section, we assume that} $\nu \not\in \bZ_+, c \not= 0$. We also continue to identify $\hhh$ with $\InnaD{\fh^*_0}$, $\fh$ with $\fh^*$.}

\subsection{Definitions}

\begin{definition}
 The character of a graded vector space $V = \bigoplus_{ k \in \bZ} V_k$ will be defined as $ch_q (V) := \sum_k q^k \dim V_k$.
\end{definition}

\begin{definition}[Character of a graded \InnaE{\rm{ind}}-object of $\InnaE{\underline{\mathrm{Rep}}}(S_{\nu})$]
The character of an \InnaE{\rm{ind}}-object $V$ of $\InnaE{\underline{\mathrm{Rep}}}(S_{\nu})$ with the grading $V = \bigoplus_{j \in \bZ_+} V_j, V_j \in \InnaE{\underline{\mathrm{Rep}}}(S_{\nu})$ is defined to be a formal power series in $t$: 

$$ \mathbf{ch}_t V = \sum_{\mu \InnaE{\in \mathcal{P}}} X_{\mu} ch_t \left( \Hom_{S_\nu}(X_{\mu}, V) \right) $$

\end{definition}

\begin{definition}[Character of a object of $\InnaE{\underline{\co}}_{\text{  } c,\nu}$]

Let $M$ be an object of the category $\InnaE{\underline{\co}}_{\text{  } c,\nu}$. The character of $M$ is defined to be a formal power series in $t$:
 $$ch_t M:=\sum_{\mu \InnaE{\in \mathcal{P}}} X_{\mu} tr \left( t^{\brh} \mid_{\Hom_{S_\nu}(X_{\mu}, M)} \right)$$
with complex degrees, and coefficients in $K_0(\InnaE{\underline{\mathrm{Rep}}}(S_{\nu}))$, which is the Grothendieck ring of $\InnaE{\underline{\mathrm{Rep}}}(S_{\nu})$.
We will usually write $ch M$ instead of $ch_t M$ for short.
\end{definition}
\begin{remark}
 The trace of the operator $t^{\brh}$ is defined since $\brh$ acts locally finitely on objects in $\InnaE{\underline{\co}}_{\text{  } c,\nu}$.
\end{remark}

\subsection{Characters of Verma objects}\label{ssec:char_Verma_object}
 
 Recall that $M(\lambda) \cong \InnaB{S\hhh \otimes  X_{\lambda}}$ as Ind-objects of $\InnaE{\underline{\mathrm{Rep}}}(S_{\nu})$. So it is enough to compute the character of $S\hhh$, and then use the formula $$ ch_t M(\lambda) = \mathbf{ch}_t (S\hhh) X_{\lambda}t^{h_{c, \nu}(\lambda) } $$ 

We now give a formula for computing the character of the \InnaE{\rm{ind}}-object $S\hhh$ of $\InnaE{\underline{\mathrm{Rep}}}(S_{\nu})$, which comes from the character formula in \cite[3.7]{EM}.

\begin{proposition}
We have the following formula for the character of $S\hhh$: $$ \mathbf{ch}_t (S\hhh)= \frac{1}{\sum_{n \geq 0} (-1)^{n} \;\; \Lambda^n \hhh \;\; t^n}$$
\end{proposition}

\begin{proof}
\InnaA{Consider the category $\it{Schur}$ of all Schur functors: the objects are Schur functors 
$$S_R: (FinVect) \rightarrow (FinVect), S_R(\cdot)= \bigoplus_{n \geq 0} R_n \otimes_{\bC[S_n]}(\cdot)^{\otimes n}$$
where $R = (R_n)_{n \geq 0}$, each $R_n$ being a finite-dimensional representation of $S_n$, and only finitely many of these representations being nonzero; the morphisms are natural transformations between such functors.} This category is equivalent to the category $\bigoplus_{n \in \bZ} \InnaE{\mathrm{Rep}}(S_n)$, and thus semisimple, \InnaC{with the simple objects parameterized (up to isomorphism) by the set of all Young diagrams.}

\InnaC{The tensor structure on $\it{Schur}$ comes from the tensor structure on $(FinVect)$. Namely, given two simple objects parametrized by Young diagrams $\lambda$ and $\mu$, their tensor product decomposes as a direct sum of simple objects $\rho$, with the multiplicity of $\rho$ being the Littlewood-Richardson coefficient $c^{\rho}_{\lambda, \mu}$.}

One can see that $Fun(\it{Schur}, \it{\InnaE{\underline{\mathrm{Rep}}}(S_{\nu})}) \cong \it{\InnaE{\underline{\mathrm{Rep}}}(S_{\nu})}$ (this is true for any symmetric tensor category), where $Fun(\it{Schur}, \it{\InnaE{\underline{\mathrm{Rep}}}(S_{\nu})})$ are tensor functors between additive linear tensor categories.


\InnaC{
Fix the object $V$ in $\it{Schur}$ corresponding to the identity functor (which is, of course, a Schur functor as well). Then we can identify the simple objects in $\it{Schur}$ with $S^{\lambda} V$ (the Schur functor $S^{\lambda}$ applied to $V$), where $\lambda$ runs through all Young diagrams. 

Consider the commutative algebra $SV$, and the exterior power $\Lambda^m V$ ($m \geq 0$). Then we have (c.f. \cite[6.8]{FH}): $$SV \otimes \Lambda^m V = \bigoplus_{\lambda \in \mathcal{I}^+_{\pi^m}} S^{\lambda} V$$ where $\mathcal{I}^+_{\pi^m}$ is the set of all Young diagrams obtained from $\pi^m$ by adding several boxes, no two in the same column.

One can easily see that for any $m \geq 1$, $$ \mathcal{I}^+_{\pi^m} =  \left( \mathcal{I}^+_{\pi^{m+1}} \cap \mathcal{I}^+_{\pi^m} \right) \sqcup \left( \mathcal{I}^+_{\pi^{m-1}} \cap \mathcal{I}^+_{\pi^m} \right)$$
and $ \mathcal{I}^+_{\pi^0} =  \mathcal{I}^+_{\emptyset}=  \mathcal{I}^+_{\pi^{1}} \sqcup \{ \emptyset\}$. So we have an infinite exact sequence in $\it{Schur} \cong \bigoplus_{n \in \bZ} \InnaE{\mathrm{Rep}}(S_n)$:

$$ ...\longrightarrow SV \otimes \Lambda^m V \longrightarrow SV \otimes \Lambda^{m-1} V \longrightarrow ...\longrightarrow  SV \otimes V \longrightarrow SV \longrightarrow S^{\emptyset} V \longrightarrow 0$$

Notice that when applying these functors to a finite-dimensional vector space $W$, \InnaE{we obtain} a finite exact sequence which is the Koszul resolution of $\bC$ over the symmetric algebra $SW$.}

\InnaB{Thus we obtain an exact complex of Ind-objects of $\InnaE{\underline{\mathrm{Rep}}}(S_{\nu})$:
$$ ...\longrightarrow S\hhh \otimes \Lambda^m \hhh  \longrightarrow S\hhh \otimes \Lambda^{m-1} \hhh  \longrightarrow ... \longrightarrow S\hhh \otimes \hhh \longrightarrow S\hhh \longrightarrow \triv \longrightarrow 0$$}
Now the character formula follows directly from Euler's formula applied to this complex.
\end{proof}

\subsection{The graded space \texorpdfstring{$Hom_{\InnaE{\underline{\mathrm{Rep}}}(S_{\nu})}(X_\mu, S\hhh \otimes X_\T)$}{Hom(X_mu, M(Tau))}}\label{ssec:graded_hom_study}

In this subsection we describe explicitly the decomposition of $S\fh \otimes X_\T$ and $S\hhh \otimes X_\T$ as $\InnaE{\underline{\mathrm{Rep}}}(S_{\nu})$ \InnaE{\rm{ind}}-objects into graded sums of simple objects.

\begin{lemma}\label{lem:char_classical_verma}
 For $n>>0$ (in fact, for $n> 2(\abs{\mu} +\abs{\T})$), we have:
 $$ch_q(\Hom_{S_n}(\tilde{\mu}(n), \bC[x_1,...,x_n] \otimes \tilde{\T}(n))) = (s_{\tilde{\T}(n)} \ast s_{\tilde{\mu}(n)})(1, q, q^2, ...)$$
 \end{lemma} 

 The latter expression can be rewritten as

 \begin{align*}
  &ch_q(\Hom_{S_n}(\tilde{\mu}(n), \bC[x_1,...,x_n] \otimes \tilde{\T}(n)))= \sum_{\lambda, \abs{\lambda}+\lambda_1 \leq n} \gamma^{\tilde{\lambda}(n)}_{\tilde{\T}(n), \tilde{\mu}(n)} s_{\tilde{\lambda}(n)}(1, q, q^2, ...) =\\
&= \sum_{\lambda} \overline{\gamma}^{\lambda}_{\T, \mu} s_{\tilde{\lambda}(n)}(1, q, q^2, ...)
 \end{align*}

Here 
\begin{itemize}
 \item $s_{\tilde{\lambda}(n)}$ is the Schur symmetric function corresponding to the partition $\tilde{\lambda}(n)$ of $n$ (see \cite[Chapter I, Par. 3, p.41]{Mac}),
\item $s_{\tilde{\T}(n)} \ast s_{\tilde{\mu}(n)}$ is the internal product of Schur symmetric functions as defined in \cite[Chapter I, Par. 7, p.116]{Mac},
\item $\gamma^{\alpha'}_{\alpha'',\alpha'''} := \frac{1}{n!} \sum_{w \in S_n} \chi^{\alpha'}(w) \chi^{\alpha''}(w) \chi^{\alpha'''}(w)$ is the Kronecker coefficient of partitions $\alpha', \alpha'', \alpha'''$ of $n$ (here $\chi^{\alpha'}(w)$ is the value at $w \InnaA{ \in S_n}$ of the character of the irreducible representation of $S_n$ corresponding to \InnaA{the Young diagram $\alpha'$).}

\item $\overline{\gamma}^{\lambda}_{\T, \mu}$ is the reduced Kronecker coefficient (equals $\gamma^{\tilde{\lambda}(n)}_{\tilde{\T}(n), \tilde{\mu}(n)}$ for $n>>0$). A good reference for standard and reduced Kronecker coefficients is \cite{BOR}.
\end{itemize}
 
\begin{proof}
 We define the map $\mathcal{F}$ (denoted by $ch$ in \cite[Chapter I, Par. 7]{Mac}) from conjugation-invariant functions on $S_n$ to symmetric functions in countably many variables by putting $$\mathcal{F}(f) : =\frac{1}{n!} \sum_{ w \in S_n} f(w) \prod_{j \geq 1} (\sum_{i} x_i^{\rho_j(w)}) $$
where $\rho(w) = (\rho_1, \rho_2, ...)$ is the cycle-type of $w$. 

Denoting $$p_{\rho} := \prod_{j \geq 1} (\sum_i x_i^{\rho_j}), m(\rho)_i \InnaB{:=\text{ number of parts of size }} i \text{ in } \rho, z_{\rho}:= \prod_{i\geq 1} i^{m(\rho)_i} m(\rho)_i!$$ 
\InnaE{we obtain}: $$\mathcal{F}(f) : = \sum_{\rho \vdash n} \frac{1}{z_{\rho}} f(\rho) p_{\rho}  $$
(here $f(\rho)$ is the value of $f$ on the conjugacy class of $S_n$ consisting of permutations of cycle-type $\rho$).

Let $V$ be a representation of $S_n$. Denote by $\chi^{V}$ the character of $V$, and by abuse of notation, $\mathcal{F}(V) :=  \mathcal{F}(\chi^{V})$. If $V = \bigoplus_{j} V_j$ is a $\bZ$-graded representation of $S_n$, then put $\mathcal{F}_q(V) :=  \sum_{j \in \bZ} \mathcal{F}(\chi^{V_j})q^j$.

We have: $\mathcal{F}(\chi^{\alpha}) = s_{\alpha}$ for partition $\alpha$ of $n$, and $\mathcal{F}(\chi^{V'}\chi^{V''}) =: \mathcal{F}(V') \ast \mathcal{F}(V'')$ ($\chi^{V'}\chi^{V''}$ is the character of the representation $\alpha' \otimes \alpha''$ of $S_n$).

Denote by $g_w$ the action of $w \in S_n$ on the $n$-dimensional complex vector space $\bC^n$, $g_w$ given by the permutation matrix corresponding to $w$.

By MacMahon's Master theorem (see \cite[Lemma 3.28]{EM}; the proof relies on an argument similar to the one used in \ref{ssec:char_Verma_object}), we have:

\begin{align*}
 &\mathcal{F}_q(\bC[x_1,...,x_n]) = \frac{1}{n!} \sum_{ w \in S_n} \sum_{k \geq 0} tr(Sym^k(g_w))q^k p_{\rho(w)}  = \frac{1}{n!} \sum_{ w \in S_n}  \frac{1}{det(1-q g_w)} p_{\rho(w)}= \\
&=\frac{1}{n!} \sum_{ w \in S_n} \prod_{1 \leq k \leq \ell(\rho(w))} \frac{1}{(1-q^{\rho(w)_k})}p_{\rho(w)} =  \sum_{\rho \vdash n} \frac{1}{z_{\rho}}\prod_{1 \leq k \leq \ell(\rho)} \frac{1}{(1-q^{\rho_k})}p_{\rho}
\end{align*}

Similarly,
 
\begin{align*}
 &\mathcal{F}_q(\bC[x_1,...,x_n]\otimes \tilde{\T}(n))) = \sum_{\rho \vdash n} \frac{1}{z_{\rho}}\prod_{1 \leq k \leq \ell(\rho)} \frac{1}{(1-q^{\rho_k})} p_{\rho} \chi^{\tilde{\T}(n)}(\rho) =\\
& = \sum_{\alpha, \rho \vdash n} \frac{1}{z_{\rho}}\prod_{1 \leq k \leq \ell(\rho)} \frac{1}{(1-q^{\rho_k})} \chi^{\alpha}(\rho) \chi^{\tilde{\T}(n)}(\rho) s_{\alpha}
\end{align*}

Thus

\begin{align*}
 &ch_q(\Hom_{S_n}(\tilde{\mu}(n), \bC[x_1,...,x_n] \otimes \tilde{\T}(n))) = \text{ coefficient of } s_{\tilde{\mu}(n)} \text{ in } \mathcal{F}_q(\bC[x_1,...,x_n]\otimes \tilde{\T}(n))) = \\
& = \sum_{\rho \vdash n} \frac{1}{z_{\rho}}\prod_{1 \leq k \leq \ell(\rho)} \frac{1}{(1-q^{\rho_k})} \chi^{\tilde{\mu}(n)}(\rho) \chi^{\tilde{\T}(n)}(\rho) = (s_{\tilde{\T}(n)} \ast s_{\tilde{\mu}(n)})(1, q, q^2, ...)
\end{align*}

The last equality holds by \cite[Chapter 6, par. 8., p.363]{Mac}.                                                                                                                                                                                 
\end{proof}

\begin{corollary}
Taking $\hhh$ to be the reflection representation of $S_n$, \InnaE{we obtain}: 
\begin{align*}
ch_q(\Hom_{S_n}(\tilde{\mu}(n), S\hhh \otimes \tilde{\T}(n))) = (1-q)(s_{\tilde{\T}(n)} \ast s_{\tilde{\mu}(n)})(1, q, q^2, ...) = (1-q)\sum_{\lambda} \overline{\gamma}^{\lambda}_{\T, \mu} s_{\tilde{\lambda}(n)}(1, q, q^2, ...)
\end{align*}
\end{corollary}
\begin{proof}
This follows directly from the fact that $\bC[x_1,...,x_n] = S \fh$, where $\fh = \hhh \oplus \bC$ is the permutation representation of $S_n$, and so $S^m \fh = \bigoplus_{0 \leq j \leq m} S^j \hhh$.
\end{proof}

\begin{corollary}\label{cor:char_hom_simple_to_Verma}
$$ch_q(\Hom_{S_{\nu}}(X_{\mu}, S\hhh \otimes X_{\T})) = (1-q)\sum_{\lambda \InnaE{\in \mathcal{P}} } \overline{\gamma}^{\lambda}_{\T, \mu} \overline{s}_{\lambda}(1, q, q^2, ...)$$
where $\overline{\gamma}^{\lambda}_{\T, \mu}$ is the reduced Kronecker coefficient, and $$\overline{s}_{\lambda}(1, q, q^2, ...) := \frac{q^{\abs{\lambda}} s_{\lambda}(1, q, q^2,...)}{\prod_{j \geq 1} (1-q^j)}$$
\end{corollary}

\begin{proof}
By the structure of Deligne's category described in Section \ref{sec:Del_cat}, if there exist integers $k, N$ such that for any $n \geq N$, $\dim \Hom_{S_n}( \tilde{\mu}(n), S^m \hhh) = k$ (here for each $n$, $\hhh$ is the reflection representation of $S_n$), then $\dim \Hom_{S_{\nu}}( X_{\mu}, S^m \hhh) = k$ in Deligne's category $\InnaE{\underline{\mathrm{Rep}}}(S_{\nu})$ as well. 

Thus the power series in $ch_q(\Hom_{S_{\nu}}(X_{\mu}, S\hhh \otimes X_{\T}))$ is the $q$-adic limit, as $n$ tends to infinity, of $ch_q(\Hom_{S_n}(\tilde{\mu}(n), S\hhh \otimes \tilde{\T}(n)))$; recall that, by definition, a sequence $\{P^{(n)}(q)\}_n, P^{(n)}(q) = \sum_{k \geq 0} a^{(n)}_k q^k$ of formal power series in $q$ converges, in the $q$-adic metric, to the formal power series $P(q)$ if
for any $k \geq 0$ there exists $N_k \in \bZ_{>0}$ such that for any $n > N_k$, $q^{k+1} \mid (P(q) - P^{(n)}(q))$. 

By \cite[Chapter I, Par. 5, Example 1]{Mac}, $$s_{\lambda}(1, q, q^2, ...) = \frac{q^{\sum_{i \geq 1} i\lambda_i - \abs{\lambda}}}{\prod_{x \in \lambda} (1-q^{h(x)})}$$ where $h(x)$ is the size of the hook in $\lambda$ with vertex $x$, so
$$s_{\tilde{\lambda}(n)}(1, q, q^2, ...) = \frac{q^{\sum_i i\lambda_i}}{\prod_{x \in \lambda} (1-q^{h(x)})} \cdot \prod_{1 \leq j \leq n-\abs{\InnaA{\lambda}}} \frac{1}{1-q^{n-\abs{\InnaA{\lambda}}-j +1+\InnaA{\lambda}\check{}_j}} $$

The $q$-adic limit of $s_{\tilde{\lambda}(n)}(1, q, q^2, ...)$, as $n$ tends to infinity, is then
$$ \overline{s}_{\lambda}(1, q, q^2, ...):=\frac{q^{\sum_i i\lambda_i}}{\prod_{x \in \lambda} (1-q^{h(x)})} \cdot  \frac{1}{\prod_{j \geq 1} (1-q^{j})} = \frac{q^{\abs{\lambda}} s_{\lambda}(1, q, q^2,...)}{\prod_{j \geq 1} (1-q^j)}$$

This completes the proof of the statement.
\end{proof}

\begin{example}\label{ex:char_poly_repr}
For $\T = \emptyset$, \InnaE{we obtain}:
 \begin{align*}
&ch_q(\Hom_{S_{\nu}}(X_{\mu}, S\hhh)) = (1-q)\overline{s}_{\mu}(1, q, q^2, ...) = \frac{q^{\abs{\mu}} s_{\mu}(1, q, q^2,...)}{\prod_{j \geq 2} (1-q^j)} =\\
& = \frac{q^{\sum_{i \geq 1} \mu_i i}}{\prod_{x \in \mu} (1-q^{h(x)}) \prod_{j \geq 2} (1-q^j)}
 \end{align*}
\end{example}

\subsection{Characters of simple objects: generic cases }\label{ssec:char_simple_obj_generic_cases}                               

\subsubsection{} 
Fix a Young diagram $\T$. Recall that for each Young diagram $\mu \neq \T$ and $m \in \bZ_{+}$, there is a set of points $(c', \nu)$ for which $M_{c, \nu}(\mu)$ maps to $M_{c, \nu}(\T)$ in degree $m$. These sets are either straight lines or finite sets. 

The union of these sets is the reducibility locus of $M_{c, \nu}(\T)$, denoted by $B_{\T}$. Outside this reducibility locus $L_{c, \nu}(\T)=M_{c, \nu}(\T)$, and the character of this object is given by the formula 

\begin{equation}\label{char_simple_obj_form_1}
  ch L(\T) = ch M(\T) = \frac{t^{h_{c, \nu}(\T)}X_{\T}}{\sum_{n \geq 0} (-1)^{n} \;\; \Lambda^n \hhh \;\; t^n}
\end{equation}

This is the most generic case.

\subsubsection{}\label{sssec:one_sing_subobj_char}
 The next most generic case is a generic point on a line $\mathcal{L}_{\T, \mu, m} \subset \mathcal{B}_{\mu, \T}$.



Fix a pair $s,r$ such that $s, r \in \bZ, r\neq 0, s \geq 0$, and consider a generic point $(c', \nu)$ on the line $c' = \frac{\nu - s}{r}$.
 
Let $\T$ be a Young diagram such that $s \in C_{\T}$ (this exactly means that $M(\T)$ is actually reducible at this point $(c, \nu)$), and denote $\mu := \Gamma(\T, s, \pm 1)$, where the sign equals $\sign(r)$.

Then there are only two distinct Verma objects which map non-trivially into $M(\T)$: $M(\mu)$ and $M(\T)$ itself (see Subsection \ref{ssec:prop_lines_L_tau_mu_m}). Both map uniquely (up to scalar multiple) into $M(\T)$.

\begin{lemma}
 The image of $M(\mu)$ in $M(\T)$ is a simple $\InnaE{\underline{\co}}_{\text{  } c,\nu}$-object.

\end{lemma}

\begin{proof}
Assume $X_{\lambda}$ is a simple singular $\InnaE{\underline{\mathrm{Rep}}}(S_{\nu})$-subobject in $\Im(M(\mu))$, and $X_{\lambda} \subset M(\T)$ lies in degree $m'$. 
Then the action of $\brh$ on $M(\T)$ and $J(\T)/ \Im(M(\mu))$ gives (see Proposition \ref{main_for_prop}):

\begin{align*}
 & c' = \frac{\nu -s}{r} \\
 &\frac{\abs{\T}^2-\abs{\lambda}^2 -(\abs{\T} - \abs{\lambda} ) }{2} +ct(\T) - ct(\lambda)  = c'm' + (\abs{\T} - \abs{\lambda})\nu
\end{align*}

By assumption, \InnaE{we obtain}: $\mu = \lambda$, which implies $m=m'$, and since $X_{\lambda}$ is a simple singular $\InnaE{\underline{\mathrm{Rep}}}(S_{\nu})$-subobject in $\Im(M(\mu))$, this means that $X_{\lambda}$ coincides with the lowest weight in $\Im(M(\mu))$.

Thus $\Im(M(\mu))$ is simple.
\end{proof}

We now want to show that there is a short exact sequence: $$ 0 \longrightarrow L(\mu) \longrightarrow M(\T) \longrightarrow L(\T) \longrightarrow 0$$

This is equivalent to saying that $\Im(M(\mu)) = J(\T)$. 

\begin{lemma}
 $M(\T)/ \Im(M(\mu))$ is a simple $\InnaE{\underline{\co}}_{\text{  } c,\nu}$-object.

\end{lemma}
\begin{proof}

Assume $X_{\lambda}$ is a simple singular $\InnaE{\underline{\mathrm{Rep}}}(S_{\nu})$-subobject in $M(\T)/ \Im(M(\mu))$, and its preimage $X_{\lambda} \subset M(\T)$ lies in degree $m'$. 

Then the action of $\brh$ on $M(\T)$ and $J(\T)/ \Im(M(\mu))$ gives (see Proposition \ref{main_for_prop}):

\begin{align*}
 & c' = \frac{\nu -s}{r} \\
 &\frac{\abs{\T}^2-\abs{\lambda}^2 -(\abs{\T} - \abs{\lambda} ) }{2} +ct(\T) - ct(\lambda)  = c'm' + (\abs{\T} - \abs{\lambda})\nu
\end{align*}

We now need to consider separately the case when $\lambda =\mu$. 
Indeed, in that case the above equations mean that $m'=m$, and thus we have $X:=X_{\mu} \subset M(\T)$ (not a singular subobject) which lies in degree $m'$ and whose image in $M(\T)/ \Im(M(\mu))$ is not zero and singular. 

The former implies that $y_{M(\T)}(\hhh \otimes X) \neq 0$ is a direct sum of simple $\InnaE{\underline{\mathrm{Rep}}}(S_{\nu})$-objects lying in degree $m-1$ of $M(\T)$. 

The fact that image of $X$ in $M(\T)/ \Im(M(\mu))$ is singular would mean that $$y_{M(\T)}(\hhh \otimes X) \subset \Im(M(\mu))$$ 

But this leads to a contradiction, since $\Im(M(\mu))$ lies in degrees strictly higher than $m-1$ of $M(\T)$.

So $\lambda \neq \mu$, and this contradicts our assumption.
\end{proof}

 Thus for generic $(c', \nu)$ such that $c' = \frac{\nu-s}{r}$, we have a short exact sequence: $$ 0 \longrightarrow L(\mu) \longrightarrow M(\T) \longrightarrow L(\T) \longrightarrow 0$$
 
 and a long exact sequence
 
 \begin{align*}
 &... \rightarrow M(\Gamma(\T, s, \pm l)) \longrightarrow M(\Gamma(\T, s, \pm (l-1))) \rightarrow ...
 \rightarrow M(\Gamma(\T, s, \pm 1)) \longrightarrow M(\T) \longrightarrow L(\T) \rightarrow 0
\end{align*}

As before, the sign corresponds to the sign of $r$, and this sequence ends (on the left) with $M(\mathbf{rec}( 0, \mathbf{core}_{(\nu-s)}(\T)))= M(\Gamma(\T, s, -\T\check{}_{j_s}))$ if $r<0$. 
 
 This allows us to compute the character of $L(\T)$ by Euler formula:
 
 
 $$ch L(\T) = \sum_{l \in \bZ_+, \text{ and } l\leq \T\check{}_{j_s} \text{ if } r<0} (-1)^l ch M(\Gamma(\T, s, \pm l))  $$
 
\InnaA{
Now, 
\begin{align*}
&h_{c, \nu}(\Gamma(\T, s, \pm l)) - h_{c, \nu}(\T) = \\
&=\sum_{1 \leq l' \leq l} h_{c, \nu}(\Gamma(\T, s, \pm l')) - h_{c, \nu}(\Gamma(\T, s, \pm (l'-1))) =\\
&= \sum_{1 \leq l' \leq l} m_{l'} = \sum_{1 \leq l' \leq l} c(\nu-s) \left(\abs{\Gamma(\T, s, \pm l')} - \abs{\Gamma(\T, s, \pm (l'-1))} \right)= \\
&=c(\nu-s)\left(\abs{\Gamma(\T, s, \pm l)} - \abs{\T}\right)= c(\nu-s)(j_s - k_{s,l} \pm l)
\end{align*}

Here we denote by $m_{l'}$ the degree of $M(\Gamma(\T, s, \pm (l'-1)))$ in which lies the image of the lowest weight of $M(\Gamma(\T, s, \pm l'))$, and use Lemma \ref{identity_for_s_and_diagrams}.

}

 Using Subsection \ref{ssec:char_Verma_object} and the fact that $$h_{c, \nu}(\Gamma(\T, s, \pm l)) - h_{c, \nu}(\T) = c(\nu-s)(j_s - k_{s,l} \pm l)$$ \InnaE{we obtain}:
 
 %
 \begin{equation}\label{char_simple_obj_form_2} 
  ch L(\T) =  \frac{t^{h_{c, \nu}(\T)} \left( \sum_{l \in \bZ_+, \text{ and } l\leq \T\check{}_{j_s} \text{ if } r<0} (-1)^l  X_{\Gamma(\T, s, \pm l)} t^{c(\nu-s)(j_s - k_{s,l} \pm l)} \right) }{\sum_{n \geq 0} (-1)^{n} \;\; \Lambda^n \hhh \;\; t^n} 
 \end{equation}

As before, the sign corresponds to the sign of $r$.


\begin{example}\label{example_cnu_is_1}
 Let $c\nu =1, Re(c)>0$. 

If $X_\lambda$ lies in degree $m'$ of $M_{\InnaE{c, \nu}}(\pi^n)$ as a singular $\InnaE{\underline{\mathrm{Rep}}}(S_{\nu})$-subobject ($\pi^n$ is a column diagram with $n$ cells), then Equation \eqref{main_for_rewritten} gives us:
$$ m = (\abs{\lambda} -n)c\nu - f(\lambda)c = (\abs{\lambda} -n) - f(\lambda)c $$
and Pieri's rule (Proposition \ref{Pieri}) implies that $m \geq \abs{\abs{\lambda} -n}$. Since $m, \abs{\lambda}, f(\lambda) \in \bZ_{\InnaA{+}}$, \InnaE{we obtain} that $c \in \bQ_{>0}$.

We also have: $f(\lambda) \geq 0$, with equality iff $\lambda$ is a column diagram (see Lemma \ref{lemma_values_f}). So 
$$ \abs{\abs{\lambda} -n} \leq m = (\abs{\lambda} -n) - f(\lambda)c \leq \abs{\lambda} -n$$
which means that $f(\lambda) =0, \abs{\lambda} \geq n$, i.e. $\lambda = \pi^k$ for some $k \geq n$ (in fact, $k >n$), and $X_{\pi^k}$ lies in degree $m=k-n$ of $M_{\InnaE{c, \nu}}(\pi^n)$.

But by Pieri's rule (Proposition \ref{Pieri}), for $k>1$, $\hhh^{\otimes k}$ contains only one copy of $X_{\pi^k} \cong \Lambda^k \hhh$, and it does not lie inside $S^k \hhh$. So $S^k \hhh$ doesn't contain subobjects isomorphic to $X_{\pi^k}$.

So $M_{\InnaE{c, \nu}}(\pi^n)$ has only one non-trivial singular $\InnaE{\underline{\mathrm{Rep}}}(S_{\nu})$-subobject: $X_{\pi^{n+1}}$, which lies in degree $1$ of $M_{\InnaE{c, \nu}}(\pi^n)$. 

 In this case, we can compute the character of $L_{\InnaE{c, \nu}}(\pi^n)$ as in Equation \eqref{char_simple_obj_form_2}. We assumed that $c\nu=1$, so $h_{c,\nu}(\pi^n) = \frac{(\nu-1)(1-c\nu)}{2} + n = n$, hence

 $$ch L_{\InnaE{c, \nu}}(\pi^n) =  \frac{ \left( \sum_{l \in \bZ_+} (-1)^l  X_{\pi^{n+l}} t^{n+l}\right) }{\sum_{l \in \bZ_+} (-1)^{l} \;\;\Lambda^l \hhh \;\; t^l}$$

But $\Lambda^{n+l} \hhh \cong X_{\pi^{n+l}}$ in $\InnaE{\underline{\mathrm{Rep}}}(S_{\nu})$, so 
$$ch L_{\InnaE{c, \nu}}(\pi^n) =  \frac{  \sum_{l \in \bZ_+} (-1)^l  \;\; \Lambda^{n+l}\hhh \;\; t^{n+l} }{\sum_{l \in \bZ_+} (-1)^{l} \;\;\Lambda^l \hhh \;\; t^l} $$

Note that $L_{\InnaE{c, \nu}}(\emptyset) \cong X_{\emptyset} =\triv$ as $\InnaE{\underline{\mathrm{Rep}}}(S_{\nu})$ (ind)-objects, with maps $x_{L_{\InnaE{c, \nu}}(\emptyset)}, y_{L_{\InnaE{c, \nu}}(\emptyset)} =0$.
\end{example}

\begin{example}\label{exm_char_cnu_k}
 Similarly, for a generic point $(c,\nu)$ on the line $c\nu=k$, $k \in \bZ_{>0}$, we have: $s=0, r=k$ and so 
$$ \Gamma(\pi^n, 0, \pm l) = X_{\pi^{n+l}} $$
\begin{align*}
 ch L(\pi^n) =  t^{kn -\frac{(\nu-1)(k-1)}{2}} \cdot \frac{ \sum_{l \in \bZ_+} (-1)^l \;\; \Lambda^{n+l} \hhh \;\; t^{kl}  }{\sum_{l \in \bZ_+} (-1)^{l} \;\; \Lambda^l \hhh \;\; t^l}
\end{align*}

and in particular 
\begin{align*}
 ch L(\emptyset) =  t^{-\frac{(\nu-1)(k-1)}{2}} \cdot \frac{ \sum_{l \in \bZ_+} (-1)^l  \Lambda^l \hhh t^{kl}}{\sum_{l \in \bZ_+} (-1)^{l}\Lambda^l \hhh t^l} = t^{-\frac{(\nu-1)(k-1)}{2}} \cdot \frac{\mathbf{ch}_t S\hhh}{\mathbf{ch}_{t^k} S\hhh}
\end{align*}

(see \cite{BEG} or \cite[Corollary 3.50]{EM} for the corresponding result for $\co(H_c(n))$). 
\end{example}

As in the case when $k=1$, we can also compute the character of $L_{c, \nu}(\emptyset)$ explicitly:


\begin{proposition}\label{prop:char_irr_repr_generic}
 Let $(c,\nu)$ be a generic point on the line $c\nu=k$, where $k \in \bZ_{>0}$ is fixed. Then 
 $$ch_q \Hom_{S_{\nu}}(X_{\mu}, L_{\InnaA{c, \nu}}(\emptyset)) =  \frac{q^{\abs{\mu}} s_{\mu}(1, q, q^2,..., q^{k-2})}{\prod_{2 \leq j \leq k} (1-q^j)}$$
\end{proposition}

\begin{proof}
From the exact sequence 
\begin{align*}
 &... \longrightarrow M( \pi^l) \longrightarrow M(\pi^{l-1}) \longrightarrow  ...
 \longrightarrow M(\pi^1 ) \longrightarrow M(\emptyset) \longrightarrow L(\emptyset) \longrightarrow 0
\end{align*} 
and the character formula Corollary \ref{cor:char_hom_simple_to_Verma} for $\Hom_{\InnaE{\underline{\mathrm{Rep}}}(S_{\nu})}(X_{\mu}, S\hhh \otimes X_{\pi^l})$, we have: 
$$ch_q \Hom_{S_{\nu}}(X_{\mu}, L(\emptyset)) = (1-q)\sum_{l \geq 0} \sum_{\lambda \InnaE{\in \mathcal{P}}} (-1)^l q^{kl} \overline{\gamma}^{\lambda}_{\Lambda^l \hhh, \mu} \frac{q^{\abs{\lambda}} s_{\lambda}(1, q, q^2,...)}{\prod_{j \geq 1} (1-q^j)}$$

One immediately sees that this is the $q$-adic limit, as $n$ tends to infinity, of the sequence of formal power series 
$$ P_n(q) := (1-q)\sum_{l \geq 0} \sum_{\lambda \vdash n} (-1)^l q^{kl} \gamma^{\lambda}_{\Lambda^l \hhh, \tilde{\mu}(n)}  s_{\lambda}(1, q, q^2,...)$$ where $\hhh$ is the reflection representation of $S_n$.

We now give a formula for $P_n(q)$. First, recall that (see \cite[Chapter I, Par. 7]{Mac}) $$\gamma^{\lambda}_{\Lambda^l \hhh, \tilde{\mu}(n)}  = \frac{1}{n!}\sum_{w \in S_n} \chi^{\Lambda^l \hhh}(w) \chi^{\tilde{\mu}(n)}(w) \chi^{\lambda}(w)$$
(here $\chi^{\beta}(w) = tr\mid_{\beta}(w)$ is the value at $w$ of the character of $S_n$ corresponding to the irreducible representation $\beta$). 

Now, from the exact sequence 
\InnaB{
$$ 0 \rightarrow S\hhh \otimes \Lambda^{n-1} \hhh \longrightarrow...\longrightarrow S\hhh \otimes \Lambda^m \hhh \longrightarrow  ...\longrightarrow S\hhh  \otimes \hhh \longrightarrow S\hhh \longrightarrow \bC \rightarrow 0$$}
we have:

\begin{align*}
&\sum_{l \geq 0} (-1)^l q^{kl} \chi^{\Lambda^l \hhh}(w) = \frac{1}{\sum_{l \geq 0} q^{kl} tr \mid_{S^l \hhh} (w)} =...\\
&...= \frac{1}{\sum_{l \geq 0} q^{kl} (1-q^k) tr \mid_{S^l \fh} (w)}= \frac{1}{1-q^k} \prod_{1 \leq j \leq \ell(\rho(w))} \left(1-q^{k\rho(w)_j}\right)
\end{align*}
(here $\rho(w)$ is the cycle type of $w$).

On the other hand, we have \InnaA{(see e.g. \cite[Par. 1, (7.7), (7.8), p. 114]{Mac})}:
$$ \sum_{\lambda \vdash n} \chi^{\lambda}(w) s_{\lambda}(1, q, q^2,...) = \prod_{1 \leq j \leq \ell(\rho(w))} \frac{1}{1-q^{\rho(w)_j}}$$

Thus we obtain:

\begin{align*}
&P_n(q) = (1-q)\sum_{l \geq 0} \sum_{\lambda \vdash n} (-1)^l q^{kl} \gamma^{\lambda}_{\Lambda^l \hhh, \tilde{\mu}(n)} s_{\lambda}(1, q, q^2,...) =\\
& = \frac{1-q}{1-q^k} \frac{1}{n!}\sum_{w \in S_n} \chi^{\tilde{\mu}(n)}(w) \prod_{1 \leq j \leq \ell(\rho(w))}  \frac{1-q^{k\rho(w)_j}}{1-q^{\rho(w)_j}} =\\
& = \frac{1-q}{1-q^k} \frac{1}{n!}\sum_{w \in S_n} \chi^{\tilde{\mu}(n)}(w) \prod_{1 \leq j \leq \ell(\rho(w))}  \left(1+q^{\InnaA{\rho(w)_j}}+q^{\InnaA{2\rho(w)_j}}+...+q^{k\rho(w)_j}\right)  = \frac{1-q}{1-q^k} s_{\tilde{\mu}(n)}(1,q,..., q^{k-1})
\end{align*}

Taking the $q$-adic limit when $n \rightarrow \infty$, we obtain: the $q$-adic limit of $s_{\tilde{\mu}(n)}(1,q,..., q^{k-1})$ is $\frac{q^{\abs{\mu}} s_{\mu}(1, q, q^2,..., q^{k-2})}{\prod_{1 \leq j \leq k-1} (1-q^j)}$, and thus
$$ch_q \Hom_{S_{\nu}}(X_{\mu}, L(\emptyset)) =  \frac{q^{\abs{\mu}} s_{\mu}(1, q, q^2,..., q^{k-2})}{\prod_{2 \leq j \leq k} (1-q^j)}$$

\end{proof}
\begin{remark}
Note that $s_{\mu}(1, q, q^2,..., q^{k-2}) =0$ if $l(\mu) \geq k$, so $\Hom_{S_{\nu}}(X_{\mu}, L(\emptyset)) =0$ whenever $l(\mu) \geq k$.
\end{remark}

\section{Length of Verma objects}\label{sec:length_of_Verma_obj}

In this section, we will discuss the set of points $(c,\nu)$ such that the Verma object $M_{c, \nu}(\T)$ is of finite length. We will prove the following theorem:

\begin{theorem}\label{thrm:length_of_Verma_obj}
For any $c \notin \bQ_{<0}$ \InnaB{and any Young diagram $\T$}, the Verma object $M_{c, \nu}(\T)$ is of finite length.
\end{theorem}

 By Subsection \ref{ssec:morphs_between_2_objects}, $M_{c, \nu}(\T)$ has infinite length whenever $(c, \nu)$ is the intersection point of infinitely many curves of form $\frac{1}{c} = \frac{\nu -s}{r}$ where $s \in C_{\T}, r\in \bZ \setminus \{0\}$. 

 For instance, when $\T =\emptyset$, we have:
 
 \begin{lemma}
  For any $c \in \bQ_{<0}$ there exists $\nu \in \bQ$ such that the Verma object $M_{c, \nu}(\emptyset)$ has infinite length.
 \end{lemma}
 \begin{proof}
 Indeed, for any $c \in \bQ_{<0}$ and any $\nu \in \bQ$ such that $den(\nu) \mid num(c),  den(c) \mid (num(\nu) +den(\nu)) $ (i.e. $num(\nu)den(\nu)^{-1} \equiv -1 \mod den(c)$), we have: $$\{s \in \bZ_+ \mid c(\nu-s) \in \bZ_{>0}\}= \{s \in \bZ_+ \mid s \equiv -1 \mod den(c), s>\nu  \}$$ 
 Now, recall that $C_{\emptyset} = \bZ_+$. So any $s$ such that $s \equiv -1 \mod den(c), s>\nu$ gives us a non-trivial morphism $M_{c, \nu}(\T^s) \longrightarrow M_{c, \nu}(\emptyset)$. The image of each of these morphisms has a simple object $L_{c, \nu}(\T^s)$ as a quotient, and since all the weights $\T^s$ are different, each of these simple objects $L_{c, \nu}(\T^s)$ contributes to the length of $M_{c, \nu}(\emptyset)$. This proves that $M_{c, \nu}(\emptyset)$ has infinite length.
 \end{proof}
 
The proof consists of two parts: proving that whenever $c \InnaB{\in \bR_{>0}}$, the Verma object $M_{c, \nu}(\T)$ is of finite length (we will prove this statement in Subsection \ref{ssec:length_of_Verma_obj_pos_c}), and proving that whenever $c\notin \bQ$, the Verma object $M_{c, \nu}(\T)$ is of finite length, too (we will prove that in Subsection \ref{ssec:length_of_Verma_obj_irrat_c}).

\begin{remark}
 In the classical case, all modules in $\co(H_c(n))$ have finite length. See \cite[Corollary 3.26]{EM}.
\end{remark}

\subsection{Bounds on the graded space \texorpdfstring{$Hom_{\InnaE{\underline{\mathrm{Rep}}}(S_{\nu})}(X_\mu, S\hhh \otimes X_\T)$}{Hom(X_mu, M(Tau))}}\label{ssec:min_degr_obj_in_Verma}
In this subsection, we give a bound on the least degree in which a simple $\InnaE{\underline{\mathrm{Rep}}}(S_{\nu})$ object $X_{\mu}$ can lie in the $\InnaE{\underline{\mathrm{Rep}}}(S_{\nu})$ \InnaE{\rm{ind}}-object $S\hhh \otimes X_{\T}$. 

Recall that we have the following corollary of Lemma \ref{lem:char_classical_verma}. 

\begin{corollary}
 For $\T =\emptyset$, we have: $$ch_q(\Hom_{S_n}(\tilde{\mu}(n), \bC[x_1,...,x_n])) =  s_{\tilde{\mu}(n)}(1, q, q^2, ...)$$
\end{corollary}

By the definition of a Schur symmetric function, $s_{\tilde{\mu}(n)}(1, q, q^2, ...)$ is divisible by $q^{\sum_k \mu_k k}$. So for any $n>>0$, the minimal degree of $S\fh = \bC[x_1,...,x_n]$ in which $\tilde{\mu}(n)$ appears is greater than or equal $\sum_k \mu_k k$. By the description of Deligne's category $\InnaE{\underline{\mathrm{Rep}}}(S_{\nu})$, this implies:

\begin{corollary}\label{least_grade_mu_in_poly_repr}
 For any $\nu$, we have: the minimal degree of $S\fh$ (and thus of $S\hhh$ as well) in which $X_{\mu}$ appears is greater than or equal \InnaE{to} $\sum_k \mu_k k$.
\end{corollary}

For the convenience of the reader, we will use the following notation:
\begin{notation}
 \InnaE{For a Young diagram $\mu$, the sum $\sum_k \mu_k k$ will be denoted by $n(\mu)$ (c.f. \cite{Mac}).}
\end{notation}

\InnaE{The function $n(\cdot): \mathcal{P} \rightarrow \bZ_+$ has the following trivial property (see also \cite[Chapter I]{Mac}):
\begin{lemma}\label{lem:func_n_max_min}
 For any Young diagram $\mu$, we have:
 $$\abs{\mu} \leq n(\mu) \leq \frac{\abs{\mu}(\abs{\mu}+1)}{2}$$
\end{lemma}
\begin{proof}
 It is easy to see that when we consider the function $n(\cdot)$ on the set of Young diagrams of fixed size $k$, the maximum
  is obtained for the Young diagram $\pi^k$ (a column of length $k$), and it is $n(\pi^k) = \frac{\abs{k}(\abs{k}+1)}{2}$, while the minimum is obtained for $\tau^{k-1}$ (a row of length $k$), and it is $n(\tau^{k-1})=k$.
\end{proof}
}

We now generalize this result to the following:

\begin{proposition}\label{prop:least_grade_mu_in_verma_obj}
 Consider the $\bZ_+$-graded complex vector space $\Hom_{\InnaE{\underline{\mathrm{Rep}}}(S_{\nu})}(X_{\mu}, S\hhh \otimes X_{\T})$ (the grading inherited from $S \hhh$). Let $m$ be a grade of this space containing a non-zero morphism. Then $$m \geq \InnaE{n(\mu)} - \ell(\mu)\abs{\T} - \frac{\InnaE{\abs{\T}^2 }}{2} -\frac{\abs{\T}}{2}$$
\end{proposition}

\begin{proof}
 \InnaE{Recall that any object of $\InnaE{\underline{\mathrm{Rep}}}(S_{\nu})$ is isomorphic to its dual, so} $\Hom_{\InnaE{\underline{\mathrm{Rep}}}(S_{\nu})}(X_{\mu}, S\hhh \otimes X_{\T}) \cong \Hom_{\InnaE{\underline{\mathrm{Rep}}}(S_{\nu})}(X_{\mu} \otimes X_{\T}, S\hhh )$ as $\bZ_+$-graded vector spaces. 
 
 \InnaE{Assume $m \geq 0$ is such that $\Hom_{\InnaE{\underline{\mathrm{Rep}}}(S_{\nu})}(X_{\mu} \otimes X_{\T}, S\hhh ) \neq 0$}. This means that $X_{\mu} \otimes X_{\T}$ and $ S^m\hhh$ have a common composition factor. Now, $X_{\T} \subset \hhh^{\otimes \abs{\T}}$, so $X_{\mu} \otimes \hhh^{\otimes \abs{\T}}$ and $S^m\hhh$ have a common composition factor.

A simple subobject $X_{\lambda}$ of $X_{\mu} \otimes \hhh^{\otimes \abs{\T}}$ satisfies the condition arising from Pieri's rule (Proposition \ref{Pieri}):

\begin{cond}\label{cond_2_diag_differ_by_k}
 \InnaE{The Young diagram} $\lambda$ can be obtained from $\mu$ by performing at most $\abs{\T}$ steps, each consisting of either adding a cell, deleting a cell or moving a cell (i.e. deleting a cell and then adding a cell).
\end{cond}


We know that for some \InnaE{Young diagram $\lambda$ satisfying Condition \ref{cond_2_diag_differ_by_k}, we have: $X_{\lambda} \subset S^m\hhh$. Corollary \ref{least_grade_mu_in_poly_repr} then tells us that

$$m \geq \sum_k \lambda_k k$$

We are left with the following problem:\InnaE{

for two Young diagrams $\mu, \lam$ satisfying Condition \ref{cond_2_diag_differ_by_k}, find an upper bound for $\InnaE{n(\mu)} - \InnaE{n(\lambda)}$ in terms of $\abs{\T}, \ell(\mu)$}. 

By the Condition \ref{cond_2_diag_differ_by_k}, the difference $\InnaE{n(\mu)} - \InnaE{n(\lambda)}$ is maximal when $\lambda$ is obtained from $\mu$ by removing a sequence of $\abs{\T}$ boxes starting with bottom row and progressing upwards, from right to left. Namely, the maximum of the difference is obtained for $\lambda:=\lambda^{0}$, where $\lambda^{0}$ is constructed as follows:

let $k_0 \geq 0$ be such that $\sum_{i \geq k_0+1} \mu_i \leq \abs{\T} < \sum_{i \geq k_0} \mu_i$. Then $\lambda^{0}$ is a Young diagram of length $k_0$, defined as
$$\lambda^{0}_i:=\begin{cases} \mu_i & \text{ if } 1 \leq i \leq k_0-1 \\
\mu_{k_0} - \left( \abs{\T}  - \sum_{i \geq k+1} \mu_i \right) &  \text{ if } i=k_0 \\
\end{cases}$$}
\begin{example}
 \InnaE{Let $\abs{\T} = 6$, $\mu = (5,4,4,3,2)$. Then $k_0 = 3$, and $\lambda^{0} = (5,4,3)$:
 $$\mu= \young(\hfil\hfil\hfil\hfil\hfil,\hfil\hfil\hfil\hfil,\hfil\hfil\hfil\circ,\circ\circ\circ,\circ\circ) \mapsto \lambda^{0} = \yng(5,4,3)$$}
\end{example}

\InnaE{We now give an upper bound for $\InnaE{n(\mu)} - n(\lambda^{0})$. Writing it out explicitly, we obtain
\begin{align*}
&\InnaE{n(\mu)} - n(\lambda^{0}) = \sum_{k \geq k_0 + 1} \mu_k k + k_0 \left( \abs{\T}  - \sum_{i \geq k+1} \mu_i \right) = \\
&=\sum_{k \geq k_0 + 1} \mu_k (k - k_0) + k_0 \abs{\T}  = n(\bar{\mu}^{(k_0)}) + k_0 \abs{\T} 
\end{align*}
where $\bar{\mu}^{(k)}:= (\mu_{k_0 +1} , \mu_{k_0 +2}, ...)$ (the Young diagram obtained by removing the first $k_0$ rows of $\mu$). We know that $\abs{\bar{\mu}^{(k)}} \leq \abs{\T}$ and $k_0 \leq \ell(\mu)$, so by Lemma \ref{lem:func_n_max_min}, we have: 
$$ n(\bar{\mu}^{(k_0)}) + k_0 \abs{\T} \leq \frac{\abs{\T} (\abs{\T}+1)}{2} + \ell(\mu) \abs{\T} $$

}

%
%
\InnaE{
Thus for any two Young diagrams $\mu, \lam$ satisfying Condition \ref{cond_2_diag_differ_by_k}, we have: 
$$\InnaE{n(\mu)} - \InnaE{n(\lambda)} \leq l(\mu)\abs{\T}  + \frac{\abs{\T}(\abs{\T}+1)}{2}$$
which implies 
 \begin{align*}
 &m \geq n(\lambda) \geq \InnaE{n(\mu)} -l(\mu)\abs{\T} -  \frac{\abs{\T}(\abs{\T}+1)}{2} 
\end{align*}  }
\end{proof}

\subsection{Length of Verma objects for \texorpdfstring{$c \in \bR_{>0}$}{positive real c}}\label{ssec:length_of_Verma_obj_pos_c}

\begin{proposition}
 For any $c \in \bR_{>0}$, the Verma object $M_{c, \nu}(\T)$ is of finite length.
\end{proposition}

\begin{proof}
Fix $c, \nu$ such that $c\in \bR_{>0}$. We need to prove that $M(\T)$ has only finitely many composition factors. \InnaB{Let $\{L(\mu^{j})\}_{j \in J}$ be the composition factors of $M(\T)$.

Since $L(\mu^{j})$ is a composition factor of $M(\T)$, there exists a positive integer $m^j$ such that $h_{c, \nu}(\mu^j) =h_{c,\nu}(\T)\InnaA{+m^j}$. By Proposition \ref{main_for_prop}, to prove that $J$ is a finite set,} it is enough to prove that there exist only finitely \InnaA{many} Young diagrams $\mu$ such that $X_{\mu}$ lies in degree $m$  of $M(\T)$ and $$m = c\nu(\abs{\mu} -\abs{\T}) + c\left(\frac{\abs{\T}^2-\abs{\mu}^2 -(\abs{\T} - \abs{\mu} ) }{2} +ct(\T) - ct(\mu) \right)$$

By Proposition \ref{prop:least_grade_mu_in_verma_obj}, it is, in fact, enough to prove that 
$$c\nu(\abs{\mu} -\abs{\T}) +c\left(\frac{\abs{\T}^2-\abs{\mu}^2 -(\abs{\T} - \abs{\mu} ) }{2} +ct(\T) - ct(\mu) \right) < \InnaE{n(\mu)} - \ell(\mu)\abs{\T} - \frac{\InnaE{\abs{\T}^2 }}{2} -\frac{\abs{\T}}{2} $$

for all but finitely many Young diagrams $\mu$.

Recall that $ct(\mu) = \InnaE{n(\mu\check{})} -\InnaE{n(\mu)}$, so subtracting LHS from RHS in the above expression, \InnaE{we obtain}

\begin{align*}
 &-c\nu\abs{\mu} +c\nu\abs{\T} -cf(\T)+c\frac{\abs{\mu}^2 - \abs{\mu}}{2}  +\InnaF{c \cdot n(\mu\check{})} -\InnaF{c \cdot n(\mu)} +\\
&+\InnaE{n(\mu)} - \ell(\mu)\abs{\T} - \frac{\InnaE{\abs{\T}^2 }}{2} -\frac{\abs{\T}}{2} = c\frac{\abs{\mu}^2 - \abs{\mu}}{2} -\\
&-c\nu\abs{\mu} +\InnaF{c \cdot n(\mu\check{})} -(c-1) \InnaF{\cdot n(\mu)} - \ell(\mu)\abs{\T} - \frac{\InnaE{\abs{\T}^2 }}{2} -\frac{\abs{\T}}{2}+ c\nu\abs{\T} -cf(\T)
\end{align*}

where $f(\T):=\frac{\abs{\T}^2 -\abs{\T}}{2} +ct(\T)$.

We need to show that for all but finitely many Young diagrams $\mu$, \InnaE{the expression below}

$$ c\frac{\abs{\mu}^2 - \abs{\mu}}{2} -c\nu\abs{\mu} +\InnaF{c \cdot n(\mu\check{})} -(c-1) \InnaF{\cdot n(\mu)} - \ell(\mu)\abs{\T} - \frac{\InnaE{\abs{\T}^2 }}{2} -\frac{\abs{\T}}{2}+ c\nu\abs{\T} -cf(\T)$$
\InnaE{ is positive (keep in mind that the parameter $\nu$ and the Young diagram $\T$ remain fixed).}

Since $l(\mu) \leq \abs{\mu}$, it is in fact enough to check that \InnaE{for all Young diagrams $\mu$, the expression}
\begin{equation}\label{eq:finite_length_pos_c}
 \frac{c}{2}\abs{\mu}^2 - \InnaE{\left(\frac{c}{2}+c\nu+\abs{\T} \right)\abs{\mu}} + \InnaF{c \cdot n(\mu\check{})} -(c-1) \InnaF{\cdot n(\mu)}  
\end{equation}
\InnaE{ is bounded below by a polynomial in $\abs{\mu}$ of positive degree, and with a positive leading coefficient}.

We now have to consider two cases separately:

- If $0 < c <1$, then $ \InnaF{c \cdot n(\mu\check{})} -(c-1) \InnaF{\cdot n(\mu)} >0$, and so \InnaE{the expression in Equation \eqref{eq:finite_length_pos_c} is bounded below by the polynomial}
$$\frac{c}{2}\abs{\mu}^2 - \InnaE{\left(\frac{c}{2}+c\nu+\abs{\T} \right)\abs{\mu}}$$

- If $c \geq 1$, then using \InnaE{Lemma \ref{lem:func_n_max_min}, \InnaE{we obtain}:

 \begin{align*}
  &\frac{c}{2}\abs{\mu}^2 - \InnaE{\left(\frac{c}{2}+c\nu+\abs{\T} \right)\abs{\mu}} + \InnaF{c \cdot n(\mu\check{})} -(c-1) \InnaF{\cdot n(\mu)}   \geq\\ 
&\geq \frac{c}{2}\abs{\mu}^2 - \InnaE{\left(\frac{c}{2}+c\nu+\abs{\T} \right)\abs{\mu}}+ c\abs{\mu} - (c-1)\frac{\abs{\mu}(\abs{\mu}+1)}{2} = \\
& =  \frac{1}{2}\abs{\mu}^2 + \left(\frac{1}{2} -c\nu - \abs{\T}\right)\abs{\mu}
 \end{align*}

We conclude that when $c \geq 1$, the expression in Equation \eqref{eq:finite_length_pos_c} is bounded below by the polynomial
$$\frac{1}{2}\abs{\mu}^2 + \left(\frac{1}{2} -c\nu - \abs{\T}\right)\abs{\mu}$$
as wanted.}
 

\end{proof}

\subsection{Length of Verma objects for \texorpdfstring{$c \notin \bQ$}{irrational c}}\label{ssec:length_of_Verma_obj_irrat_c}

\begin{proposition}
For any $c \notin \bQ$, the Verma object $M_{c, \nu}(\T)$ is of finite length.

\end{proposition}

\begin{proof}
 Assume that for some $c \notin \bQ, \nu$, a Verma object $M_{c, \nu}(\T)$ has infinite length.

\InnaB{Let $\{L_{c, \nu}(\mu^{j})\}_{j \in J}$ be the composition factors of $M_{c, \nu}(\T)$.

Since $L_{c, \nu}(\mu^{j})$} is a composition factor of $M_{c, \nu}(\T)$, there exists a positive integer $m^j$ such that $h_{c, \nu}(\mu^j) =h_{c,\nu}(\T)\InnaA{+m^j}$, and $X_{\mu^j} \subset S^{m^j} \hhh \otimes \T$. Similarly to the proof of Proposition \ref{main_for_prop}, for any $j$ we have: 
$$ (\InnaE{c'}, \nu) \in \mathcal{L}_{\T, \mu^j, m^j}$$

 Since $ (\InnaE{c'}, \nu) \in \mathcal{L}_{\T, \mu^j, m^j}$ for any $j$, and $c \notin \bQ$, Subsection \ref{ssec:prop_lines_L_tau_mu_m} implies that for all $j$ the lines $\mathcal{L}_{\T, \mu^j, m^j}$ coincide; that is, the equations $$\frac{1}{c} m^j + \nu(\abs{\T} -\abs{\mu^j}) = \frac{\abs{\T}^2-\abs{\mu^j}^2 -(\abs{\T} - \abs{\mu^j} ) }{2} +ct(\T) - ct(\mu^j) $$ define the same line for all $j$.

This means that there exist constants $C, C' \in \bQ$ such that for any $j$,
\begin{align*}
& \frac{1}{\abs{\mu^j} -\abs{\T}} \left(\frac{\abs{\T}^2-\abs{\mu^j}^2 -(\abs{\T} - \abs{\mu^j} ) }{2} +ct(\T) - ct(\mu^j)\right) =C'\\ 
& \text{ and } \frac{\abs{\mu^j} -\abs{\T}}{m^j} =C 
\end{align*}

Since $X_{\mu^j} \subset S^{m^j} \hhh \otimes \T$, Pieri's rule (Proposition \ref{Pieri}) implies that $\abs{C} \leq 1$, and Proposition \ref{prop:least_grade_mu_in_verma_obj} implies that for any $j$,
$$m^j \geq \sum_{k \geq 1} (\mu^j)_k  k - \ell(\mu^j)\abs{\T} - \frac{\InnaE{\abs{\T}^2 }}{2} -\frac{\abs{\T}}{2}$$

So there exist constants $C, C'$ such that for any $j$, the following conditions hold for $\mu:= \mu^j, m: =m^j$:
\begin{cond}\label{cond:irr_c_fin_length}
\begin{align}
&\frac{1}{\abs{\mu} -\abs{\T}} \left(\frac{\abs{\mu}^2 - \abs{\mu} }{2} +ct(\mu) - f(\T)\right) =C'\\ 
&\frac{\abs{\mu} -\abs{\T}}{m} =C \\
&m^j \geq \sum_{k \geq 1} (\mu)_k  k - \ell(\mu)\abs{\T} - \frac{\InnaE{\abs{\T}^2 }}{2} -\frac{\abs{\T}}{2}
\end{align}
\end{cond}

with $f(\T) = \frac{\abs{\T}^2 - \abs{\T} }{2} +ct(\T)$ as in Subsection \ref{ssec:rmrks_main_for}. 

We will show that for \InnaE{fixed} $C, C'$ as above, Condition \ref{cond:irr_c_fin_length} can only hold for a finite number of Young diagrams $\mu$. 

First, notice that we can assume that $C>0$ (otherwise $\mu \leq \abs{\T}$ and we are done). 

\InnaE{Combining} the second and third \InnaE{requirements} of Condition \ref{cond:irr_c_fin_length} give: 
$$C^{-1}(\abs{\mu} -\abs{\T}) \geq \InnaE{n(\mu)} - \ell(\mu)\abs{\T} - \frac{\InnaE{\abs{\T}^2 }}{2} -\frac{\abs{\T}}{2}$$
i.e.
\begin{equation}\label{eq:irrat_c_eq1}
\InnaE{n(\mu)}  \leq \frac{\InnaE{\abs{\T}^2 }}{2} +\frac{\abs{\T}}{2} - C^{-1}\abs{\T}+ \ell(\mu)\abs{\T} + C^{-1}\abs{\mu} 
\end{equation}

The first \InnaE{requirement} of Condition \ref{cond:irr_c_fin_length} gives:
\begin{equation}\label{eq:irrat_c_eq2}
 \frac{\abs{\mu}^2 - \abs{\mu} }{2} +ct(\mu) -C'\abs{\mu} = f(\T)-C'\abs{\T}
\end{equation}
 
Now, $ct(\mu) = \InnaE{n(\mu\check{})} - \InnaE{n(\mu)}$, so \InnaE{combining Inequality \eqref{eq:irrat_c_eq1} and Equation \eqref{eq:irrat_c_eq2} gives}
\begin{align*}
 &\frac{\abs{\mu}^2 - \abs{\mu} }{2}  -C'\abs{\mu} = f(\T)-C'\abs{\T} - ct(\mu) = f(\T)-C'\abs{\T} +\InnaE{n(\mu)} -  \InnaE{n(\mu\check{})} \leq \\
& \leq f(\T)-C'\abs{\T} +\frac{\InnaE{\abs{\T}^2 }}{2} +\frac{\abs{\T}}{2} - C^{-1}\abs{\T}+ \ell(\mu)\abs{\T} + C^{-1}\abs{\mu} - \InnaE{n(\mu\check{})}
\end{align*}
 
\InnaE{The last inequality means that for any Young diagram $\mu$ satisfying Condition \ref{cond:irr_c_fin_length}, the expression
\begin{equation}\label{eq:irrat_c_eq3}
  \frac{\abs{\mu}^2 }{2} -\left(\frac{1 }{2}+C'+C^{-1} \right)\abs{\mu} -l(\mu)\abs{\T} + \InnaE{n(\mu\check{})} 
\end{equation}
 is bounded by a function of $\T$ (which is fixed).}
 
\InnaE{We} now consider the expression $\InnaE{n(\mu\check{})}$. We can rewrite it as $$\InnaE{n(\mu\check{})} = \sum_{1 \leq k \leq \ell(\mu)} (\mu\check{}_k -1) (k-1) +\frac{l(\mu)^2}{2} -\frac{l(\mu)}{2} +\abs{\mu}$$

The summand $\sum_{1 \leq k \leq \ell(\mu)} (\mu\check{}_k -1) (k-1)$ is clearly non-negative, so $$\InnaE{n(\mu\check{})}  \geq \frac{l(\mu)^2}{2} -\frac{l(\mu)}{2} +\abs{\mu}$$


 Applying the inequality $\InnaE{n(\mu\check{})}  \geq \frac{l(\mu)^2}{2} -\frac{l(\mu)}{2} +\abs{\mu}$, to Inequality \ref{eq:irrat_c_eq3}, we  see that \InnaE{for any Young diagram $\mu$ satisfying Condition \ref{cond:irr_c_fin_length}, the expression 
\begin{align*}
 \frac{\abs{\mu}^2 }{2} +\left(\frac{1 }{2}-C'-C^{-1} \right)\abs{\mu}+ \frac{l(\mu)^2}{2} -\left(\frac{1}{2} +\abs{\T}\right)l(\mu) 
\end{align*}
is bounded by a function of $\T$. This clearly means that in such a case, $\abs{\mu}$ is bounded (by a function of $\T$); we conclude that $M_{c, \nu}(\T)$ has finitely many composition factors.}

\end{proof}

\end{document}